\documentclass[aap]{imsart}

\RequirePackage{amsthm,amsmath,amsfonts,amssymb}
\RequirePackage[numbers]{natbib}
\RequirePackage[colorlinks,citecolor=blue,urlcolor=blue]{hyperref}
\RequirePackage{graphicx}
\usepackage{cleveref}

\startlocaldefs
\numberwithin{equation}{section}
\theoremstyle{plain}
 \newtheorem{theorem}{Theorem}[section]
 \newtheorem{lemma}[theorem]{Lemma}
 \newtheorem{prop}[theorem]{Proposition}
 \newtheorem{cor}[theorem]{Corollary}

 \newtheorem{claim}[theorem]{Claim}
\theoremstyle{remark}
 \newtheorem{assumption}[theorem]{Assumption}
 \newtheorem{remark}[theorem]{Remark}
 \newtheorem{defn}[theorem]{Definition}

\newcommand\nc\newcommand
\nc\dmo\DeclareMathOperator
\nc\bs\boldsymbol
\newcommand{\abbr}[1]{{\sc\lowercase{#1}}}
\DeclareFontFamily{U}{mathx}{\hyphenchar\font45}
\DeclareFontShape{U}{mathx}{m}{n}{
      <5> <6> <7> <8> <9> <10>
      <10.95> <12> <14.4> <17.28> <20.74> <24.88>
      mathx10
      }{}
\DeclareSymbolFont{mathx}{U}{mathx}{m}{n}
\DeclareFontSubstitution{U}{mathx}{m}{n}
\DeclareMathAccent{\widecheck}{0}{mathx}{"71}

\nc{\cf}{cf.\ }
\nc{\ie}{i.e.\ }
\nc{\eg}{e.g.\ }

\nc{\cA}{\mathcal{A}}
\nc{\cB}{\mathcal{B}}
\nc{\cC}{\mathcal{C}}
\nc{\cD}{\mathcal{D}}
\nc{\cE}{\mathcal{E}}
\nc{\cF}{\mathcal{F}}
\nc{\cG}{\mathcal{G}}
\nc{\cH}{\mathcal{H}}
\nc{\cI}{\mathcal{I}}
\nc{\cJ}{\mathcal{J}}
\nc{\cK}{\mathcal{K}}
\nc{\cL}{\mathcal{L}}
\nc{\cM}{\mathcal{M}}
\nc{\cN}{\mathcal{N}}
\nc{\cO}{\mathcal{O}}
\nc{\cP}{\mathcal{P}}
\nc{\cQ}{\mathcal{Q}}
\nc{\cR}{\mathcal{R}}
\nc{\cS}{\mathcal{S}}
\nc{\cT}{\mathcal{T}}
\nc{\cU}{\mathcal{U}}
\nc{\cV}{\mathcal{V}}
\nc{\cW}{\mathcal{W}}
\nc{\cX}{\mathcal{X}}
\nc{\cY}{\mathcal{Y}}
\nc{\cZ}{\mathcal{Z}}

\nc{\E}{\mathbb{E}}
\nc{\Z}{\mathbb{Z}}
\nc{\N}{\mathbb{N}}
\nc{\R}{\mathbb{R}}
\nc{\B}{\mathbb{B}}

\def \HS {\mathrm{HS}}

\nc{\eps}{\varepsilon}
\renewcommand{\t}{\tilde}
\nc{\ul}{\underline}
\nc{\wt}{\widetilde}
\nc{\wh}{\widehat}
\nc{\wch}{\widecheck}
\nc{\ER}{Erd\H{o}s--R\'enyi }
\renewcommand{\P}{\mathbb{P}}
\nc{\eqd}{\stackrel{\text{\tiny $d$}}{=}}
\dmo{\ind}{\mathbb{I}}
\dmo{\1}{\mathbf{1}}
\nc{\ls}{\lesssim}
\nc{\gs}{\gtrsim}
\def \lf {\lfloor}
\def \rf {\rfloor}
\renewcommand{\d}{\,d}
\nc\notni{\not\owns}

\nc{\edges}{\mathsf{e}}
\nc{\verts}{\mathsf{v}}
\nc{\Edges}{\mathsf{E}}
\nc{\Verts}{\mathsf{V}}
\dmo{\DKL}{D}
\nc\lam\lambda
\nc{\bG}{{\bs G}}
\nc{\bT}{{\mathsf T}}
\nc{\Gnp}{{\bG_{n,p}}}
\nc{\Ham}{{\mathrm{H}}}
\nc{\hamlet}{{h}}
\nc{\ube}{\hamlet}
\nc{\uF}{{\underline{F}}}
\nc{\uFs}{{\underline{F}^+}}
\nc{\up}{{s}}
\nc{\uup}{{\ul{s}}}
\nc{\uups}{{\ul{s}^+}}
\nc\ugamma{{\ul{\gamma}}}
\nc{\rate}{{r}}
\nc\field{{\alpha}}
\nc{\CH}{{Q}}
\nc{\graphs}{{\cG}}
\nc{\tf}{{\widetilde{f}}}
\nc{\hf}{{\widehat{f}}}
\nc{\chf}{{\widecheck{f}}}
\nc{\ipoly}{{P}}
\nc{\sfC}{{C}}
\dmo{\I}{{I}}
\dmo{\II}{{II}}
\nc{\Base}{{\mathsf{B}}}		
\dmo{\ME}{{ME}}
\nc{\tmax}{{T}}
\nc{\csize}{{a}}
\nc{\bsize}{{b}}
\nc{\bsizes}{{b_\star}}
\nc{\csizes}{{a_\star}}
\nc{\bsizesp}{{b'_\star}}
\nc{\csizesp}{{a'_\star}}
\nc{\betas}{{\beta_\star}}
\nc{\qs}{{q_\star}}
\nc{\Ks}{{K^\star}}
\nc{\Region}{{R}}
\nc{\hx}{{\hat{x}}}
\nc{\hy}{{\hat{y}}}
\nc{\ha}{{\hat{a}}}
\nc{\hb}{{\hat{b}}}
\nc{\eye}{\mathrm{I}}
\nc{\jay}{\mathrm{J}}
\nc{\UTnmf}{\Phi}
\nc{\probFEnmf}{\Psi}
\nc{\UTab}{\phi}
\nc{\FEab}{\psi}
\nc{\UTg}{\upsilon}
\nc{\bvec}{{\ul{\smash{\beta}}}}
\nc{\logmgf}{{\Lambda_{n,p}^\Ham}}
\nc{\nmf}{{\Psi_{n,p}^\Ham}}
\nc{\ratenp}{{\rate_{n,p}}}
\dmo{\Opt}{{Opt}}
\nc{\amir}[1]{{\color{blue} #1}}
\nc{\nicka}[1]{{\color{blue} #1}}
\endlocaldefs

\begin{document}

\begin{frontmatter}
\title{Typical structure of sparse\\ exponential random graph models}
\runtitle{Sparse exponential random graph models}

\begin{aug}
\author[A]{\fnms{Nicholas A.}~\snm{Cook}\ead[label=e1]{nickcook@math.duke.edu}}
\and
\author[B]{\fnms{Amir}~\snm{Dembo}\ead[label=e2]{adembo@stanford.edu}}
\address[A]{Department of Mathematics, Duke University\printead[presep={,\ }]{e1}}

\address[B]{Department of Mathematics, Stanford University\printead[presep={,\ }]{e2}}
\end{aug}

\begin{abstract}
We consider general exponential random graph models (\abbr{ERGM}s) where the sufficient statistics are functions of homomorphism counts for a fixed collection of simple graphs $F_k$. 
Whereas previous work has shown a degeneracy phenomenon in dense \abbr{ERGM}s, we show this can be cured by raising the sufficient statistics to a fractional power. 
We rigorously establish the na\"ive mean-field approximation for the partition function of the corresponding
Gibbs measures, and in case of  ``ferromagnetic'' models with vanishing edge density
show that typical samples 
resemble a typical Erd\H{o}s--R\'enyi graph 
with a planted 
clique and/or a planted complete bipartite graph of appropriate sizes. We 
establish such behavior also for the conditional structure 
of the Erd\H{o}s--R\'enyi graph in the large deviations regime for excess $F_k$-homomorphism counts. 
These structural results are obtained by combining quantitative large deviation principles, established in previous works, with a novel stability form of a result of \cite{BGLZ} on the asymptotic solution for the associated entropic variational problem. A technical ingredient of independent interest is a stability form of Finner's generalized H\"older inequality.
\end{abstract}



\end{frontmatter}
\tableofcontents


\section{Introduction}

With $[n]:=\{1,\dots, n\}$, let $\cG_n$ be the set of symmetric $\{0,1\}$-valued functions on $[n]^2$ with zero diagonal, i.e.
\begin{equation}	\label{def:Gn}
G:[n]^2\to \{0,1\}\,,\quad (i,j)\mapsto G_{i,j}\,,\qquad \text{ with }\quad G_{i,j}=G_{j,i} \;,\quad G_{i,i}=0 \quad \forall i,j\in [n].
\end{equation}
Elements of $\cG_n$ are naturally identified with simple graphs (undirected and loop-free with at most one edge between any pair of vertices) over the labeled vertex set $[n]$, with $G_{i,j}=1$ when $i,j$ are joined by an edge. We are concerned with the case that $n$ is large or tending to infinity.

An exponential random graph model (\abbr{ergm}) is a probability measure on $\cG_n$ with density of the form
\begin{equation}	\label{ergm.intro}
\frac1{Z_n(\alpha,\bvec)} \exp\big(\, n^2\Ham(G;\bvec) - \alpha \edges(G)\,\big)\,,\qquad G\in \cG_n
\end{equation}
for parameters $\alpha\in \R, \bvec\in \R^m$, where $Z_n(\alpha,\bvec)$ is the normalizing constant, or \emph{partition function}, $\edges(G):= \sum_{i<j} G_{i,j}$ is the total number of edges in $G$, and the \emph{Hamiltonian} $\Ham(G;\bvec)$ is  of the form
\begin{equation}	\label{Ham.intro}
\Ham(G;\bvec) = \sum_{k=1}^m \beta_k f_k(G)
\end{equation}
for a fixed collection of functions $f_k:\cG_n\to \R$. The $f_k$ are typically taken as graph functionals that can be estimated through sampling, such as the frequency of a particular subgraph.
The parameter $\alpha$ controls the sparsity of typical samples from the \abbr{ergm}, with large positive/negative values of $\alpha$ leading to sparse/dense graphs, respectively; in physical terms $\alpha$ plays the role of the strength of an external field.
(We could have included $\alpha \edges(G)$ as one of the terms in the Hamiltonian \eqref{Ham.intro}, but it will be convenient to keep this term separate.)

E\abbr{rgm}s generalize the \ER (or binomial) random graph. Indeed, with 
\begin{equation}	\label{alpha-p}
\alpha= \log\frac{1-p}p
\end{equation}
for $p\in (0,1)$ and the Hamiltonian set to zero, the partition function is $Z_n= (1-p)^{-{n\choose2}}$, and the density \eqref{ergm.intro} is then $
p^{\edges(G)} (1-p)^{{n\choose 2} - \edges(G)}
$.
In general, under the parametrization \eqref{alpha-p} one can view the \abbr{ergm} \eqref{ergm.intro} as the tilt of the Erd\H os--R\'enyi($p$) distribution by the function $\exp(n^2 \Ham(\,\cdot\;;\bvec))$. 
This is the perspective we will take later in the article -- see \eqref{def:ergm} -- as it will be convenient for importing results on the large deviation theory of \ER graphs.
We further note that \eqref{def:ergm} covers a more general class of distributions than in \eqref{ergm.intro}--\eqref{Ham.intro} by accommodating sparse models, where 
$p=p(n) \to 0$ as $n \to \infty$.

Apart from the \ER model, perhaps the best-known non-trivial \abbr{ergm} is the \emph{edge-triangle model}, with Hamiltonian 
\begin{equation}	\label{edgeK3.intro}
\Ham(G;\beta) =\beta t(C_3,G)=\beta\frac1{n^3} \sum_{i_1,i_2,i_3=1}^n G_{i_1,i_2}G_{i_2,i_3} G_{i_3,i_1}\,.
\end{equation}
With the normalization by $n^3$, $t(C_3,G)$ is the probability that three independently and uniformly sampled vertices of $G$ form a triangle. 
(We use the standard notation $C_\ell$ for the cycle on $\ell$ vertices.)
Thus, typical samples should have more (resp.\ fewer) triangles than an \ER graph when $\beta$ is taken to be positive (resp.\ negative).

\subsection{Previous works}
\label{sec:background}

E\abbr{RGM}s were introduced and developed in the statistics and social sciences literature in the 80s and 90s \cite{HoLe81, FrSt86, WaPa96}; see \cite{Fienberg1,Fienberg2} for a survey of the subsequent vast literature.
The motivation was to develop a parametric class of distributions on graphs that could be fit to social networks via parameter estimation. 
A key feature of social networks is \emph{transitivity} -- that friends of friends are more likely to be friends -- a feature that is not present in typical \ER graphs. In particular, it was hoped that transitivity would arise by re-weighting the \ER distribution to promote triangles, as in the edge-triangle model \eqref{edgeK3.intro} with $\beta>0$.

The general form of the  \abbr{ergm}s \eqref{ergm.intro}--\eqref{Ham.intro} is appealing as they are exponential families, and the separable form of the Hamiltonian means that the functions $f_k(G)$ and edge density $n^{-2}\edges(G)$ are sufficient statistics for the model parameters $\beta_k$ and $\alpha$, respectively. 
For Bayesian inference and maximum likelihood estimation of the parameters it is important to have an accurate approximation for the partition function $Z_n(\alpha, \bvec)$, which has often been obtained via Markov chain Monte Carlo sampling schemes.
Sampling is also used to understand the typical structure of \abbr{ergm}s.

Problems with various aspects of this program were noted empirically from early on, ranging from the inability of \abbr{ergm}s to fit realistic social networks, to the inability of sampling algorithms to converge in a reasonable time \cite{Strauss86,Snijders02,Handcock}. With regards to the former, it was observed that in large regimes of the parameter space, typical samples exhibit no transitivity.
Moreover, in some regimes \abbr{ergm}s seem to concentrate in neighborhoods of a small number of graphs with trivial structure -- such as the empty graph and the complete graph -- a phenomenon known as \emph{degeneracy}.
We refer to \cite{SPRH06} for discussion of these and other problems, as well as some proposals to circumvent them.

More recent mathematically rigorous works have helped to clarify these issues.
An important work of Bhamidi, Bressler and Sly \cite{BBSmixing} considered the case that the functions $f_k$ in \eqref{Ham.intro} are densities $t(F_k,G)$ of a fixed collection of graphs $F_k, k=1,\dots, m$, that is
\begin{equation}	\label{dense-model}
\Ham(G;\bvec) = \sum_{k=1}^m \beta_k t(F_k,G)
\end{equation}
generalizing the edge-triangle model \eqref{edgeK3.intro} (see \eqref{def:t} below for the general definition of $t(F,G)$).
For the case that the model is in the ``ferromagnetic'' parameter regime with all $\beta_k>0$,
they are able to characterize a ``low-temperature'' parameter regime where local MCMC sampling schemes take exponential time to converge; in the complementary high-temperature regime where sampling algorithms have polynomial convergence time, typical samples exhibit the structure of an \ER graph, in particular lacking the transitivity property. 

Another major development on models of the form \eqref{dense-model} was made in work of Chatterjee and Diaconis \cite{ChDi}, where they applied the large deviation theory of \cite{ChVa11} for the \ER graph to rigorously establish a variational approach to estimating the partition function known as the n\"aive mean-field (\abbr{nmf}) approximation from statistical physics. They also show that in the ferromagnetic regime, \abbr{ergm}s are close to a mixture of \ER graphs -- that is, with the parameter $p$ sampled from some distribution. 
For the case of the edge-triangle model \eqref{edgeK3.intro}, they rigorously establish a degeneracy phenomenon wherein for large positive values of $\alpha$, as $\beta$ increases through a critical threshold $\betas(\alpha)\ge0$ the expected edge density  jumps from nearly zero to nearly one; see \cite[Theorem 5.1]{ChDi}. 

The works \cite{BBSmixing,ChDi} focus on \abbr{ergm}s with Hamiltonians having the special form \eqref{dense-model}, and with parameters $\alpha,\beta_k$ fixed independent of $n$, which means that samples are typically \emph{dense}, i.e. with constant edge density.
One further work on dense \abbr{ergm}s of particular relevance to the present work is that of Lubetzky and Zhao \cite{LuZh12}, who had the insight to consider a modified edge-$F$ model, with Hamiltonian of the form
\begin{equation}	\label{Ham.LuZh}
\Ham(G;\beta) = \beta t(F,G)^\gamma
\end{equation}
for a fractional power $\gamma\in(0,1)$. 
In particular, they show that for $\Delta$-regular $F$ and $\alpha$ above a certain threshold depending on $\Delta$, taking $\gamma\in (0,\frac\Delta{\edges(F)})$ cures the degeneracy phenomenon established in \cite{ChDi}, in the sense that there exists an open interval of values of $\beta$ for which a typical sample from the \abbr{ergm} does not look like an \ER graph with high probability. However, they left open the problem of determining what a typical sample \emph{does} look like. 
A similar strategy to cure degeneracy problems was also proposed in the physics literature \cite{HCT15}, where it was shown to yield better fits to real social network data.
The structure of such models in the limit as $(\alpha,\beta)$ tend to infinity along rays was studied in \cite{Demuse}, extending the work of \cite{YRF} for the case $\gamma=1$.

Starting with the work \cite{ChDe14} of Chatterjee and the second author, the recent development of quantitative approaches to large deviations for nonlinear functions on product spaces -- such as subgraph counts in \ER graphs -- has opened up the analysis of \abbr{ergm}s in the sparse regime where $\alpha$ and $\bvec$ depend on $n$. We refer to \cite{Chatterjee:book} for an introduction to this recent and rapidly developing area. In particular, the works \cite{ChDe14,Eldan} established the \abbr{nmf} approximation for the partition function under some growth conditions on the parameters.
In \cite{ElGr18ergms}, building on the nonlinear large deviations theory from \cite{Eldan, ElGr-decomp}, Eldan and Gross showed that under some growth conditions on $\alpha,\bvec$, models of the form \eqref{dense-model} are close to low-complexity mixtures of stochastic block models, with barycenters close to critical points of the \abbr{nmf} free energy. 
We also mention that the correlation between fixed edges in sparse \abbr{ergm}s with negative $\beta$ were studied in \cite{YiZh17}.
Mixing properties of the Glauber dynamics were used to establish concentration inequalities and \abbr{CLT}s in \cite{GaNa}.

In the language of statistical physics, \abbr{ergm}s are \emph{grand canonical ensembles}, and we mention there has been a long line of works on the structure of the corresponding \emph{microcanonical ensembles}, which are
graphs drawn uniformly under hard constraints on subgraph counts; see \cite{RaSa13,RRS14,KRRS17-multipodal,KeYi,NRS20}.

\subsection{Generalized ERGMs}

In the present work we apply results from \cite{CoDe,CDP} on a quantitative large deviations theory for the \ER graph to extend the \abbr{nmf} approximation to sparser \abbr{ergm}s than in previous works, and to establish the typical structure of samples.
Our setup also allows for a more general Hamiltonian than the separable form of \eqref{dense-model}.

For sparse models we need to introduce some scaling in the model \eqref{ergm.intro}--\eqref{Ham.intro}. 
Generalizing \eqref{def:Gn}, 
for a set $S$, 
an $S$-\emph{weighted graph} over $[n]$ is a symmetric $S$-valued function on $[n]^2$ with zero diagonal. 
Thus, elements of $\cG_n$ are $\{0,1\}$-weighted graphs. 
For an $\R$-weighted graph $X$ over $[n]$ and a fixed graph $F=(\Verts(F),\Edges(F))$, we define the \emph{homomorphism density of $F$ in $X$} by
\begin{equation}
\label{def:t}
t(F,X):= \frac{1}{n^{\verts(F)}}\sum_{\phi:\Verts(F)\to[n]} \prod_{\{u,v\}\in \Edges(F)} X_{\phi(u),\phi(v)}.
\end{equation}
For the case that $X$ is a $\{0,1\}$-weighted graph, $t(F,X)$ is the probability that a uniform random mapping of the vertices of $F$ into $[n]$ is a graph homomorphism from $F$ to the graph associated with $X$ -- that is, maps edges onto edges. Letting $p\in (0,1)$ (possibly depending on $n$),
and $\bG=\bG_{n,p}\in \cG_n$ denoting an Erd\H{o}s--R\'enyi($p$) graph, note  
that for $p=o(1)$, the typical value of $t(F,\bG)$ decays at rate $p^{\edges(F)}$ 
which depends on $F$. Hence, to combine 
on equal footing different graphs $F_k$ in the Hamiltonian, we instead use
$t(F_k,\bG/p)$ which are of $O(1)$. Specifically, fixing $m \ge 1$, 
graphs $\uF=(F_1,\dots, F_m)$ and a function $\hamlet:\R_{\ge0}^m\to\R$,
we define a Hamiltonian function on the space of $\R$-weighted graphs by
\begin{equation}	\label{hamilton}
\Ham(X) :=   \hamlet\big(\,t(F_1,X),\ldots,t(F_m,X)\,\big) \,,
\end{equation}
and use in the \abbr{ergm} the Hamiltonian $\Ham(\bG/p)$. 
The rate of decay for $O(1)$ upper deviations of $t(F,\bG/p)$ is known to be 
\begin{equation}	\label{def:rate}
r = \rate_{n,p} := n^2 p^\Delta \log(1/p)\,,
\end{equation}
where $\Delta \ge 2$ denotes the maximal degree of $F$ (see \cite{Chatterjee:book}). Thus, for $p=o(1)$ 
one can assume \abbr{wlog} that the graphs $F_k$ in \eqref{hamilton} all have the 
same maximal degree $\Delta \ge 2$ (as the effect of any $F_k$ of maximal degree 
less than $\Delta$ be negligible on scale $r_{n,p}$), and to simplify our presentation 
we make this assumption hereafter even for $p=O(1)$. Moreover, to avoid degeneracy in \abbr{ergm}s 
based on such $\Ham(\bG/p)$, one has to replace the factor $n^2$ of \eqref{ergm.intro}, 
by $r_{n,p}$ of \eqref{def:rate}, thereby making the effect of the Hamiltonian comparable 
to that of the large deviations induced by $\bG_{n,p}$. Specifically, such an \abbr{ergm} 
with $o(r_{n,p})$ scaling is merely a small perturbation of the corresponding \ER graph model, 
whereas using factors $r \gg r_{n,p}$ will result with an \abbr{ergm} those sample be 
close at large $n$ to the complete graph, regardless of $p$. We shall thus consider
the measure $\nu_{n,p}^\Ham$ on $\cG_n$ whose density with respect to $\bG_{n,p}$ is
\begin{equation}	\label{def:ergm}
\nu_{n,p}^\Ham(G) =  \exp(\ratenp\Ham(G/p) -\Lambda_{n,p}^\Ham)\,,\quad G\in \cG_n
\end{equation}
where we denote the log-moment generating function
\begin{equation}\label{def:mgf}
\Lambda_{n,p}^\Ham := \log\E \exp( \ratenp \Ham(\bG_{n,p}/p)).
\end{equation}
A sample from $\nu_{n,p}^\Ham$ is denoted by $\bG_{n,p}^\Ham$ (thus $\bG_{n,p}\eqd\bG_{n,p}^0$). 

We require $\hamlet$ of \eqref{hamilton} to correspond to a ferromagnetic model, i.e. 
be non-decreasing. Also, even at the proper scale of \eqref{def:rate} we must restrict the growth 
of $\hamlet$ at infinity in order to avoid degeneracy (similarly to the proposal \eqref{Ham.LuZh} 
of \cite{LuZh12} in the context of the dense edge-$F$ model). Indeed, such restricted growth 
guarantees 
the existence of finite optimizers for 
$\psi_{\uF,\hamlet}$ of \eqref{def:psi}, on which our main results are based.
Specifically, we make the following assumptions on $\hamlet$.
\begin{assumption}
\label{assume:h}
The function $\hamlet:\R_{\ge0}^m\to\R$ in \eqref{hamilton} is continuous, coordinate-wise 
non-decreasing, and satisfies the growth condition
\begin{equation}	\label{assu-h.growth}
\hamlet(\ul{x})=o_{\|\ul{x}\| \to\infty}\Big(\sum_{k=1}^m x_k^{\Delta/\edges(F_k)}\Big) \,.
\end{equation}
\end{assumption}

To connect our notation with the normalizing constant $Z_n$ for \abbr{ergm}s discussed in Section \ref{sec:background}, with  $\alpha=\log\frac{1-p}p$ we can alternatively express the density \eqref{def:ergm} in the form 
\begin{equation}	\label{ergmZ}
\frac1{Z}\exp\big(\ratenp \Ham(G/p) - \alpha \edges(G)\big)
\end{equation}
(as in \eqref{ergm.intro} but with $\rate_{n,p}$ in place of $n^2$ and $G$ scaled by $1/p$).
Then 
the normalizing factor $\exp(\Lambda_{n,p}^\Ham)$ in \eqref{def:ergm} is related to the \emph{free energy} $\log Z$ by
\begin{equation}	\label{MGF-Z}
\Lambda_{n,p}^\Ham = \log Z + {n\choose2} \log (1-p)\,.
\end{equation}
We note that (under mild conditions on $\hamlet$) $\logmgf$ is of order $\ratenp$, which is of lower order than the second term on the right hand side above when $p\ll1$. Hence, in the sparse case, our results on $\logmgf$ below provide asymptotics for the nontrivial sub-leading order of the free energy.

As an example, under our scaling the edge-triangle model (with some choice of $\hamlet$) has density proportional to
\begin{align}
&\exp\Big( r_{n,p} \cdot h\big(  t(C_3,G/p) \big) - \log\Big( \frac{1-p}p\Big) \edges(G) \Big)	\label{edgeK3.sparse}\\
&\qquad =\exp\bigg( n^2 p^2\log(1/p) \cdot h\Big( \frac1{n^3p^3} \sum_{i_1,i_2,i_3=1}^n G_{i_1,i_2} G_{i_2,i_3}G_{i_3,i_1}\Big) -  \log\Big(\frac{1-p}p\Big) \sum_{i<j} G_{i,j}  \bigg)\,.	\notag
\end{align}
In Corollary \ref{cor:edge-triangle} we determine the typical structure of samples from this model when $\hamlet=\beta f$ for a parameter $\beta>0$ and a fixed function $f$.
Following the insight of \cite{LuZh12} (see \eqref{Ham.LuZh}) it will be crucial to impose the growth condition \eqref{assu-h.growth}, which translates to $x^{-2/3}f(x)\to0$ as $x\to\infty$.

\subsection{The NMF approximation}

Under the definition \eqref{def:ergm} we have that $\nu_{n,p}^0$ is the Erd\H{o}s--R\'enyi($p$) measure on $\cG_n$.
The Gibbs (or Donsker--Varadhan) variational principle states that
\begin{equation}	\label{gibbs}
\logmgf=\sup_\mu\Big\{ \rate_{n,p} \E_{G\sim \mu} \Ham(G/p) - \DKL(\mu\|\nu_{n,p}^0)\Big\}
\end{equation}
where the supremum is taken over all probability measures on $\cG_n$, and $\DKL(\cdot\|\cdot)$ is the relative entropy. 
While this yields a formula for the normalizing constant of \abbr{ergm}s via \eqref{MGF-Z}, it is not a computationally feasible alternative to sampling since $\mu$ ranges over a set of dimension exponential in $n^2$. 

The \abbr{nmf} approximation posits that the supremum in \eqref{gibbs} is approximately attained on the subset of product probability measures on $\cG_n$, which are parametrized by the ${n\choose2}$-dimensional cube of $[0,1]$-weighted graphs
\begin{equation*}
\cQ_n:= \big\{ Q:[n]^2\to [0,1]: Q_{i,j}=Q_{j,i}\,,\; Q_{i,i}=0\; \forall i,j\in [n]\big\}.
\end{equation*}
Indeed, product probability measures on $\cG_n$ are of the form $\mu_Q$, the distribution of an inhomogeneous \ER graph that independently includes edges $\{i,j\}$ with probability $Q_{i,j}$. 

In our setting the \abbr{nmf} approximation takes the form\footnote{In \eqref{NMF} we have also 
replaced $\E_{G\sim \mu_Q} \Ham(G/p)$ by $\Ham(\E_{G\sim \mu_Q} G/p)$. The difference turns out to be negligible in our setting and we follow \cite{ChDe14,BGLZ} by taking $\nmf$ as the definition of the \abbr{nmf} approximation.}
\begin{equation}
\label{NMF}
\logmgf \approx \sup_{Q\in\cQ_n} \Big\{ \rate_{n,p} \Ham(Q/p) - \sum_{i<j} \eye_p(Q_{i,j})\Big\} =: \nmf
\end{equation}
where we use the common notation 
\begin{equation}	\label{def:eyepq}
\eye_p(q) :=\DKL(\text{Ber}(q)\|\text{Ber}(p)) =  q\log\frac{q}p + (1-q)\log\frac{1-q}{1-p}
\end{equation}
for the relative entropy of the Bernoulli($q$) law on $\{0,1\}$ with respect to the Bernoulli($p$) law. 

While \eqref{NMF} is not always true for Gibbs measures (as is notably the case for spin glass models), our first result shows it is a good approximation in our setting of generalized \abbr{ergm}s, under some conditions on $p$ and $\hamlet$.
See Section \ref{sec:notation} for our conventions on asymptotic notation. 

\begin{prop}
\label{prop:nmf}
Assume $n^{-1/(\Delta+1)}\ll p$ and that $\hamlet$ satisfies Assumption \ref{assume:h}.
Then 
\begin{equation*}
\logmgf=\nmf + o(\ratenp).
\end{equation*}
\end{prop}

\begin{remark}
Proposition \ref{prop:nmf} extends to $n^{-1/\Delta}\ll p$, when  
every vertex of degree $\Delta$ in $F_k$, $k \in [m]$, is in an isolated star.
Using results from \cite{CoDe} we could also allow $p$ as small as $(\log n)^Cn^{-1/\Delta}$ for a sufficiently large constant $C$ if every $F_k$ is a cycle (so $\Delta=2$ in this case). 
From \cite{CDP} 
we also have a matching \abbr{LDP} lower bound 
in Proposition \ref{prop:LD.Bstar}, and by relying on the latter result one can 
eliminate the restriction to non-decreasing $h(\cdot)$ in Proposition \ref{prop:nmf}.
 We skip the proofs of such refinements here.
\end{remark}

In the case that $p=o(1)$, our next result says that \eqref{NMF} can be further reduced to a \emph{two}-dimensional variational problem.
It involves a certain function $T_F:\R_{\ge0}^2\to \R_{\ge0}$ associated to a graph $F$ that was identified in the work \cite{BGLZ} on the upper tail problem for $t(F,\Gnp)$.
Recall the independence polynomial of a graph $F$ is defined as
\begin{equation}	\label{def:indep.poly}
\ipoly_F(x) = 1 + \sum_{\emptyset \ne U\in \cI(F)} x^{|U|}
\end{equation}
where the sum runs over the non-empty independent sets in $\Verts(F)$. 
Letting $F^\star$ denote the induced subgraph of $F$ on the vertices of maximal degree, we set
\begin{equation}	\label{def:TF}
T_F(a,b) := \ipoly_{F^\star}(b) + a^{\verts(F)/2} \ind(\text{$F$ is regular}).
\end{equation}
Define 
\begin{equation}
\label{def:psi}
\psi_{\uF,\hamlet} := \sup_{a,b\ge0} \Big\{ h( T_{F_1}(a,b), \dots, T_{F_m}(a,b)) - \frac12a-b\Big\}.
\end{equation}

\begin{theorem}
\label{thm:nmf2}
Assume $n^{-1/(\Delta+1)}\ll p\ll 1$ and that $\hamlet$ satisfies Assumption \ref{assume:h}. Then
\begin{equation}
\lim_{n\to\infty} \frac1\ratenp\logmgf = \psi_{\uF,\hamlet}\,.
\end{equation}
\end{theorem}

\subsection{Typical structure for sparse ERGMs and conditioned \ER graphs}

Proposition \ref{prop:nmf} and Theorem \ref{thm:nmf2} follow in a relatively straightforward way  from recent results on joint upper tails for homomorphism densities in \ER graphs \cite{CDP,BGLZ} via Varadhan lemma-type arguments.
Now we state our main results, which determine the typical structure of samples -- both from the generalized \abbr{ergm}s \eqref{def:ergm}, as well as the \ER graph conditioned on a joint upper-tail event for homomorphism densities. 
In both cases, samples concentrate around a two-parameter family of weighted graphs identified in \cite{BGLZ}.
This is the reason for the two-dimensional reduction of Theorem \ref{thm:nmf2} -- indeed, the functions $T_F(a,b)$ (appropriately rescaled) from \eqref{def:TF} give the behavior of the homomorphism density functionals $t(F, \cdot/p)$ on this two-parameter family. 

For $\xi>0$ and $I,J\subset[n]$, let $\cG_n^{I,J}(\xi)$ be the set of $G\in\cG_n$ satisfying
\begin{equation}	\label{almost-CH}
\sum_{i,j\in I} G_{i,j} \ge |I|^2 - 2 \xi n^2p^\Delta
\quad \text{ and } \quad 
\sum_{i\in J, j\in J^c} G_{i,j} \ge |J| (n-|J|) - \xi n^2 p^\Delta\,,
\end{equation}
and for $a,b\ge0$, set
\begin{equation}	\label{def:Gn'}
\cG_n^1(a,b,\xi) := \bigcup_{ \substack{I,J\subset[n] \text{ disjoint}\\ |I| = \lf (ap^\Delta)^{1/2} n\rf, |J| = \lf b p^\Delta n\rf}} \cG_n^{I,J}(\xi). 
\end{equation}
That is, for small $\xi$ the elements of $\cG_n^1(a,b,\xi)$ correspond to graphs containing an almost-clique over the vertex set $I$ and an almost-biclique across the vertex bipartition $(J, J^c)$, with $I,J$ of appropriate size. 
Note that for 
$n^{-\Delta}\ll p\ll1$ and $a,b>0$ fixed independent of $n$
we have $1\ll |J|\ll |I|\ll n$. 
(In particular it is only a matter of convenience to insist that $I,J$ be disjoint, as the number of edges in $I\times J$ is $o(p^\Delta n^2)$ and hence negligible.)

\begin{figure}
\includegraphics[width=3cm]{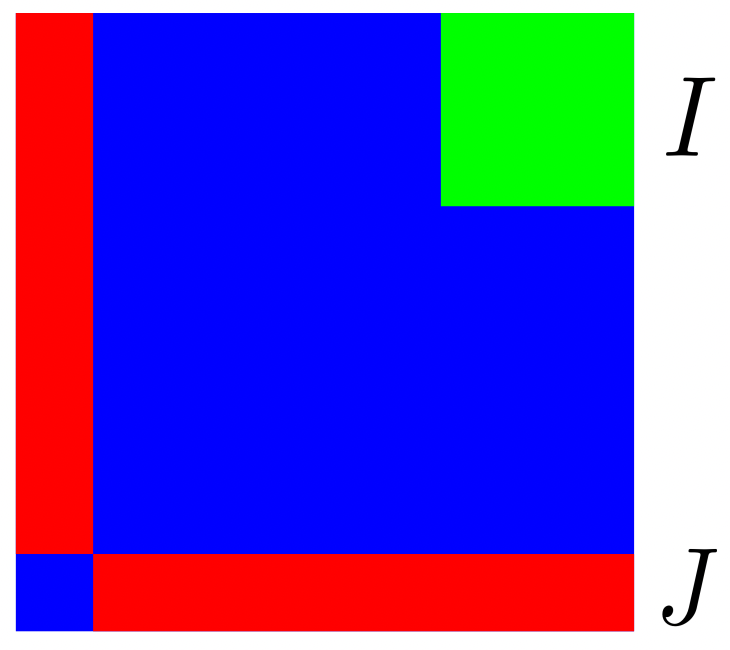}
\caption{
Depiction of a weighted clique-hub graph $Q^{I,J}:[n]\times[n]\to[0,1]$, defined in \eqref{def:QIJ}, taking value 1 in the clique region $I\times I$ (green) and hub region $(J\times J^c)\cup(J^c\times J)$ (red), and value $p$ elsewhere (blue). (We also have $Q^{I,J}_{i,i}\equiv 0$, but the diagonal is invisible when $n$ is large.) \Cref{thm:struct.ERGM} states that typical samples from sparse \abbr{ergm} models are close to such a weighted graph for some $I,J$ with $|I|\sim (ap^\Delta)^{1/2}n$ and $|J|\sim bp^\Delta n$, where ``closeness'' is quantified in two different ways for appropriate ranges of $p=o(1)$. (Hence $I$ and $J$ will cover a vanishing proportion of vertices, so  the figure exaggerates their relative sizes.)}
\label{fig:Qij}
\end{figure}

We also define another neighborhood of the graphs with an almost-clique at $I$ and an almost-biclique at $(J,J^c)$. 
Given $I,J\subset[n]$ disjoint, define the associated \emph{weighted clique-hub graph} $Q^{I,J}\in\cQ_n$, with 
\begin{equation}	\label{def:QIJ}
Q^{I,J}_{i,j} = p \ind_{i\ne j} + (1-p)\big[ \ind_{i,j\in I, i\ne j} + \ind_{(i,j)\in J\times J^c}
+ \ind_{(i,j)\in J^c\times J}
\big]\,.
\end{equation}
See \Cref{fig:Qij}.
(We informally refer to $J$ as a ``hub'' as it is possible to move between any two vertices $i,j$ using two edges of weight 1 by passing through a vertex of $J$.)
For $a,b\ge0$ let
\begin{equation}	\label{def:Qab}
\cQ_n(a,b) = \big\{ Q^{I,J} : |I| = \lf (ap^\Delta)^{1/2} n\rf, |J| = \lf b p^\Delta n\rf\big\}
\end{equation}
and for $\xi>0$ let
\begin{equation}	\label{def:Gn''}
\cG_n^2(a,b,\xi) := \bigcup_{Q\in \cQ_n(a,b)} \Big\{ G\in \cG_n: \|G-Q\|_{2\to2} < \xi np^{\Delta/2} \Big\}
\end{equation}
where with slight abuse we extend the spectral operator norm for matrices to $\R$-weighted graphs $X$, that is
\begin{equation}	\label{def:specnorm}
\|X\|_{2\to 2} = \sup_{0\ne u,v:[n]\to \R} \frac{ \sum_{i,j=1}^n X_{i,j} u(i)v(j)}{ \|u\|_{\ell^2([n])} \|v\|_{\ell^2([n])}}. 
\end{equation}

We note how the spectral norm enforces control on edge discrepancies: 
Consider any $Q^{I,J}\in\cQ_n(a,b)$ and $G\in\cG_n$ with $\|G-Q\|_{2\to2}<\xi n p^{\Delta/2}$. 
For $A,B\subseteq [n]$ let
\[
\edges_G(A,B):= \sum_{i\in A, j\in B} G_{i,j}
\]
be the number of edges in $G$ with one end in $A$ and the other in $B$ (counting edges contained in $A\cap B$ twice).
Then taking $u,v$ of the form $\pm\1_{A},\1_{B}$ in \eqref{def:specnorm}, 
 we have that for either $A,B\subseteq I$ or $A\subseteq J, B\subseteq J^c$,
\begin{equation}	\label{discrep1}
0 \le 1 - \frac{\edges_G(A,B)}{|A||B|} < \xi \Big( \frac{n^2 p^\Delta}{|A||B|}\Big)^{1/2}
\end{equation}
and for all $A,B\subset J^c$ such that at least one of $A,B$ lies in $I^c$, 
\begin{equation}	\label{discrep2}
\Big|\frac{\edges_G(A,B)}{p|A||B|} - 1\Big| < \xi \Big( \frac{n^2 p^{\Delta-2}}{|A||B|}\Big)^{1/2} .
\end{equation}
Comparing \eqref{almost-CH} with \eqref{discrep1} at $A=B=I$ and $A=J, B=J^c$, we see
that $\cG_n^2(a,b,\xi) \subseteq \cG_n^1(a,b,\sqrt{c} \xi)$ for $c=\max(a,b)$. In view of \eqref{discrep2} 
we conclude that for small $\xi$ the elements of $\cG_n^2(a,b,\xi)$ not only have an almost-clique and almost-hub of appropriate size, but also look uniformly like \ER graphs outside of $I,J$. 

In the following we use the parameter
\begin{equation*}
\Delta_\star=\Delta_\star(\uF) := \frac12\max_{1\le k\le m} \max_{\{u,v\}\in \Edges(F_k)}  \big\{ \deg_{F_k}(u) + \deg_{F_k}(v) \big\} \,,
\end{equation*}
where $\deg_{F_k}(u) =|\{u'\in\Verts(F): \{u,u'\}\in \Edges(F)\}|$ is the degree of $u$ in $F_k$. 
Note that $\Delta+1\le 2\Delta_\star\le 2\Delta$, with the lower bound holding when each $F_k$ is a $\Delta$-armed star, and the upper bound when each $F_k$ is $\Delta$-regular. 
We denote the set of optimizers in \eqref{def:psi} by
\begin{equation}	\label{def:opt.psi}
\Opt(\psi) := \Big\{ (a,b)\in \R_{\ge0}^2: h( T_{F_1}(a,b), \dots, T_{F_m}(a,b)) - \frac12a-b = \psi_{\uF,\hamlet} \Big\}\,.
\end{equation}
(Note this set depends on $\uF$ and $\hamlet$, but we suppress this from the notation.)

\begin{theorem}
\label{thm:struct.ERGM}
Assume that the graphs $F_1,\dots, F_m$ are connected and that $\hamlet$ satisfies Assumption \ref{assume:h}.
Then for any $\xi>0$ there exists $\eta_0>0$ depending only on $\uF,\hamlet$ and $\xi$ such that the following hold:
\begin{enumerate}
\item[(a)] If $n^{-1/(\Delta+1)}\ll p \ll 1$, then for all $n$ sufficiently large,
\begin{equation}\label{eq:Gn1}
\P\Big( \bG_{n,p}^\Ham \in \bigcup_{(a,b)\in \Opt(\psi)} \cG_n^1(a,b,\xi) \Big) \ge 1- \exp(-\eta_0\ratenp).
\end{equation}

\item[(b)] If further $n^{-1}\log n \ll p^{2\Delta_\star} \ll 1$, then for all $n$ sufficiently large,
\begin{equation}\label{eq:Gn2}
\P\Big( \bG_{n,p}^\Ham \in \bigcup_{(a,b)\in \Opt(\psi)} \cG_n^2(a,b,\xi) \Big) \ge 1- \exp(-\eta_0\ratenp).
\end{equation}

\end{enumerate}

\end{theorem}

\begin{remark}
With $n^{-1} \ratenp$ bounded away from zero in the range of $p$ values considered 
in Theorem \ref{thm:struct.ERGM}, it follows by Borel--Cantelli that for arbitrarily small $\xi$,
a.s. $\bG_{n,p}^\Ham$ must be for all sufficiently large $n$, in the sets given on the \abbr{lhs} of \eqref{eq:Gn1} or \eqref{eq:Gn2}, respectively. We also note in passing that 
since $p \ll 1$, for large $n$ one should be able to recover with high 
accuracy the predicted hub $J$ within the sample $\bG_{n,p}^\Ham$, by thresholding its 
degree sequence at $(1-\delta) n$. Further, in the range of $p$ values 
considered here, the predicted clique size $|I| 
\gg \sqrt{n p}/p$. Thus, a second thresholding of the remaining degrees at $n p + \sqrt{n}$ should reveal 
the clique $I$ with high accuracy for large $n$ (and thereby $p \ll 1$ small enough).
\end{remark}

\begin{remark} \Cref{thm:struct.ERGM} is a somewhat negative result 
from a modeling perspective, as it shows that for large $n$ and $n^{-1/(\Delta+1)} \ll p \ll 1$, 
a ferromagnetic \abbr{ergm} of the type $\bG_{n,p}^\Ham$ is 
essentially equivalent to the corresponding \ER graph with a uniformly chosen 
planted clique-hub of a suitable size (or, to a mixture of such, if $\Opt(\psi)$
is not a singleton). For $\Delta$-regular $F$, the results of \cite{HMS,BaBa} 
suggest this behavior of $\bG_{n,p}^\Ham$ even up to 
$p \, n^{2/\Delta} \gg (\log n)^{c(F)}$ (with $c(F)=\frac{2}{\Delta(\verts(F)-2)}$).
That is, only for $(\log n)^{c(F)} \gg p\, n^{2/\Delta} \gg 1$
can we expect a non-local distribution of the excess edges in such ferromagnetic \abbr{ergm}s.
\end{remark}

We deduce Theorem \ref{thm:struct.ERGM} from the next result concerning the structure of an \ER graph $\bG_{n,p}$ conditioned on a joint upper tail event for homomorphism densities. For given $\uup\in \R_{\ge0}^m$, denote the joint superlevel set
\begin{equation}\label{def:Up}
\cU_p(\uF,\uup) := \bigcap_{k\in[m]} \big\{ Q\in\cQ_n: t(F_k,Q/p)\ge 1+\up_k 
\,\big\} \,.
\end{equation}
Analogously to \eqref{NMF}, the \abbr{nmf} approximation for joint upper tail probabilities states that $-\log\P(\Gnp\in\cU_p(\uF,\uup))$ is approximately given by
\begin{equation}	\label{NMF.tails}
 \Phi_{n,p}(\uF,\uup):=\inf_{Q\in\cQ_n} \Big\{ \sum_{i<j} \eye_p(Q_{i,j}) : Q\in\cU_p(\uF,\uup) \Big\} \,.
\end{equation}
This and closely related asymptotics were established in several recent works under various hypotheses on $p$ and $\uF$ \cite{ChVa11,ChDe14,Eldan,CoDe,Augeri,HMS,BaBa,CDP}. 
Parallel works \cite{LuZh14,BGLZ,BhDe} have shown that the optimization problem $\Phi_{n,p}(\uF,\uup)$ over the ${n\choose2}$-dimensional cube $\cQ_n$ asymptotically reduces (after normalization by $\ratenp$) to the following optimization problem over the plane:
\begin{equation}\label{def:phi}
\UTab_\uF (\uup)
:= 
\inf_{(a,b)\in V_\uF(\uup)} \Big\{ \frac12a+b \Big\} \,,
\quad
V_\uF(\uup):=\bigcap_{k\in[m]}\Big\{ (a,b)\in\R_{\ge0}^2: \,
T_{F_k}(\csize,\bsize)\ge1+s_k 
\Big\}\,.
\end{equation}
An illustration of this optimization problem for the case $\underline{F}=(K_{1,2},C_3,C_4)$ is provided in Figure~\ref{fig:ab1}.

An easy computation shows that the probability that $\Gnp$ lies in an appropriate neighborhood of $\cQ_n(a,b)$ (such as $\cG_n^1(a,b,\xi)$ or $\cG_n^2(a,b,\xi)$) 
is roughly $\exp( - (\frac12a+b)\ratenp)$, and moreover that $t(F_k,Q/p)\sim T_{F_k}(a,b)$ for $Q\in\cQ_n(a,b)$. 
The asymptotic 
\begin{equation}	\label{UT.asymp}
\log\P(\Gnp\in\cU_p(\uF,\uup)) \sim- \phi_{\uF}(\uup) \ratenp
\end{equation}
established in previous works
thus suggests that the joint upper tail event roughly coincides with the event that $\Gnp$ lies in a neighborhood of $\cQ_n(a,b)$ for some $(a,b)$ attaining the infimum in \eqref{def:phi}. Note that such a conditional structure result does not follow from \eqref{UT.asymp} since the tail probability is only determined to leading order in the exponent.

The following establishes such conditional structure results for joint upper tail events under decay conditions on $p$.
We denote the set of optimizers in \eqref{def:phi} by $\Opt(\phi; \uup)$ (suppressing the dependence on $\uF$). 

\begin{theorem}
\label{thm:struct.ER}
Suppose $F_1,\ldots,F_m$ are connected. For any fixed $\xi>0$ and $\uup\in\R_{\ge0}^m$
there exists $\eta_1=\eta_1(\uF,\uup,\xi)>0$ such that the following hold: 
\begin{enumerate}
\item[(a)] If $n^{-1/(\Delta+1)}\ll p\ll1$, then for all $n$ sufficiently large,
\begin{equation}	\label{Gn'}
\P\Big(  \Gnp \in \bigcup_{(\csize,\bsize)\in \Opt(\phi;\uup)} \cG_n^1(\csize,\bsize,\xi) \,\bigg|\, \Gnp \in \cU_p(\uF,\uup)\,\Big) \ge 1 - \exp(-\eta_1\ratenp) \,.
\end{equation}
\item[(b)] If further $n^{-1}\log n \ll p^{2\Delta_\star}\ll 1$, then for all $n$ sufficiently large,
\begin{equation}	\label{Gn''}
\P\Big(  \Gnp \in \bigcup_{(\csize,\bsize)\in \Opt(\phi;\uup)} \cG_n^2(\csize,\bsize,\xi) \,\bigg|\, \Gnp\in \cU_p(\uF,\uup)\,\Big) \ge 1 - \exp(-\eta_1\ratenp)\, .
\end{equation}
\end{enumerate}
\end{theorem}

\begin{remark}
We note that 
\cite{HMS} establishes \eqref{Gn'} in an essentially optimal 
range of $p$ for the case 
$m=1$ and $F_1$ a clique. 
However, in view of \cite{BaBa}
one expects to find conditionally on $\cU_p(\uF,\uup)$ different structures
for bipartite $\Delta$-regular $F_1$ and $p = o(n^{-1/\Delta})$ (and the same thus 
applies for typical samples from the \abbr{ERGM}-s corresponding 
to such such $F_1$ and $p$).
\end{remark}

\begin{figure}
\includegraphics[width=9cm]{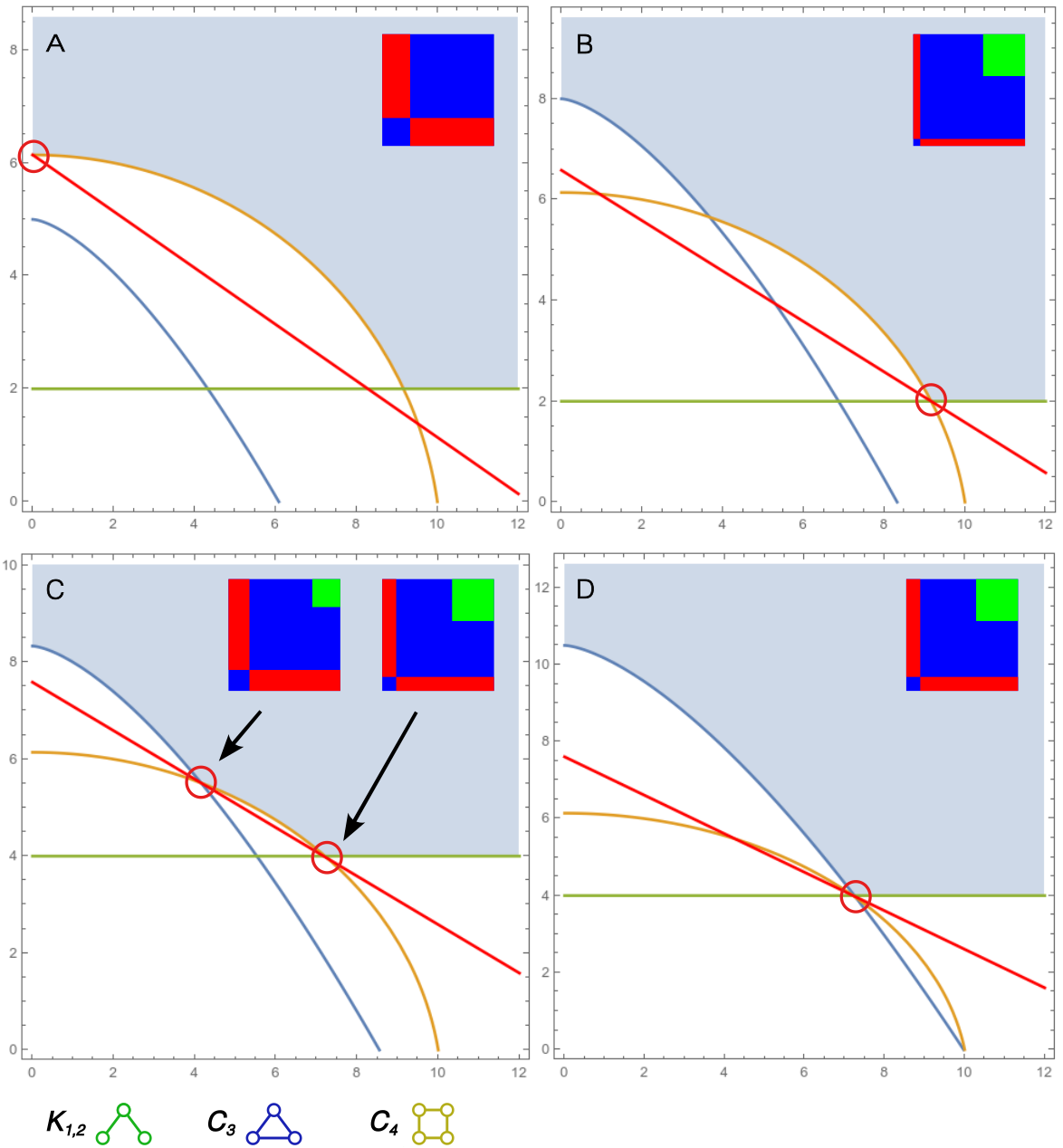}
\caption{
Plot of the feasible region $V_\uF(\uup)$ (light blue) of the $(a,b)$-plane for the optimization problem \eqref{def:phi}
for $(F_1,F_2,F_3)=(K_{1,2},C_3,C_4)$, with $s_3=100$ and four choices of $(s_1,s_2)$: $(2,15)$ (A); $(2,24)$ (B); $(4,25)$ (C); (4,31.5) (D).  The level curves $T_{F_k}(a,b)=1+s_k$ are plotted for $k=1,2,3$ in green, blue and yellow, respectively, and the line $\frac12a+b=\phi_\uF(\uup)$ is in red. Points $(\csizes,\bsizes)$ in the set $\Opt(\phi;\uup)$ of minimizers are circled in red.
The conditional structure of $\Gnp$ for each choice of $\uup$ is indicated in the top-right of each frame with a depiction of the corresponding weighted graph $Q^{I,J}$ (see \Cref{fig:Qij}). 
In C we see a phase transition in the sizes of the clique and hub at $\uup=(4,25,100)$.
C and D show that $\uup=(4,25,100)$ and $(4,31.5,100)$ are ``tricritical points'' where a qualitative change in the conditional graph structure can be achieved by perturbing any of $s_1,s_2,s_3$.
\label{fig:ab1}
}
\end{figure}

The following gives some information on the relation between the optimization problems $\psi_{\uF,\hamlet}, \phi_\uF$ in \eqref{def:psi},\eqref{def:phi}, and on their sets of optimizers $\Opt(\psi),\Opt(\phi;\uup)$.
The proof is given in Appendix \ref{sec:opt}.

\begin{prop}\label{prop:opt} ~\\
(a).\ We have that $\phi_\uF$ is continuous and non-decreasing in each argument,
with $\phi_\uF(\underline{0})=0$
and
\begin{equation}	\label{phi.LB}
\phi_\uF(\uup) \gs_{\uF} \sum_{k=1}^m s_k^{\Delta/\edges(F_k)}\qquad \forall\; \uup\in \R_{\ge0}^m : \|\uup\|_\infty\ge C(\uF)
\end{equation}
for a sufficiently large constant $C(\uF)>0$. \\
(b).
For any $\uup\in \R_{\ge0}^m$,
the set $\Opt(\phi;\uup)\subset\R_{\ge0}^2$ of
optimizers in \eqref{def:phi} is a non-empty, finite set of points on the closed line segment $\{ (\csize,\bsize): \frac12\csize+\bsize=\phi_{\uF}(\uup), \csize\ge0,\bsize\ge0\}$.\\
(c).
For $h$ satisfying Assumption \ref{assume:h}, we have the duality relation
\begin{equation}\label{psi-phi}
\psi_{\uF,\hamlet} = \sup_{\uup \in \R_{\ge 0}^m} 
\big\{ \hamlet (\1+\uup) - \phi_\uF (\uup) \big\} \,.
\end{equation}
The supremum is attained on a nonempty bounded set $S^\star\subset\R_{\ge0}^m$.\\
(d).
For $h$ satisfying Assumption \ref{assume:h}, we have
\begin{equation}	\label{opt-opt}
\Opt(\psi) = \bigcup_{\uup \in S^\star} \Opt(\phi;\uup)\,,
\end{equation}
and in particular that $\Opt(\psi)$ is a nonempty bounded subset of $\R_{\ge0}^2$.
\end{prop}\begin{remark}
For certain choices of $\hamlet$ we can have that $S^\star$, and hence $\Opt(\psi)$, are uncountable, though for generic choices they will be finite sets, such as when $\hamlet$ is of the form $\hamlet(\ul{x}) = \sum_k \beta_k (x_k-a_k)_+^{\gamma_k}$ for constants $\beta_k>0$, $a_k\in\R$ and $\gamma_k\in (0,\Delta/\edges(F_k))$. 
\end{remark}

\subsection{Edge-$F$ models}

In this subsection we specialize the measures \eqref{def:ergm} to the case that the Hamiltonian of \eqref{hamilton} takes the form
\begin{equation}
\Ham(X) = \Ham(X;\beta) = \beta f( t(F,X))
\end{equation}
for fixed connected graph $F$, 
$\beta>0$ 
and 
$f:\R_{\ge0}\to \R$ that is continuous, and strictly increasing and differentable on $(1,\infty)$,
with $f(x)=o_{x\to\infty}(x^{\Delta/\edges(F)})$.
We let $\bG_{n,p,\beta}$ denote a sample from the corresponding generalized \abbr{ergm} $\nu_{n,p}^\Ham$.

In this case it turns out that the optimizers in \eqref{def:psi} lie on the $a$ or $b$ axis, as was shown to be the case for the large deviation optimization problem \eqref{def:phi} in \cite{BGLZ}. 
Indeed, the main result in \cite{BGLZ} states that for $n^{-1/\Delta}\ll p\ll1$ we have
\begin{equation*}
\phi_F(s) = \begin{cases} P_{F^\star}^{-1}(1+s) & F \text{ is irregular} \\ \min\{ \frac12s^{2/\verts(F)} , P_F^{-1}(1+s)\} & F \text{ is $\Delta$-regular}.\end{cases}
\end{equation*}
In the latter case, the minimum is given by the first expression $ \frac12s^{2/\verts(F)}$ for 
$s\in [s_c(F),\infty)$ and  the second expression $P_F^{-1}(1+s)$ for $s\in [0,s_c(F)]$, where $s_c(F)$ is the unique positive solution of the equation 
\begin{equation}\label{def:sc-F}
\frac12s^{2/\verts(F)} = P_F^{-1}(1+s) \,.
\end{equation}

Let
\begin{equation}	\label{def:Ubs}
U(\beta,s) = \beta f(1+s) - \phi_F(s) 
\end{equation}
be the objective function in \eqref{psi-phi}. When $F$ is regular we further denote the restrictions of $U(\beta,\cdot)$ to $[0,s_c(F)]$ and $[s_c(F),\infty)$ by
\begin{align}\label{eq:Uhub}
U_{hub}(\beta,s) &= \beta f(1+s) - P_{F^\star}^{-1}(1+s)\,,\quad s\in [0,s_c(F)],\\
U_{clique}(\beta,s) &= \beta f(1+s) - \frac12 s^{2/\verts(F)}\,,\quad\quad\, s\in[s_c(F),\infty).
\label{eq:Uclique}
\end{align}

The following result about the typical structure of Edge-$F$ models is 
proved in Section \ref{sec:edgeF}.

\begin{prop}
\label{prop:edgeF}
In the above setting, the following hold:

\underline{Case of $F$ irregular.} Suppose that $f$ is such that for all $\beta\ge0$, $U(\beta,\cdot)$ attains its maximum on $\R_{\ge0}$ at a unique point $s^\star(\beta)$. 
Then for every $\beta,\xi>0$ there exists $c=c(\beta,\xi,F,f)>0$ such that when $n^{-1/(\Delta+1)}\ll p\ll1$, we have
\begin{equation*}
\P\big(\, \bG_{n,p,\beta}\in \cG_n^1(0, \bsizes(\beta), \xi)\,\big) \ge 1-\exp(-c\ratenp)
\end{equation*}
for all $n$ sufficiently large, where $\bsizes(\beta)= P_{F^\star}^{-1}(1+s^\star(\beta))$. 

\underline{Case of $F$ regular.} Suppose $f$ is such that for all $\beta\ge0$, 
$U_{hub}(\beta,\cdot)$ and $U_{clique}(\beta,\cdot)$ attain their maxima at unique points 
$s_{clique}^\star(\beta)$ and $s_{hub}^\star(\beta)$ in their respective domains. 
Then there exists $\beta_c(F,f)>0$ such that for any fixed $\beta,\xi>0$ there exists $c=c(\beta,\xi,F,f)>0$ such that
the following holds when $n^{-1/(\Delta+1)}\ll p\ll1$.
If $\beta<\beta_c(F,f)$, then for all $n$ sufficiently large,
\[
\P\big(\, \bG_{n,p,\beta}\in \cG_n^1(0, \bsizes(\beta), \xi)\,\big) \ge 1-\exp(-c\ratenp)
\]
while if $\beta>\beta_c(F,f)$, then for all $n$ sufficiently large,
\[
\P\big(\, \bG_{n,p,\beta}\in \cG_n^1(\csizes(\beta), 0, \xi)\,\big) \ge 1-\exp(-c\ratenp)\,,
\]
where $\csizes(\beta)= s_{clique}^\star(\beta)^{2/\verts(F)}$ and $\bsizes(\beta) = P_F^{-1}(1+s_{hub}^\star(\beta))$. 

Moreover, the same conclusions hold in all cases with $\cG_n^2$ in place of $\cG_n^1$ when $n^{-1}\log n\ll p^{2\Delta_\star(F)}\ll1$. 
\end{prop}


We note in passing 
that $s^\star(\beta)>0$ if and only if
\begin{equation}
\beta>\beta_o :=\inf_{s>0} \{\phi_F(s)/(f(1+s)-f(1))\} \,.
\end{equation}

\begin{remark}[Absence of degeneracy]
\label{rmk:degeneracy}
Recall the degeneracy phenomenon established in \cite{ChDi} for the version \eqref{edgeK3.intro} of the edge-triangle model,
wherein there exists $\betas=\betas(\alpha)\ge0$ such that typical samples transition from almost-empty to almost-full as $\beta$ increases through $\betas$.
The sparse setting of Proposition \ref{prop:edgeF} corresponds to the limiting case $\alpha\to-\infty$, and without \eqref{assu-h.growth} we have such a degeneracy transition at $\betas=0$.
Indeed, for $\Delta$-regular $F$, $\phi_F(s)$ grows like $s^{2/\verts(F)}=s^{\Delta/\edges(F)}$ when $s$ is large. Thus, taking $f$ to be linear as in \eqref{edgeK3.intro}, we have that the optimizing value of $s$ for \eqref{def:Ubs} is zero for $\beta=0$ and $+\infty$ as soon as $\beta>0$.
Thus, the growth condition \eqref{assu-h.growth} is crucial for eliminating the degeneracy phenomenon. 
\end{remark}

\begin{remark}
To find $\beta_c(F,f)$ one computes the unique positive solution $s_c(F)$ of \eqref{def:sc-F} and then
the maximizers $s_{hub}^\star(\beta) \le s_c(F) \le s_{clique}^\star(\beta)$ 
of the functions $U_{hub}(\beta,\cdot)$ and $U_{clique}(\beta,\cdot)$, of \eqref{eq:Uhub} and
\eqref{eq:Uclique}, respectively. Now, $\beta_c(F,f)$ is the largest $\beta$ such that  
$U_{hub}(\beta,s_{hub}^\star(\beta)) > U_{clique}(\beta,s_{clique}^\star(\beta))$.
\end{remark}

The following result gives a more explicit version of Proposition \ref{prop:edgeF} for the case of $F=C_3$ and a specific choice of $f$.
The computations are given in Section \ref{sec:edgeF}.

\begin{cor}[The edge-triangle model]
\label{cor:edge-triangle}
With hypotheses as in Proposition \ref{prop:edgeF}, take $F=C_3$ and 
\[
f(x)= (x-1)_+^{\gamma/3}\,,
\]
for some $\gamma \in (0,2)$. Let
\[
\beta_c=\frac{1}{\gamma} \Big( \frac{6-2\gamma}{6-3\gamma} \Big)^{(2-\gamma)(3-\gamma)/\gamma} \,.
\]
Assume $n^{-1/3}\ll p\ll1$. For any fixed $\beta,\xi>0$ there exists $c(\beta,\xi)>0$ 
such that if $\beta<\beta_c$, then for all $n$ sufficiently large,
\[
\P\Big(\, \bG_{n,p,\beta}\in \cG_n^1\big(0, 
\frac{1}{3} (\gamma \beta)^{\frac{3}{3-\gamma}}, \xi\big)\,\Big) \ge 1-\exp(-c\ratenp)
\]
while if $\beta>\beta_c$, then for all $n$ sufficiently large,
\[
\P\big(\, \bG_{n,p,\beta}\in \cG_n^1((\gamma\beta)^{\frac{2}{2-\gamma} }, 0, \xi)\,\big) \ge 1-\exp(-c\ratenp)\,.
\]
The same conclusions hold with $\cG_n^2$ in place of $\cG_n^1$ when $n^{-1/4}\log^{1/4} n\ll p\ll1$. 
\end{cor}

\subsection{Stability form of Finner's inequality}

Our proof of Theorem \ref{thm:struct.ER} builds on analysis in \cite{BGLZ} of the entropic optimization problem \eqref{def:probUTnmf}.
That work makes use of a Brascamp--Lieb-type inequality for product measure spaces due to Finner \cite{Finner}, restated below for the case of product probability spaces. 

We note that the stability of other special cases of the Brascamp--Lieb inequalities, such as the Riesz--Sobolev inequality, as well as ``reverse'' Brascamp--Lieb inequalities such as the Brunn--Minkowski inequality, have been a subject of recent interest -- see \cite{Christ-RS,ChON,FiJe,Figalli:ICM} and references therein.

In the following we consider a finite set $V$ and a set system $\cA$ over $V$ -- that is, a finite collection of subsets $A\subseteq V$, allowing repetitions (thus $\cA$ is in general a multiset).
We assume $\emptyset\notin\cA$. 
Say two elements $u,v\in V$ are equivalent if for every $A\in\cA$, either $\{u,v\}\subseteq A$ or $\{u,v\}\subseteq V\setminus A$, and let $\cB$ denote the partition of $V$ into equivalence classes. Thus, $\cB$ is the smallest partition of $V$ so that every element of $\cA$ can be expressed as a union of elements of $\cB$. 
To each $v\in V$ we associate a probability space $(\Omega_v,\mu_v)$.
For nonempty $A\subseteq V$ we write $\Omega_A:=\prod_{v\in A} \Omega_v$, $\mu_A:=\bigotimes_{v\in A} \mu_v$, and let $\pi_A:\Omega_V\to\Omega_A$ denote the associate coordinate projection mapping.
We abbreviate $(\Omega,\mu)=(\Omega_V,\mu_V)$. 

\begin{theorem}[Finner's inequality]
\label{thm:finner}
In the above setting, let $\Lambda=(\lam_A)_{A\in\cA}$ be a collection of positive weights such that $\sum_{A\ni v}\lam_A\le1$ for each $v\in V$.
Suppose $(f_A)_{A\in\cA}$ is a collection of functions $f_A:\Omega_A\to\R_{\ge0}$ such that $\int f_A d\mu_A \le1$ for all $A\in\cA$.
Then
\begin{equation}	\label{bd:finner}
\int_{\Omega} \prod_{A\in\cA} f_A^{\lam_A}\circ\pi_A \,d\mu \le 1. 
\end{equation}
\end{theorem}

The case that $\cA$ consists of two copies of $V$ is H\"older's inequality. 
The case that $\cA$ consists of all subsets of $V$ of a fixed size was obtained earlier by 
Calder\'on \cite{Calderon76}.

In \cite{Finner} Finner also shows that equality holds in \eqref{bd:finner} if and only if there are functions $h_{A,B}:\Omega_B\to\R_{\ge0}$ for each $A\in \cA$ and $\cB\ni B\subseteq A$, such that $f_A=\bigotimes_{B\subseteq A} h_{A,B}$, with $h_{A,B}$ and $h_{A',B}$ $\mu_B$-almost-surely equal up to a constant multiple $K_{A,A',B}>0$ whenever $B\subseteq A\cap A'$. 

For the proofs of Theorems \ref{thm:struct.ERGM} and \ref{thm:struct.ER} we make use of the following stability version of Theorem \ref{thm:finner}, which is a robust statement of the case for equality.
The proof is given in \Cref{sec:finner}.

\begin{theorem}[Stability version of Finner's inequality]
\label{thm:finner.stab}
With hypotheses as in Theorem \ref{thm:finner}, suppose 
\begin{equation}	\label{finner.stab:LB}
1-\eps\le \int_\Omega\prod_{A\in\cA} f_A^{\lam_A} \circ \pi_A \,d\mu
\end{equation}
for some $\eps\ge0$.
Then there is a collection of functions $(h_B)_{B\in\cB}$ with $h_B:\Omega_B\to\R_{\ge0}$ and $\int h_Bd\mu_B=1$ for each $B\in\cB$, such that for every $A\in \cA$,
\begin{equation}	\label{finner.stab:compare}
\Big\| f_A- \bigotimes_{B\subseteq A} h_B \Big\|_{L_1(\Omega_B)} \ls \eps^{c}
\end{equation}
where $c>0$ and the implicit constant depend only on $|V|,\cA$ and $\Lambda$. 
\end{theorem}

\begin{remark}
\label{rmk:hB1}
Note it follows from the theorem statement that $\|h_B-1\|_{L_1(\Omega_B)}\ls \eps^c$ for any $B\in \cB$ such that $\sum_{A\supseteq B}\lam_A<1$. Indeed, if such a set $B$ exists, then we can add a copy of $B$ to $\cA$, taking $f_B=1$ and $\lam_B=1-\sum_{A\supseteq B}\lam_A$.
\end{remark}

\begin{remark}
It would be interesting to determine the optimal exponent $c$ in \eqref{finner.stab:compare} depending on the structure of the set system $\cA$. It is not hard to see that $c=1/2$ is optimal for H\"older's inequality, as well as the generalized H\"older inequality (see Remark \ref{rmk:holder.c}). Inspection of the proof shows that we can also take $c=1/2$ in our application to Theorem \ref{thm:struct.ER}, where $\cA$ is the edge set of a simple graph. However, in general the proof gives a smaller value of $c$.
We mention that the work \cite{EFKY} obtains a similar stability result for the special case of the uniform cover inequality, which concerns the case that the sets in $\cA$ are all distinct and $|\{A\in\cA:v\in A\}|\equiv d$ for some $d$, and where the $f_A$ are taken to be indicators of bodies (open sets with compact closure). (The uniform cover inequality is a generalization of the Loomis--Whitney inequality.) Notably, in that setting they obtain an approximation as in \eqref{finner.stab:compare} with $c=1$, which is optimal under their hypotheses.
\end{remark}

\subsection{Discussion and future directions}

Theorem \ref{thm:struct.ERGM} shows that sparse \abbr{ergm}s can exhibit non-trivial structure different from \ER graphs, provided the Hamiltonian satisfies a growth condition. In particular, the clique and hub structures of Theorem \ref{thm:struct.ERGM} introduce some amount of transitivity, and cure the most severe forms of the degeneracy phenomenon studied in \cite{ChDi}.
However, the class of clique-hub graphs is still unlikely to be rich enough to provide useful models for social networks, which has been the main motivation for these models. 

On the other hand, Theorem \ref{thm:struct.ERGM} at least demonstrates a clique phase for \abbr{ergm}s. The appearance of cliques is a common feature of social networks that is not exhibited by other models for real-world networks, such as the preferential attachment model, which are locally tree-like \cite{ZNature}. 
Based on recent works on large deviations for random graphs, it seems likely that extensions of the models considered here incorporating degree constraints would exhibit more involved dense structures with multiple cliques \cite{BhDe, Gunby}. 
Furthermore, richer structures may result from considering Hamiltonians that depend on \emph{induced} homomorphism densities\footnote{An induced graph homomorphism from $F$ to $G$ is a mapping of the vertices of $F$ into the vertices of $G$ such that edges of $F$ are mapped onto edges of $G$ \emph{and} non-edges of $F$ are mapped onto non-edges of $G$.}, as suggested by recent work on the upper tail for induced $4$-cycle counts in the \ER graph \cite{CohenAntonir}.
The \abbr{nmf} approximation was established for the upper tail of induced subgraph counts in the \ER hypergraph in \cite{CDP} in a certain range of $p$, but apart from the aforementioned work \cite{CohenAntonir} on $4$-cycles, the structure of optimizers for the \abbr{nmf} variational problem remains open in general.

In this work we consider the ferromagnetic regime with $\hamlet$ in \eqref{hamilton} nondecreasing, or alternatively, taking $\beta_k>0$ in \eqref{Ham.intro}. The \abbr{nmf} approximation could be extended to the case of decreasing $\hamlet$ using results from \cite{CDP} on joint lower tails for homomorphism counts in $\Gnp$; see also \cite{KoSa:LT} for a result for the case $m=1$ in an optimal sparsity range. 
For the more delicate structural result of Theorem \ref{thm:struct.ERGM}, the first step of providing an asymptotic solution to the lower-tail large deviations variational problem -- analogous to the result of \cite{BGLZ} for the upper tail -- remains open. 

Finally, we mention that the large deviations results we import from \cite{CDP} were developed in the more general setting of $r$-uniform hypergraphs, and would permit extensions of Proposition \ref{prop:nmf} to exponential random \emph{hypergraph} models, which have been proposed in \cite{SSR14-hypergraphs, Battison} to model multiway interactions in social networks. As with the problem of allowing $\hamlet$ to be decreasing, the extension of the structure result of Theorem \ref{thm:struct.ERGM} requires an analysis of the \abbr{nmf} optimization problem, which has been done for hypergraphs in only a few cases \cite{LuZh12,MuBh,LiZh}.

\subsection{Notational conventions}
\label{sec:notation}

We generally write $\R_+$ for $(0,\infty)$ and $\R_{\ge0}$ for $[0,\infty)$. 
For $J\subset[n]$ we write $J^c$ for $[n]\setminus J$. 

We use the following standard asymptotic notation.
For a nonnegative real quantity $g$ and a parameter (or vector of parameters) $q$ we write $O_q(g)$ to denote a real quantity $f$ such that $|f|\le C(q)g$ for some finite constant $C(q)$ depending only on $q$. 
We also write $f\ls_q g$, $g\gs_q f$ and $g=\Omega_q(f)$ to say that $f=O_q(g)$, and $f\asymp_qg$ to mean $f\ls_qg\ls_qf$. 
When there is no subscript it means the implied constant is universal, unless noted otherwise. 
We use $C,c,c_0$, etc.\ to denote positive, finite constants whose value may change from line to line, assumed to be universal unless dependence on parameters is indicated.

For $g$ depending on an asymptotic parameter $\eps$ we write $o_{\eps\to \eps_0}(g)$ to denote a real quantity $f$ depending on $\eps$ such that $\lim_{\eps\to\eps_0} f/g=0$. In most of the article the asymptotic parameter is $n$, tending to $\infty$ through the nonnegative integers, and we suppress the limit from the subscript in this case.
The exception is Section \ref{sec:graphon} where the asymptotic parameter is generally $p\in (0,1)$ tending to zero, and we similarly suppress the subscript $p\to0$ there ($n$ does not appear in that section). We often write $f\ll g$, $g\gg f$ to mean $f=o(g)$, and $f\sim g$ to mean $f/g\to 1$. 
We also use the standard notation $\omega(a_n)$ for a positive quantity $b_n$ such that $b_n/a_n\to +\infty$ as $n\to \infty$ (this notation is only used in \Cref{sec:ideas}).
Asymptotic notation applied to a vector $\uup\in \R^m$ refers to its norm, e.g. $\uup = o(1)$ means $\|\uup\|_2=o(1)$. (We only use this notation when $m$ is fixed independent of the asymptotic parameter.)

All graphs in the article are simple -- that is, with undirected edges and no self loops. 
For a graph $F$ we use $\Verts(F), \Edges(F)$ to denote its sets of vertices and edges, respectively, and write $\verts(F)=|\Verts(F)|$, $\edges(F)=|\Edges(F)|$. By abuse of notation we extend the edge-counting function to $\cG_n$ (see \eqref{def:Gn}) as $\edges(G):= \sum_{1\le i<j\le n} G_{i,j}$. 
With $\eye_p$ as in \eqref{def:eyepq} we often use the shorthand notations
\begin{equation}	\label{def:IpE}
\eye_p(Q):= \sum_{i<j} \eye_p(Q_{i,j})\quad\text{for }\, Q\in\cQ_n\,,\quad\text{and} \qquad
\eye_p(\cE):=\inf_{Q\in\cE} \eye_p(Q) \quad\text{for }\, \cE\subseteq\cQ_n.
\end{equation}

\subsection{Organization of the paper}

We begin in  \Cref{sec:ideas} with an overview of the main ideas for the proof of \Cref{thm:struct.ER} on the conditional structure of \ER graphs, which is carried out in Sections \ref{sec:graphon} and \ref{sec:struct.ER}.
The core 
of the argument is the stability analysis for an entropic optimization problem on 
graphon space, 
done in Section \ref{sec:graphon} (see Proposition \ref{prop:graphon.stab}), to which end
Theorem \ref{thm:finner.stab} plays a key role.
The proof of \Cref{prop:graphon.stab} is illustrated in \Cref{sec:ideas} for the case of the joint upper tail for counts of $K_{1,2}, C_3$ and $C_4$.
In Section \ref{sec:struct.ER} we combine Proposition \ref{prop:graphon.stab} with the
quantitative large deviation results of Propositions \ref{prop:LD.spec} and \ref{prop:LD.Bstar} due to
\cite{CoDe,CDP} to establish Theorem \ref{thm:struct.ER}.
In Section \ref{sec:ERGMs} we establish our main results on \abbr{ergm}s, namely Proposition \ref{prop:nmf}, Theorem \ref{thm:nmf2} and Theorem \ref{thm:struct.ERGM} using corresponding large deviation results for \ER graphs
(where the non-asymptotic upper-tail estimate of Proposition \ref{prop:rough-ubd} allows for 
an a-priori truncation to a compact set).
Proposition \ref{prop:edgeF} and Corollary \ref{cor:edge-triangle} on edge-$F$ models are proved in Section \ref{sec:edgeF}.
The appendices contain the proofs of the stability version of Finner's inequality (Theorem \ref{thm:finner.stab}), Proposition \ref{prop:opt}, and a non-asymptotic upper-tail estimate of the correct shape for homomorphism densities (Proposition \ref{prop:rough-ubd}).

\section{Proof ideas for the conditional structure of \ER graphs}
\label{sec:ideas}

In this section we overview the proof of \Cref{thm:struct.ER}, which occupies Sections \ref{sec:graphon} and \ref{sec:struct.ER}. 
We write $\bG=\Gnp$ for the \ER graph. 
We view the $m$-tuple $\uF=(F_1,\dots, F_m)$ as fixed and suppress it from the notation, thus writing $\cU_p(\uup),\Phi_{n,p}(\uup),\phi(\uup)$ etc.
We sometimes further drop $\uup$ to lighten notation. 

\subsection{Tail asymptotics from covering and continuity}

To establish the structure of $\bG$ conditioned to lie in $\cU_p(\uup)$ we make use of tools developed in  \cite{CoDe,CDP,BhDe} to show 
\begin{equation}	\label{UT.asymp2}
\log\P(\bG\in\cU_p(\uup)) \sim- \phi(\uup) \ratenp\,.
\end{equation}
We begin by recalling these tools and how they are used to establish \eqref{UT.asymp2}.
Roughly speaking, \cite{CoDe,CDP} show that for  certain  norms $\|\cdot\|_\star$ over the space of all $\R$-weighted graphs, and an exceptional set $\cE$ in the ${n\choose2}$-dimensional cube $\cQ_n$ of $[0,1]$-weighted graphs over $[n]$, we have
\begin{enumerate}
\item[(a)] $\P(\bG\in\cE)=\exp(-\omega(\rate_{n,p}))$,
\item[(b)] $\cQ_n\setminus \cE$ can be covered by $\exp(o(\rate_{n,p}))$ $\eps$-balls, and
\item[(c)] the rescaled homomorphism counts $X\mapsto t(F,X/p)$ are $O_F(1)$-Lipschitz on $\cQ_n\setminus\cE$. 
\end{enumerate}
(See  \Cref{sec:LDPs} for full statements, which involve interdependent parameters controlling the measure of $\cE$, the size of the covering and Lipschitz constants. Here we make informal use of asymptotic notation to glide over these technical issues.)
Approximating superlevel sets
$
\cU_p(\uup) 
$
from within and without by unions of $\eps$-balls, one can then deduce that
\begin{align}
\sup_{B\subseteq \cU_p(\uup)\setminus \cE} \P(\bG\in B)\le 
\P(\bG\in\cU_p(\uup)) &\le \P(\bG\in\cE) + \sum_B \P(\bG\in B)	\label{ideas.covering}\\
&\le e^{-\omega(\rate_{n,p})} + e^{o(\rate_{n,p})}\max_B\P(\bG\in B)	\notag
\end{align}
where the sum and maximum in the last two expressions are taken over $\eps$-balls in the covering of $\cU_p(\uup)$. 
For each ball in the covering we can use the convexity of $B$ to show 
\[
\log\P(\bG\in B)\sim- \eye_p(B)
\]
(recall the notation \eqref{def:IpE}). Indeed, the right hand side is always an upper bound for the left hand side by the minimax theorem, and a near-matching lower bound under some generic assumptions on $B$ can be established by a standard tilting argument.
Then using the continuity of $t(F,\cdot/p)$ we can show the union of balls covering $\cU_p(\uup)$ is contained in $\cU_p(\uup-C\eps \1)$ for some constant $C=C(\uF)>0$. 
Arguing similarly for the lower bound we obtain 
\begin{equation}	\label{ideas.approx}
\log \P(\bG\in\cU_p(\uup)) \sim -\min\big\{ \omega(r_{n,p})\,,\; \Phi_{n,p}(\uup+O(\eps)) + o(r_{n,p}) \big\}.
\end{equation}
From the works \cite{BGLZ,BhDe} we know
\begin{equation}	\label{ideas.phis}
\Phi_{n,p}(\uup) \sim \phi(\uup)\rate_{n,p}
\end{equation}
which combines with \eqref{ideas.approx} to give \eqref{UT.asymp2}.

\subsection{Reduction to a stability problem over $\cQ_n$}

Now we describe how the above covering argument is adapted in the present work to reduce the proof of the refined result of \Cref{thm:struct.ER} to a stability analysis of the optimization problem defining $\Phi_{n,p}(\uup)$.
Recalling the weighted clique-hub graphs from \eqref{def:QIJ}--\eqref{def:Qab}, we let $\cO_\xi=\cO_\xi(\uup)\subset\cQ_n$ denote the $\xi$-neighborhood under the norm $\|\cdot\|_\star$ of the set 
\begin{equation}	\label{opt0}
\cO_0=\cO_0(\uup):= \bigcup_{(a,b)\in \Opt(\phi;\uup)} \cQ_n(a,b).
\end{equation}
Since $\eye_p(\cU_p)=\Phi_{n,p}(\uup)\sim \phi(\uup)\ratenp$, 
it suffices to show 
\begin{align}
\log \P(\bG\in \cU_p\setminus\cO_\xi) &\le - \eye_p(\cU_p) - (\eta(\xi)+o(1))\rate_{n,p}
\label{ideas.goal1} 
\end{align}
where $\xi\mapsto\eta(\xi)>0$ is a continuous, strictly increasing function (depending on $\uF$ and $\uup$ but independent of $n$ and $p$), with $\eta(0)=0$.
(In fact we only take $\cO_\xi$ to be a $\|\cdot\|_\star$-neighborhood as above in the proof of \Cref{thm:struct.ER}(b); for (a) we need to argue slightly differently, but we do not discuss this technical issue here.) We in fact show this with $\eta(\xi)= \xi^{O_{\uF,\uup}(1)}$. 

Following \eqref{ideas.covering}, the left hand side in \eqref{ideas.goal1} can be bounded roughly like
\begin{equation}
\log \P(\bG\in \cU_p\setminus\cO_\xi) \le - 
\eye_p(\cU_p\setminus\cO_{\xi/2})
\end{equation}
(along with some errors that we omit here). 
To obtain \eqref{ideas.goal1} it thus suffices to show 
\begin{equation}
\label{ideas.goal2}
\eye_p(\cU_p\setminus\cO_{\xi/2})\ge \eye_p(\cU_p) + (\eta(\xi)+o(1))\rate_{n,p}\,.
\end{equation}
In particular, the weighted clique-hub graphs are the \emph{only} near-optimizers of $\eye_p$ over $\cU_p$,
but \eqref{ideas.goal2} importantly expresses a stronger stability property: any weighted graph in $\cU_p$ that is a distance $\Omega(1)$ from the class \eqref{opt0} has entropy $\Omega(\ratenp)$ above the minimal value. 


\subsection{A stability problem on graphon space}
\label{sec:ideas.graphon}

To establish \eqref{ideas.goal2} we need to look into the proof in \cite{BGLZ,BhDe} of the asymptotic \eqref{ideas.phis} reducing the optimization problem over $\cQ_n$ to an optimization problem over the plane. 
Following \cite{LuZh14,BGLZ}, for this task it is convenient to suppress the asymptotic parameter $n$ by passing to an optimization problem over the \emph{infinite}-dimensional space of graphons. (The benefit of working with graphons is essentially notational;  these arguments can also be carried out over $\cQ_n$, but working with graphons eliminates consideration of rounding errors.) 

Let $\cW$ be the set of all symmetric (Lebesgue) measurable function $g:[0,1]^2\to[0,1]$, which are called \emph{graphons}. Any $Q\in\cQ_n$ is associated to an element $g_Q\in\cW$ with $g_Q(x,y):= Q_{\lf xn\rf, \lf yn\rf}$. 
The homomorphism density functionals extend to symmetric measurable $f:[0,1]^2\to\R_+$ in the natural way:
\begin{equation}	\label{def:tF.graphon}
t(F, f):=\int_{[0,1]^{\Verts(F)}} \prod_{e\in\Edges(F)} f(x_e)dx
\end{equation}
where $x_e:=(x_u,x_v)$ for $\{u,v\}\in \Edges(F)$. One easily verifies {that} for any $Q\in\cQ_n$, 
\begin{equation}	\label{tF.Qg}
t(F,Q)= t(F, g_Q)\,, \qquad \eye_p(Q) = n^2\eye_p(g_Q) \,,
\end{equation}
where we abbreviate $\eye_p(g):=\frac12\int_{[0,1]^2}\eye_p(g(x,y))dxdy$. 
The graphon optimization problem is 
\begin{equation}	\label{def:probUTg.ideas}
\UTg_p(\uup):= \eye_p(\cV_p(\uup)),\qquad \cV_p(\uup):= \bigcap_{k\in[m]}\{ g\in\cW: t(F_k, g/p) \ge 1+\up_k \big\}
\end{equation}
where $\eye_p(\cV):=\inf\{ \eye_p(g):g\in \cV\}$. 
From the embedding $Q\mapsto g_Q$ of $\cQ_n$ in $\cW$ and the identities \eqref{tF.Qg} we see that $n^2\UTg_p(\uup)\le \Phi_{n,p}(\uup)$ for every $n$. 

For $a,b>0$ we let
\begin{equation}\label{def:Wab}
\cW(a,b) := \big\{\, g_{S,T}: |S|= (a p^\Delta)^{1/2},\, |T|= b p^\Delta \,\big\},
\end{equation}
analogous to the set $\cQ_n(a,b)$ from \eqref{def:Qab}, 
where for disjoint $S$ and $T$ we denote the graphon
\begin{equation}	\label{def:gST}
g_{S,T}= p + (1-p)\big[ \chi_{S\times S} + \chi_{T\times T^c} + \chi_{T^c\times T} \big]
\end{equation}
with $\chi_{S \times T}$ the indicator function for $S\times T\subset[0,1]^2$. 
For any $g_{S,T}\in\cW(a,b)$ we have
\[
\eye_p(g_{S,T}) = (\frac12|S|^2 + |T|(1-|T|))\log(1/p) \sim (\frac12a + b)p^\Delta\log(1/p) \,,
\]
where the asymptotic is in the limit $p\to0$. 
Furthermore, recalling the functions $T_F:\R_+^2\to\R$ from \eqref{def:TF}, it is not hard to show (see \cite{BGLZ}), that
\begin{equation}
t(F, g/p) \sim T_F(a,b) \quad\forall g\in \cW(a,b).
\end{equation}
By restricting the infimum in \eqref{def:probUTg.ideas} to $\cW(a,b)$ and optimizing over $a,b\ge0$, we thus see that
\begin{equation}
\UTg_p(\uup) \le 
(\phi(\uup) +o(1))p^\Delta \log(1/p). 
\end{equation}
The main step in the proof of \eqref{ideas.phis} in \cite{BGLZ,BhDe} is to  establish a matching lower bound, i.e. to show the following:
\begin{equation}	\label{ideas.INF}
g\in\cV(\uup) \quad\Longrightarrow\quad
\eye_p(g) \ge (\phi(\uup) - o(1)) p^\Delta \log(1/p)\qquad \text{{\bf (INF)}}.
\end{equation}

While the optimum $\UTg_p(\uup)\sim\phi(\uup) p^\Delta \log(1/p)$ is nearly attained by the elements of $\cW(\csizes,\bsizes)$ for any $(\csizes,\bsizes)\in\R_+^2$ attaining the infimum for the planar optimization problem $\phi(\uup)$, it does not rule out the existence of near-optimizers of a different nature. 
The main technical result of the present work is \Cref{prop:graphon.stab} showing that the clique-hub graphons are the only near-optimizers in \eqref{def:probUTg.ideas}, and moreover we have the following stability property:\\

\noindent{\bf (STAB):} \emph{
For any $\eta\in(0,1)$,}
\begin{align}
&g\in \cV_p(\uup-\eta \1) 	\label{ideas.hyp1}\\
\text{ and }\quad&\eye_p(g) \le (\phi(\uup)+\eta)p^\Delta\log(1/p)		\label{ideas.hyp2}\\
\Longrightarrow\quad& 
\exists\, (\csizes,\bsizes)\in \Opt(\phi;\uup)\,,\;g_{S,T}\in \cW(\csizes,\bsizes)\;:\; \|g-g_{S,T}\|_2 =O(\eta^cp^{\Delta/2})	\label{ideas.concl}
\end{align}
\emph{where the implicit constant and $c$ depend only on $\uF$ and $\uup$.} \\

It is not hard to see that (STAB)$\Rightarrow$(INF) -- see \Cref{rmk:graphon.stab}. 

From (STAB) it is straightforward to deduce \eqref{ideas.goal2} and conclude the proof of \Cref{thm:struct.ER}. (In fact we need a slight variant of \eqref{ideas.concl} in order to pass back to a weighted graph $Q$, but we do not discuss this technical issue here.)

\subsection{Sketch of the proof of (STAB) for the joint upper tail of $(K_{1,2},C_3,C_4)$-counts}

The basic plan of the proof of (STAB) is to go through the proof of (INF) from \cite{BGLZ,BhDe}, and use the additional hypothesis \eqref{ideas.hyp2} to deduce that various bounds from that argument in fact hold with near-equality. We then establish stability versions of these bounds (such as \Cref{thm:finner.stab} for Finner's inequality), that let us deduce structural information on $g$. 

Here we overview the main steps 
for the case $\uF=(F_1,F_2,F_3):=(K_{1,2}, C_3,C_4)$, where we recall that $K_{1,2}$ is the complete bipartite graph with one vertex on one side and two on the other (or the path on three vertices), and $C_\ell$ is the cycle on $\ell$ vertices.
Illustrations of the planar optimization problem defining $\phi(\uup)$ for this example were shown in \Cref{fig:ab1}.
In this case we have that $\Delta=2$. We	
assume throughout that $p=o(1)$.\\


\noindent{\bf (INF) Step 1: Approximation of $t(F_k,g/p)$ via degree thresholding}.
The first step for proving \eqref{ideas.INF} is to expand the homomorphism densities $t(F_k,g/p)$ in terms of $f:=g-p$ and identify the dominant terms. 
For example, for $C_3$ we start with the identity
\[
t(C_3, g/p) = 1+ 3t(K_{1,2}, f/p) + 3 t(K_{1,1}, f/p) + t(C_3, f/p). 
\]
One can show (see \cite[Lemma 4.2]{BGLZ}), that if $\jay_p(f) \ls p^2$ (which we have from \eqref{ideas.hyp2a}), then for the third term we have $t(K_{1,1},f/p) = p^{-1}\int f=o(1)$, so 
\[
t(C_3, g/p) = 1+ 3t(K_{1,2}, f/p) + t(C_3, f/p) + o(1). 
\]
The right hand side is further refined based on a decomposition of the support of $f$ 
according to a partition of $[0,1]$ into points of ``high degree'' and ``low degree'' for $f$.
Specifically, with a cutoff parameter $d=d(p)=o(1)$ to be chosen later, we let 
\[
D=D_d(g)=\Big\{ x\in[0,1]: \int_0^1 f(x,y)dy \ge d\Big\}.
\]
Then set
\begin{align*}
\tf:= f\chi_{D\times D^c}\,,\quad
\hf:= f\chi_{D^c\times D^c} \,,\quad \chf:= f(\chi_{D\times D^c} + \chi_{D^c\times D}).
\end{align*}
Thus $\chf(x,y)= \tf(x,y) + \tf(y,x)$. (Note that $\hf,\chf$ are graphons, while $\tf$ is asymmetric and hence not a graphon.)
For an asymmetric function $h:[0,1]^2\to \R$ we can define a bipartite analogue of the homomorphism densities: let 
\begin{equation}
\t t(K_{1,2}, h):= \int_{[0,1]^3} h(x,y)h(x,z) dxdydz\,.
\end{equation}
That is, the side of $K_{1,2}$ with one vertex is mapped to the first coordinate of $h$. 
(An extension for embeddings of general bipartite graphs is given in \eqref{def:t.asymmetric}.)
Thus, 
\[
\t t(K_{1,2},\tf/p) = p^{-2}\int_{[0,1]^3} \tf (x,y)\tf(x,z) dxdydz = p^{-2}\int_D\int_{D^c}\int_{D^c} f(x,y)f(x,z)dzdydx
\]
gives the density of embeddings of $K_{1,2}$ with the vertex of degree 2 mapped to the high-degree part of $f$ and the others mapped to the low-degree part. 
A highly non-trivial result from \cite{BGLZ} that we essentially take as a black box shows that there exists a choice of $d$ depending on $g$ and satisfying $\sqrt{p}\ll d\ll 1$ such that
\begin{align}
t(K_{1,2},g/p)
&=1+3 \t t(K_{1,2}, \tf/p) + o(1)	\label{ideas.tK12}\\
t(C_3,g/p)
&=1 + 3\t t(K_{1,2},\tf/p)+t(C_3,\hf/p)+o(1)	\label{ideas.tC3}\\
t(C_4, g/p)
&= 1+ 4 \t t(K_{1,2}, \tf/p) +t(C_4, \hf/p) + t(C_4, \chf/p) + o(1). \label{ideas.tC4}
\end{align}
See \Cref{lem:terms} for the general statement. \\

\noindent{\bf (INF) Step 2: Finner's inequality}.
Applying \Cref{thm:finner}, we can bound
\begin{equation}	\label{finner.K12}
\t t(K_{1,2}, \tf/p) \le \|\tf/p\|_2^2
\end{equation}
and for each $\ell\ge3$,
\begin{align}	\label{finner.Ck}
t(C_\ell,\hf/p) &\le \|\hf/p\|_2^\ell\,,\qquad
t(C_\ell,\chf/p) \le \|\chf/p\|_2^\ell = 2\|\tf/p\|_2^\ell \,.
\end{align}
Setting 
\begin{equation}
a'_g:= \|\hf/p\|_2^2\,,\qquad b'_g:= \|\tf/p\|_2^2
\end{equation}
and combining these bounds with \eqref{ideas.tK12}--\eqref{ideas.tC4}, we obtain
\begin{align}
t(K_{1,2},g/p)&\le 1+ 3b_g' +o(1) 
\label{tK12b}\\
t(C_3,g/p)&\le 1+ 3b_g' + (a_g')^{3/2}+o(1) 
\label{tC3b} \\
t(C_4,g/p)&\le 1+ 4b_g' + 2(b_g')^2 + (a_g')^2 +o(1).
\label{tC4b}
\end{align}
\quad\\

\noindent{\bf (INF) Step 3: $L_2$-entropy inequality.}
The final step for \eqref{ideas.INF} is to connect bounds \eqref{tK12b}--\eqref{tC4b} with the entropy $\eye_p(g)$. 
It is convenient to work with the shifted and scaled function
\begin{equation}	\label{def:jay}
\jay_p(x) = \frac{\eye_p(p+x)}{\log(1/p)},\quad x\in [-p,1-p].
\end{equation}
As with $\eye_p$ we will abbreviate $\jay_p(g):= \frac12\int \jay_p(g(x,y))dxdy$.
For (INF) our aim is to show
\begin{equation}	\label{ideas.INFgoal}
\jay_p(f) \ge (\phi(\uup)-o(1)) p^2. 
\end{equation}
A key estimate from \cite{BGLZ} is the pointwise bound
\begin{equation}	\label{jp.LB}
\jay_p(x)\ge x^2\,,\quad  x\in [-p,1-p].
\end{equation}
Setting
\begin{align*}
a_g:= p^{-2} \jay_p(\hf) \,,
\qquad b_g:= p^{-2} \jay_p(\chf)
\end{align*}
then integrating \eqref{jp.LB} immediately yields
\begin{equation}	\label{ag'bg'}
a_g' \le a_g\,,\qquad b_g'\le b_g. 
\end{equation}
\quad\\

\noindent{\bf (INF) Step 4: The planar problem.}
Recall the functions $T_F$ from \eqref{def:TF}. 
Writing $T_k:=T_{F_k}$, one easily shows
\begin{align}	\label{ideas.Tab}
T_1(a,b) &=1+b\,,\quad
T_2(a,b) =1+3b + a^{3/2}\,,\quad
T_3(a,b) =1+4b + 2b^2 + a^2.
\end{align}
We denote the feasible region for the optimization problem defining $\phi(\uup)$ by $V(\uup)$, thus:
\begin{equation}
V(\uup):=\bigcap_{k=1}^3V_k(s_k)
\,,\quad
V_k(s):=\{ (a,b)\in\R_+^2: T_k(a,b)\ge 1+s\}\,.
\end{equation}
See \Cref{fig:ab1} for plots of the region $V(\uup)$ for various values of $(s_1,s_2,s_3)$. 

Note that the right hand sides in \eqref{tK12b}--\eqref{tC4b} are $T_k(a_g',b_g')+o(1)$ for $k=1,2,3$, respectively. 
Since the functions $T_k$ are monotone, by combining \eqref{ag'bg'} with \eqref{tK12b}--\eqref{tC4b} we see that 
\begin{equation}	\label{cVtoV}
g\in \cV(\uup) \quad\Longrightarrow \quad (a_g,b_g) \in V(\uup-o(1)). 
\end{equation}
On the other hand,
\begin{equation}	\label{ideas.ip-ab}
\jay_p(g) =\jay_p(\hf) + \jay_p(\chf) +\jay_p(f\chi_{D\times D})\sim (\frac12a_g+b_g) p^2
\end{equation}
(the contribution from the small set $D\times D$ is negligible). 
Since $\phi(\uup)=\inf\{ \frac12a+b: (a,b)\in V(\uup)\}$, from \eqref{cVtoV} and \eqref{ideas.ip-ab} we obtain \eqref{ideas.INFgoal} to conclude the proof of \eqref{ideas.INF}.\\

\noindent{\bf Proof of (STAB)}.
Now suppose 
\eqref{ideas.hyp1} and \eqref{ideas.hyp2} hold.
We continue to denote $f:=g-p$.
In this sketch we further assume 
$f\ge0$ a.e. (the alternative case has to be dealt with but it is a minor technical point). 
We want to locate a point $(\csizes, \bsizes)\in \Opt(\phi;\uup)$ and $g_{S,T}\in\cW(\csizes,\bsizes)$ such that $g\approx g_{S,T}$ in $L_2$. 
In fact we will show this holds with $T=D$. 

\begin{figure}
\includegraphics[width=7cm]{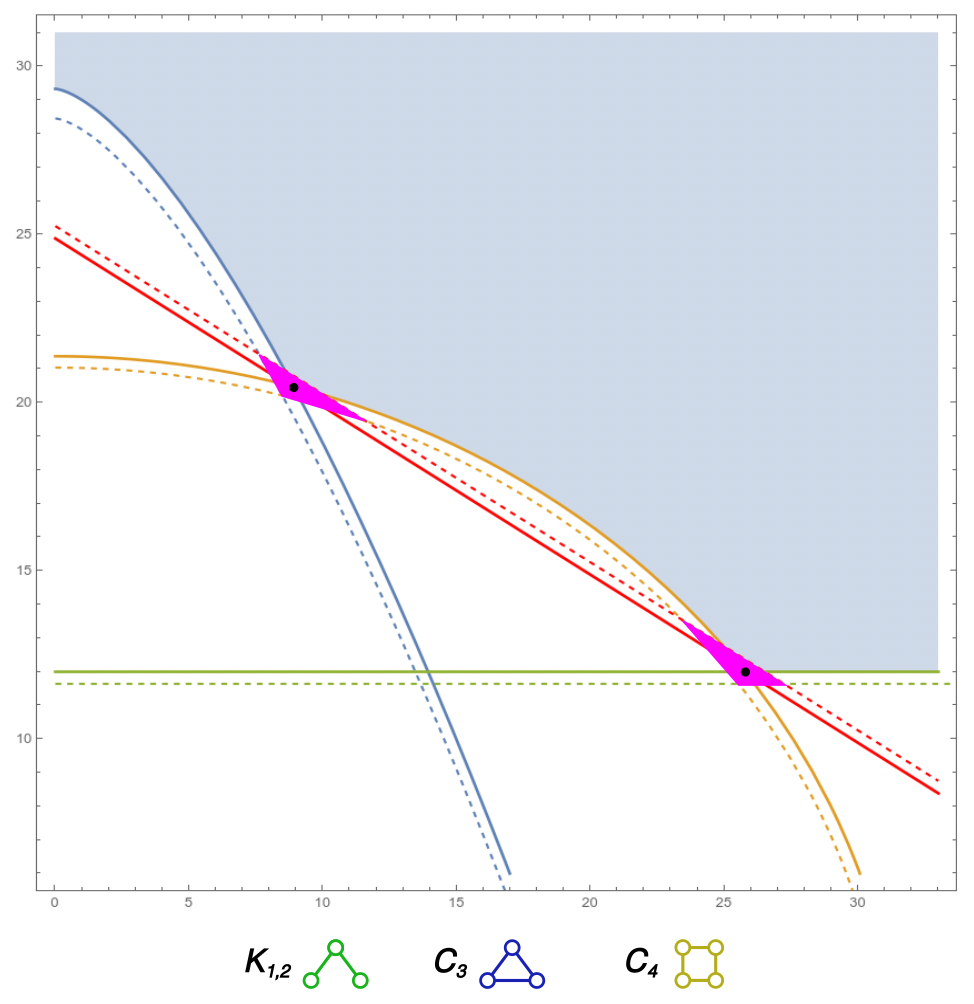}
\caption{Stability for the planar optimization problem: (With setup as in \Cref{fig:ab1}, again with $a$ on the horizontal axis and $b$ on the vertical axis.) For $(s_1,s_2,s_3) = (12,88,1000)$, the set $\Opt(\phi;\uup)$ of minimizers of $\frac12a+b$ over $V(\uup)$ contains two points (black). $V(\uup)$ is in light blue. 
For $\eta$ sufficiently small,  $V(\uup-2\eta\1)\cap \{ (a,b): \tfrac12a+b\le \phi(\uup) + \eta\}$ (magenta, bounded by dashed lines) is contained in a $O_{\uF,\uup}(\eta)$-neighborhood of $\Opt(\phi;\uup)$.  
\label{fig:ab2}
}
\end{figure}

\begin{itemize}
\item \emph{Stability for the planar problem}.\;
The entropy bound \eqref{ideas.hyp2} implies
\begin{equation}	\label{ideas.hyp2a}
\frac12a_g + b_g \le p^{-2}\jay_p(f) \le \phi(\uup)+\eta.
\end{equation}
On the other hand, the hypothesis $g\in \cV(\uup-\eta \1)$ and \eqref{cVtoV} give 
\begin{equation}	\label{agbg.in}
(a_g,b_g)\in V(\uup-\eta\1-o(1))\subset V(\uup-2\eta\1). 
\end{equation}
Hence,
$
(a_g,b_g)\in V(s-2\eta)\cap \{ (a,b): \tfrac12a+b\le \phi(\uup) + \eta\}. 
$
If $\eta$ is sufficiently small, this region is a union of small neighborhoods of the points in $\Opt(\phi;\uup)$ -- see \Cref{fig:ab2}. 
This fact can be seen as a stability statement for the planar optimization problem defining $\phi(\uup)$.
The general statement is given in \Cref{lem:planar.stab}.

We hence deduce that $(a_g,b_g)$ is within distance $O_{\uF,\uup}(\eta)$ of some point $(\csizes,\bsizes)\in \Opt(\phi;\uup)$. It is clear that at any point in $\Opt(\phi;\uup)$, two of the four constraints $a\ge0$ and $T_k(a,b)\ge 1+s_k$, $k=1,2,3$ hold with equality (see \Cref{fig:ab2}, where $\Opt(\phi;\uup)$ consists of two such intersection points; any of the 6 pairings of constraints can yield a point in $\Opt(\phi;\uup)$ for some choice of $\uup$). 
For concreteness let us suppose $(\csizes,\bsizes)$ is the point such that $T_k(\csizes,\bsizes)=1+s_k$ for $k=2,3$ (in \Cref{fig:ab2} this is where the solid blue and yellow lines intersect). Since the functions $T_k$ are continuous (indeed they are locally Lipschitz), we deduce that $T_k(a_g,b_g)\le 1+s_k+O_{\uF,\uup}(\eta)$. 
But we already showed in \eqref{agbg.in} that $T_k(a_g,b_g)\ge 1+s_k-2\eta$. Thus, every bound leading to \eqref{agbg.in} -- namely,  \eqref{finner.K12}, \eqref{finner.Ck} and \eqref{ag'bg'} -- actually hold with equality up to an additive error $O_{\uF,\uup}(\eta)$. \\

\item \emph{Stability for Finner's inequality}.\;
From the near-equality in the applications of Finner's inequality in \eqref{finner.K12}, \eqref{finner.Ck} we obtain from \Cref{thm:finner.stab} that 
\begin{equation}	\label{ideas.tensor}
\tf\approx h_1\otimes h_2\quad\text{ and } \quad \hf\approx h_3\otimes h_3
\end{equation}
in $L_2$, for some $h_1,h_2,h_3:[0,1]\to [0,1]$ supported on $D, D^c$ and $D^c$, respectively. \\

\item \emph{Stability for the $L_2$-entropy inequality}. \;
The bound \eqref{ag'bg'} came from integrating the pointwise bound \eqref{jp.LB}.
One can show that this bound is only a near-equality near $x=0$ and $x=1-p$. 
In \Cref{lem:Jp} we give a quantitative version of this, which allows us to show that near-equality in \eqref{ag'bg'} implies that $\tf$ and $\hf$ are close to indicator functions in $L_2$. 
With \eqref{ideas.tensor} and a bit of work one can deduce that $\tf\approx \chi_{T\times T^c}$ and $\hf \approx \chi_{S\times S}$ for some $T\subseteq D$, $S\subset D^c$
(see \Cref{claim:indicator}), and putting these together we conclude that $g\approx g_{S,T}$,
as desired.
\end{itemize}

\section{Stability for the upper-tail entropic optimization problem}
\label{sec:graphon}

Throughout this section the asymptotic notation $o(1)$ and $\sim$ is with respect to the limit $p\to 0$ unless indicated otherwise
(the asymptotic parameter $n$ makes no appearance here).  
We denote the Lebesgue measure on $[0,1]^d$ by $|\cdot|$. All integrals are understood to be with respect to Lebesgue measure unless otherwise indicated. For the Lebesgue spaces $L_q([0,1]^d)$ with $d=1,2$ we write $\|g\|_q$ for the
$L_q$-norm for $q \ge 1$.
For a set $E\subset[0,1]^d$ we take $E^c$ to mean $[0,1]^d\setminus E$. 

As was shown in \cite{BGLZ,BhDe}, the infimum in the upper-tails \abbr{nmf} optimization problem \eqref{NMF.tails} over the ${n\choose2}$-dimensional domain $\cQ_n$ is asymptotically attained by matrices having off-diagonal entries in $\{p,1\}$, taking value 1 on the edge sets of a clique and complete bipartite graph of appropriate sizes, effectively reducing \eqref{NMF.tails} to the two-dimensional problem \eqref{def:phi}.
In this section we prove Proposition \ref{prop:graphon.stab} below, showing that \emph{near}-optimizers for \eqref{NMF.tails} are close to such ``clique-hub" matrices, a key step towards proving Theorems  
\ref{thm:struct.ERGM} and \ref{thm:struct.ER}.
Following \cite{BGLZ,BhDe}, we establish this stability in the broader setting of an infinite-dimensional optimization problem over the space of graphons, whose definition we now recall.

A \emph{graphon} is a symmetric measurable function $g:[0,1]^2\to [0,1]$, and an \emph{asymmetric graphon} is a measurable function $g:[0,1]^2\to [0,1]$ with no symmetry constraint. 
We denote the space of graphons by $\cW$.
Given a partition $\cP$ of $[0,1]$ into finitely many measurable sets, we let $\cW_\cP$ denote the subspace of graphons that are a.e.\ constant on sets $S\times T$ with $S,T\in\cP$.
The $F$-homomorphism density functional on symmetric measurable $f:[0,1]^2\to\R_+$ was recalled in \eqref{def:tF.graphon}.
For $F$ bipartite with ordered bipartition 
$(A,B)$
the definition extends unambiguously to asymmetric functions, with $x_e:= (x_u,x_w)$ for $u\in A,w\in {B}$. We write $t(F,f; A)$ to indicate which part of the vertex bipartition we take to map to the first argument of $f$, thus:
\begin{equation}	\label{def:t.asymmetric}
t(F,f;A):= \int_{[0,1]^{A}} \int_{[0,1]^{B}}\prod_{(u,v)\in \Edges(F)} f(x_u,y_w) dydx\,.
\end{equation}
(In \Cref{sec:ideas} we wrote $\t t(K_{1,2},f)$ for $t(K_{1,2},f; A)$ with $A$ the single vertex of degree 2.)

For $\uup\in\R_+^m$, $p\in (0,1)$, and $\uF=(F_1,\dots,F_m)$ a fixed sequence of distinct, connected graphs of maximum degree $\Delta\ge2$,
the graphon upper-tail entropic optimization problem -- referred to hereafter as the \emph{graphon problem} -- is defined 
\begin{equation}	\label{def:probUTg}
\UTg_p(\uF,\uup):= \eye_p(\cV_p(\uF,\uup)),\qquad \cV_p(\uF,\uup):= \bigcap_{k\in[m]}\{ g\in\cW: t(F_k, g/p) \ge 1+\up_k 
\big\}
\end{equation}
where we recall the notation $\eye_p(g):=\frac12 \int \eye_p\circ g$ and $\eye_p(\cV):=\inf_{g\in\cV}\{ \eye_p(g) \}$ from \Cref{sec:ideas}.

The following result, extracted from an argument in \cite{BhDe} (which builds upon \cite{BGLZ}) shows that in the sparse limit $p\to0$, the infinite-dimensional graphon problem \eqref{def:probUTg} reduces to the 2-dimensional problem \eqref{def:phi}.

\begin{theorem}[Solution of the graphon problem]
\label{thm:BhDe2}
For fixed $\uF,\uup$ we have
\begin{equation}	\label{BhDe.asymp}
\UTg_p(\uF,\uup)\sim \UTab_{\uF}(\uup)p^\Delta\log(1/p).
\end{equation}
\end{theorem}


Recall 
our notation $\Opt(\phi;\uup)$ for the set of optimizers for $\phi_{\uF}(\uup)$.
As we reviewed in \Cref{sec:ideas}, 
the optimum 
\begin{equation}\label{eq:gst-opt}
\eye_p(g_{S,T}) \sim \UTab_{\uF}(\uup)p^\Delta \log(1/p),
\end{equation}
is attained on 
a class $\cW(\csizes,\bsizes)$ of \emph{clique-hub} graphons $g_{S,T}$ for $(\csizes,\bsizes)\in \Opt(\phi;\uup)$.
The following stability result shows that these are the only optimizers:
any 
near-minimizer for $\eye_p$ over $\cV_p(\uF,\uup)$ must be close to an
element of $\cW(\csizes,\bsizes)$ for some $(\csizes,\bsizes)\in\Opt(\phi;\uup)$.

\begin{prop}[Stability for the graphon problem]
\label{prop:graphon.stab}
Let $\uF=(F_1,\dots,F_m)$ be a sequence of graphs as above
and let $\uup\in \R_+^m$ and $\eta>0$.
There exist $c_0(\uF)>0$
and $p_0(\uF,\uup,\eta)>0$ 
such that the following holds for all $0<p\le p_0$. 
For any graphon $g$ satisfying
\begin{equation}
t(F_k,g/p) \ge 1+\up_k -\eta \qquad\forall \; 1\le k\le m
\end{equation}
and
\begin{equation}	\label{eyep.assume}
\eye_p(g)\le (\UTab_{\uF}(\uup)+\eta ) p^\Delta\log(1/p)\,,
\end{equation}
there exist $(\csizes,\bsizes)\in \Opt(\phi;\uup)$ and 
$g_{S,T} \in \cW(\csizes,\bsizes)$ such that 
\begin{equation}	\label{approx:gST}
\|g-g_{S,T}\|_2 \ls_{\uF,\uup} \eta^{c_0(\uF)} p^{\Delta/2}\,.
\end{equation}
Moreover, if $g\in \cW_\cP$ for some finite partition $\cP$ of $[0,1]$ then we may take 
$g_{S,T}\in \cW_\cP\cap \cW(\csizesp, \bsizesp)$ for some $\csizesp,\bsizesp$ such that $|\csizesp-\csizes|, | \bsizesp-\bsizes|
\ls_{\uF,\uup} \eta^{c_0(\uF)}$.
\end{prop}

\begin{remark}
The point here is that \eqref{approx:gST} improves over the trivial bound
\begin{equation}
\label{gST.triv}
\|g-g_{S,T}\|_2 \ls_{\uF,\uup} p^{\Delta/2}\,.
\end{equation}
To see that \eqref{gST.triv} always holds, set $f_{S,T}:= g_{S,T}-p$, so that 
\[
\|g-g_{S,T}\|_2= \|g-p-f_{S,T}\|_2\le \|g-p\|_2+\|f_{S,T}\|_2.
\]
From the definition \eqref{def:gST}--\eqref{def:Wab} we clearly have $\|f_{S,T}\|_2^2\le (\csizes+ \bsizes)p^\Delta\ls_{\uF,\uup}p^\Delta$. On the other hand, in \Cref{lem:Jp} below we see $\eye_p(p+y)/\log(1/p)\ge y^2$ for $y\in[-p,1-p]$, giving
\begin{equation}
\int \eye_p\circ g \le Kp^\Delta\log(1/p) \quad\Longrightarrow\quad \|g-p\|_2^2\le Kp^\Delta
\end{equation}
and so we deduce \eqref{gST.triv} from the assumption \eqref{eyep.assume}.
\end{remark}

\begin{remark}
The conclusion for the case $g\in \cW_\cP$ is needed for the proof of Theorem \ref{thm:struct.ER}, where we take $\cP$ to be the partition of $[0,1]$ into intervals of length $1/n$ in order to pass from graphons to weighted graphs over $[n]$.
\end{remark}

\begin{remark}
\label{rmk:graphon.stab}
By \eqref{eq:gst-opt}, \eqref{eyep.assume} and \eqref{approx:gST} we have
$\int\eye_p\circ g_{S,T} - \eye_p\circ g \ls_{\uF,\uup}  (\eta^{c_0 \wedge 1} + o(1))p^\Delta\log(1/p)$,
so that \Cref{prop:graphon.stab} is 
a stability-type strengthening of Theorem \ref{thm:BhDe2}.
Indeed, setting $E=\{g_{S,T}=1\}$, since 
$\eye_p(p) \le \eye_p(\cdot)$
it suffices to show that for any $g \in \cW$,
\begin{equation*}
\int_E  (1-g)^2  \le \eps^2 p^\Delta \quad \Longrightarrow \quad
\frac{1}{I_p(1)} \int_E ( \eye_p(1) -\eye_p \circ g ) \ls_{\uF,\uup} (\eps + o(1))p^\Delta 
\end{equation*}
which, as $\eye_p(x)/\eye_p(1) \ge x - o(1)$ 
and 
$|E|=O_{\uF,\uup}(p^\Delta)$, 
follows by Cauchy--Schwarz.
\end{remark}

We prove \Cref{prop:graphon.stab} in Subsection \ref{sec:graphon.stab.pf} after gathering some lemmas in Subsections \ref{sec:graphon.dominant}--\ref{sec:graphon.prelim}. Subsection \ref{sec:graphon.dominant} extracts a key estimate from highly nontrivial arguments in \cite{BGLZ}, while the lemmas in Subsection \ref{sec:graphon.prelim} are of an elementary nature. 

\subsection{Dominant terms in the expansion for $t(F_k,g/p)$}
\label{sec:graphon.dominant}


As we noted in \Cref{rmk:graphon.stab}, \Cref{prop:graphon.stab} is a stability-type strengthening of the solution to the graphon variational problem established in \Cref{thm:BhDe2}, and our proof of the former relies heavily on some highly nontrivial results established in \cite{BGLZ} towards the latter. The next lemma distills a key estimate that follows from the arguments in \cite{BGLZ}.

Letting $f=g-p$, for a graph $H$ we can expand
\begin{equation}	\label{tp.full}
%
t(H,g/p) = 1+\sum_{F\subseteq H}N(F,H) t(F,f/p)
\end{equation}
where the sum runs over nonempty subgraphs 
$F$ of $H$ (up to isomorphism), and $N(F,H)$ is the number of subgraphs of $H$ isomorphic to $F$.
It is shown in \cite[Corollary 6.2]{BGLZ} that for $g\ge p$ satisfying
\begin{equation}	\label{Jp.bound}
\eye_p(g) \le K p^\Delta \log(1/p)
\end{equation}
for some $K=O(1)$,
the only non-negligible terms in \eqref{tp.full} are for 
$F=H$, as well as 
$F=H^{A}$ for some ${A}\in \cI(H^\star)$
(recall our notation from \eqref{def:indep.poly}--\eqref{def:TF}), where $H^{A}$ denotes the 
bipartite subgraph of $H$ induced between ${A}$ and its vertex neighborhood $\cN_{H}({A})$ 
in $H$.
The expansion is further refined based on a decomposition of $f$ that we now recall.
For 
$d>0$ we denote
\begin{equation}\label{dfn:Dd}
D_{d}(g):= \bigg\{ x\in [0,1]: \int_0^1 \max(g(x,y)-p,0) dy \ge d\bigg\}\,.
\end{equation}
(
Note this differs slightly from the definition in \Cref{sec:ideas} and \cite{BGLZ}: the integrand is $f$ if $g\ge p$, but we do not assume this in general.)
We abbreviate
\begin{equation}\label{dfn:tf}
\hf:= f\chi_{D_d(g)^c\times D_d(g)^c} \,, \qquad 
\tf:= f\chi_{D_d(g)\times D_d(g)^c}
\,.
\end{equation}
Note that $\tf$ is an asymmetric graphon.
We denote by
\begin{equation}	\label{def:chf}
\chf(x,y) := \tf(x,y) + \tf(y,x)
\end{equation}
the symmetrization of $\tf$.
Recalling the notation \eqref{def:t.asymmetric},
 we note that if $F$ is bipartite with vertex bipartition $(A,B)$, then
\begin{equation}	\label{tF.chf}
t(F,\chf) = t(F, \tf; A) + t(F, \tf; B).
\end{equation}

Our next lemma combines \cite[Cor.\ 6.2 and Prop.\ 6.5]{BGLZ}, making 
explicit certain quantitative bounds which are extracted from their proofs.

\begin{lemma}
\label{lem:terms}
Let $p\in(0,1)$ and let $F_1,\dots, F_m$ be connected graphs of maximum degree $\Delta\ge2$. 
For any graphon $g=p+f$ with $f\ge0$ satisfying \eqref{Jp.bound} and any 
$0<\eps<\frac12$ there 
exists $\kappa_0(\eps)=\kappa_0(\uF,K,\eps)
>0
$ and $d=d(\uF,K,p,g,\eps)\in [p^{1/3},p^{\kappa_0(\eps)}]$ such that for each $1\le k\le m$ we have
\begin{align}	
t(F_k,g/p) &\le 1+ \eps + p^{\kappa_0(\eps)}+ \sum_{\emptyset\ne {A}\in \cI(F_k^\star)} t(F_k^{A}, \tf/p \,;A)
+\ind(\text{$F_k$ regular}) \,t(F_k, \hf/p)  \label{dom.terms}
\end{align}
(recall the notation \eqref{def:t.asymmetric}).
\end{lemma}

\begin{remark} 
The proof gives the dependence $\kappa_0(\uF,K,\eps)= \exp(-O_{K}(\eps^{-O_{\uF}(1)}))$. By
optimizing 
$\eps$ we can replace $\eps+p^{\kappa_0(\eps)}$ in \eqref{dom.terms} with 
$(\log\log\frac1p)^{-c}$ for a sufficiently small $c=c(\uF,K)>0$.
As can be seen from the proof, this can be improved to $O_{\uF,K}(p^{c'})$ for a reasonable constant $c'=c'(\uF)>0$ when none of the $F_k$ are regular and bipartite.
\end{remark}

\begin{remark}
The differences between \Cref{lem:terms} and the results of \cite{BGLZ} are that the latter are stated with qualitative errors and only address the case of a single graph ($m=1$). For us the crucial point to verify is that the parameters $\kappa_0$ and $d$ can be chosen uniformly for a collection of graphs $F_1,\dots, F_m$.
We do this in the proof below, by outlining the arguments in \cite{BGLZ}, pointing out the nature of the parameter dependencies and making explicit certain error terms. We encourage the interested reader to look at the full argument in \cite{BGLZ}, which involves many beautiful ideas that are skipped over in the summary that follows; we emphasize that the proof below involves no new ideas over the arguments in \cite{BGLZ}.
\end{remark}

\begin{proof}
Fix $g$ and $\eps$ as in the statement of the lemma.
Since $g\le 1$ we have the trivial bound $t(F_k, g/p)\le (1/p)^{|\Edges(F_k)|}$ so we may assume $p$ is bounded away from 1. Then by replacing $\eps$ with $\eps/C$ for a large constant $C=C(\uF,K)$ and shrinking $\kappa_0$ it suffices to prove \eqref{dom.terms} with error $O_{\uF,K}(\eps+p^{\kappa_0(\eps)})$ in place of $\eps + p^{\kappa_0(\eps)}$.

We consider first the problem of bounding $t(H,g/p)$ for an arbitrary connected graph $H$ of maximum degree $\Delta\ge2$, and will later show how $\kappa_0$ and $d$ can be chosen uniformly for $H$ ranging over a finite collection $\{F_1,\dots,F_m\}$. 
Fixing such a graph $H$, 
let $\cF_H$ be the class of (isomorphism classes of) graphs $F$ such that $H$ contains a subgraph isomorphic to $F$, and for which $\tau(F)=|\Edges(F)|/\Delta$, where $\tau(F)$ is the size of a minimal vertex cover for $F$ (i.e.\ a minimal subset $A\subseteq\Verts(F)$ such that every edge of $F$ contains some element of $A$).
Then from \cite[Cor.\ 6.2]{BGLZ} and its proof we have the following refinement of \eqref{tp.full}:
\begin{align}
\label{tp.1st}
t(H,g/p) = 1 &+  t(H,f/p)\ind(\text{$H$ regular, non-bipartite}) \\
&+\sum_{F\in\cF_{H}} N(F,H) t(F,f/p) + O_{H,K}(p^{\kappa_1}), \notag
\end{align}
for some $\kappa_1=\kappa_1(H)>0$ depending only on $H$. (Note that when $H$ is regular and bipartite then $t(H,f/p)$ arises as a term in the sum over $F$.) The error term in \eqref{tp.1st} is only stated as $o(1)$ in \cite[Cor.\ 6.2]{BGLZ}, but is easily seen to be of the above form from inspection of the short argument deducing \cite[Cor.\ 6.2]{BGLZ} from \cite[Lem.\ 6.1]{BGLZ}.

As noted above \cite[Cor.\ 6.2]{BGLZ}, each element $F\in\cF_H$ is bipartite with a vertex bipartition $(A,\Verts(F)\setminus A)$ such that $A$ is a minimal vertex cover of $F$ (thus $|A|=\tau(F)$) and every element of $A$ has degree $\Delta$ in $F$. Since $H$ has maximum degree $\Delta$ one sees that each $F\in\cF_H$ in fact has a unique minimal vertex cover $A$ except when $H$ is $\Delta$-regular and bipartite and $F=H$, in which case the two sides of $H$ are the two minimal vertex covers. 
Moreover, every element of $\cF_H$ arises as $H^A$ for an independent set $A$ in $H^\star$, and each $A\in\cI(H^\star)$ gives rise to an element of $\cF_H$ in this way. 
(Recall that $H^A$ is bipartite subgraph of $H$ induced between ${A}$ and its vertex neighborhood, and $H^\star$ is the induced subgraph of $H$ on its vertices of degree $\Delta$.)
We hence see that the sum over $F$ in \eqref{tp.1st} may be re-expressed as
\begin{align}
\label{tp.2nd}
t(H,g/p) = 1&+   t(H,f/p)\ind(\text{$H$ regular}) \\
&+\sum_{A\in\cI'(H^\star)}  t(H^A,f/p) + O_{H,K}(p^{\kappa_1}),\notag
\end{align}
where $\cI'(H^\star):= \{H^A: \emptyset\ne A\in \cI(H^\star), |A|<\tfrac12|\Verts(H)|\}$.
Indeed, $\{H^A: A\in \cI'(H^\star)\}$ is exactly $\cF_H$, except when $H$ is bipartite and regular, in which case it misses only the element $H\in\cF_H$. 

Consider an arbitrary fixed $A\in\cI'(H^\star)$ and let $H^{A_1},\dots, H^{A_k}$ be the connected components of $H^A$ (so $A_1,\dots, A_k$ is a partition of $A$), so that
\begin{equation}
\label{tF.factorize}
t(H^A, f/p) = \prod_{i=1}^k t(H^{A_i}, f/p). 
\end{equation}
Note that all of the $H^{A_i}$ are necessarily irregular; indeed, if one of them were regular then it would be a connected component of $H$, which we assumed is connected, and hence $H^A=H$, but then we would have $|A|=\frac12|\Verts(H)|$, contradicting that $A\in\cI'(H^\star)$. Since the $H^{A_i}$ are irregular, bipartite and connected, from \cite[Prop.\ 6.5(a)]{BGLZ} we have that for any $p^{1/3}\le d\ll 1$, 
\[
t(H^{A_i}, f/p) = t(H^{A_i}, \tf/p ; A_i) + O_{H,K}(p^{1/3}+ d^{1/6})
\]
for each $i$, where the error term is easily verified by inspection of the proof in \cite{BGLZ}. 
(Note that here and below we compute $t(H^{A_i},\tf/p; A_i)$ as in \eqref{def:t.asymmetric} under the ordered bipartition $(A_i,\cN_H(A_i))$, and similarly with $A$ in place of $A_i$.) 
Furthermore, from \Cref{thm:finner} and \cite[Lemma 4.2]{BGLZ} we have $t(H^{A_i},\tf/p;A_i)\le t(H^{A_i}, f/p) =O_{H,K}(1)$. Combined with the factorization \eqref{tF.factorize} we get
\begin{equation}	\label{tF.FA}
t(H^A, f/p) = t(H^A,\tf/p;A) + O_{H,K}(p^{1/3}+ d^{1/6})\;\; \forall A\in \cI'(H^\star),\, 
p^{1/3}\le d\ll 1.
\end{equation}
In case of $H$ regular and non-bipartite, from \cite[Prop.\ 6.5(b)]{BGLZ} there exists $\kappa_2=\kappa_2(H)>0$ such that 
\begin{equation}	\label{tF.reg-nonbip}
t(H, f/p) = t(H, \hf/p) + O_{H,K}( p^{\kappa_2}) \qquad\forall\;p^{\kappa_2}\le d\le 1,
\end{equation}
where the error term is readily verified from inspection of the proof. 
Finally, when $H$ is regular and bipartite, 
the proof of \cite[Prop.\ 6.5(c)]{BGLZ} shows that for any $\alpha\in(0,\frac13]$ and any integer $L\ge3$ there exists $d\in [p^{\alpha},p^{\alpha \eps'}]$ for $\eps'= \exp( - O_{K,L}(\eps^{-L}))$ such that (recalling the notation \eqref{def:chf}),
\begin{equation}	\label{tF.reg-bip}
t(H,f/p) = t(H, \hf/p) + t(H, \chf/p) + O_{H,K}(\eps + d^{1/2}),
\end{equation}
for any graph $H$ with $|\Verts(H)|\le L$ (for the proof in \cite{BGLZ} the point is that $L$ bounds the length of any cycle in a covering of $\Verts(H)$ by disjoint cycles). 

Now given $F_1,\dots, F_m$, take $L=L(\uF)$ and $\alpha=\alpha(\uF)$ with
\[
L(\uF):=\max_k|\Verts(F_k)|\,,\qquad 
\alpha(\uF):=\min\big\{ \tfrac13,\,\min_k \kappa_2(F_k)\big\}
\]
with $\kappa_2$ as in \eqref{tF.reg-nonbip}.
Then there exists $d\in [p^{\alpha}, p^{\alpha\eps'}]$ with $\eps' = \exp( - O_K(\eps^{-O_{\uF}(1)}))$ such that the estimates \eqref{tF.FA}, \eqref{tF.reg-nonbip}, \eqref{tF.reg-bip} all hold for $H=F_k$ for each $1\le k\le m$. 
Substituting these estimates into \eqref{tp.2nd} we get that for each $1\le k\le m$,
\begin{align}
t(F_k,g/p) = 1&+   t(F_k,\hf/p)\ind(\text{$F_k$ regular}) 
+ t(F_k, \chf/p)\ind(\text{$F_k$ regular, bipartite})\notag \\
&+\sum_{A\in\cI'(F_k^\star)}  t(F_k^A,\tf/p\,;A) + O_{\uF,K}(\eps+ p^{\min\{ \kappa_1\,,\,\alpha\eps'/6\}}).\label{tp.3rd}
\end{align}
Finally, we note that by the identity \eqref{tF.chf}, when $F_k$ is regular and bipartite the term  $t(F_k, \chf/p)$ can be absorbed into the sum over $A$ by extending the sum to all nonempty $A\in\cI(F_k^\star)$ (with $F_k^\star=F_k$ in this case).
\end{proof}

\subsection{Further preliminary lemmas}
\label{sec:graphon.prelim}

In addition to \Cref{lem:terms}, for the proof of Proposition \ref{prop:graphon.stab} we 
need
a few elementary lemmas.
The first is a stability result for the set of optimizers of the 2-dimensional problem \eqref{def:phi}.
See \Cref{fig:ab2} for an illustration.
Let
\begin{equation}	\label{def:Qpad}
\Region(\uup,\eta) = \big\{ (a,b) \in \R_{\ge 0}^2: \frac12a+b \le 
\UTab_{\uF}(\uup)+\eta
\big\}
\cap
V_\uF(\uup-\eta \1)\,,\quad \uup\in \R_+^m\,,\;\eta>0.
\end{equation}
We write $B_q(r)$ for the open ball in $\R^2$ of radius $r$ centered at $q$, and use sumset notation $R+S=\{r+s:r\in R, s\in S\}$.
In the sequel we abbreviate $T_k:=T_{F_k}$.

\begin{lemma}[Stability for the planar problem]
\label{lem:planar.stab}
For each $\uup\in\R_+^m$ and $\eta>0$ we have  that 
\[
\Opt(\phi;\uup) \subset \Region(\uup,\eta) \subset 
\Opt(\phi;\uup)+B_{(0,0)} (\eps_\eta) \,,
\]
for some $\eps_\eta = O_{\uF,\uup}(\eta)$.
\end{lemma}

\begin{proof}  With $\eta \mapsto \Region(\uup,\eta)$ non-decreasing and 
$\Region(\uup,0)=\Opt(\phi;\uup)$, the first containment is obvious. 
Further, by the compactness of $\Region(\uup,\eta)$ and continuity of all 
functions of  $(\csize,\bsize)$ in its definition,
any collection $q_\eta \in  \Region(\uup,\eta)$ must have a limit point $q_0 \in \Region(\uup,0)$,
implying the second containment for some $\eps_\eta \to 0$. With $\Opt(\phi;\uup)$ a finite set
(see Proposition \ref{prop:opt}(b)), it remains only to show that for 
fixed $q=(\csize,\bsize) \in \Opt(\phi;\uup)$ and small $\eps>0$ the set
$\Region(\uup,\eta) \cap B_{q}(\eps)$ has diameter $O_{\uF,\uup}(\eta)$. To this end, fixing 
such $q$, it is argued in the proof of Proposition \ref{prop:opt}(b) (see Appendix \ref{sec:opt}) that $q$ must be a point
of non-smoothness on the boundary of $\bigcap_k\{T_k \ge 1+ s_k\}$ where
the linear function $T_0(\csize,\bsize):=\frac{1}{2} \csize + \bsize$ of slope $m_0=-\frac{1}{2}$
is minimized (incorporating hereafter the constraint of being in $\R^2_{\ge 0}$ via
$s_{m+1}=s_{m+2}=-1$, $T_{m+1}=\csize$ and $T_{m+2}=\bsize$). As such,
at least two constraints, $k_L$ and $k_R$ in $[m+2]$, must hold with equality 
at $q$, where the corresponding curves intersect transversely with slopes  
$-\infty\le m_{k_L} <m_0 <m_{k_R} \le 0$ (and the strict inequalities here 
are due  to the strict convexity of $T_k, k \le m$). Setting 
\[
S(\eta,\eps) := B_q(\eps) \cap \{T_{k_L}\ge T_{k_L}(q) -\eta\}\cap \{T_{k_R}\ge T_{k_R}(q) -\eta\} 
\cap \{T_0 \le T_0(q) +\eta\} \,,
\]
clearly $\Region(\uup,\eta) \cap B_{q}(\eps)  \subseteq S(\eta,\eps)$. Further,
when $k_L \in [m]$ or $k_R  \in  [m]$, applying the mean-value theorem for 
the corresponding function of smooth gradient of norm $\gs_{\uF,\uup,q} 1$,
yields that $S(\eta,\eps) \subset \wt S(2\eta,\eps)$ for all $\eps \le \eps_0(\uF,\uup,q)$, 
where $\wt S(\eta,\eps)$ is defined as $S(\eta,\eps)$ except for replacing  
$T_{k_L}$ and $T_{k_R}$ by the corresponding linearizations around $q$ of slopes
\[
m_{k_L} \vee O_{\uF,\uup,q}(\eps^{-1})+O_{\uF,\uup,q}(\eps)  < m_0  < m_{k_R} - O_{\uF,\uup,q}(\eps) \,.
\]
In particular, $\wt S (2\eta,\eps)$ is contained within a closed triangle, 
whose interior point $q$ is of distance $O_{\uF,\uup,q}(\eta)$ from all 
three sides, thereby having a diameter $O_{\uF,\uup,q}(\eta)$, as claimed. 
\end{proof}

Next, we need a stability version of a quadratic approximation used in \cite{LuZh14,BGLZ} for the function $\eye_p:[0,1]\to\R_{\ge 0}$.
Recall the function $\jay_p$ from \eqref{def:jay}.

\begin{lemma}[Estimates on $\jay_p$]
\label{lem:Jp}
For any $p\in (0,1)$ and $x\in [-p,0]$,
\begin{equation}	\label{Jp.LB-left}
\jay_p(x) \ge \frac{x^2}{2p\log(1/p)}.
\end{equation}
Moreover,
there exists a constant $c>0$ such that for any $0<p\le c$ and $x\in [0,1-p]$,  
\begin{equation}	\label{Jp.LB-right}
\jay_p(x) - x^2 \gs \min( x^2, (1-p-x)^2). 
\end{equation}
\end{lemma}

The key point is that we only have $\jay_p(x)\approx x^2$ near 0 and $1-p$, and \eqref{Jp.LB-right} can hence be viewed as a stability version of the inequality $\jay_p(x)\ge x^2$ used in \cite{LuZh14,BGLZ}. This will allow us to deduce that near-optimizers of the graphon problem \eqref{def:probUTg} are well approximated by functions taking values in $\{p,1\}$. 


\begin{proof}
For \eqref{Jp.LB-left}, letting $L_p(x) = \eye_p(x) - (x-p)^2/2p$, we have $L_p(p)=L_p'(p)=0$ and $L_p''(x) = \frac1x+\frac1{1-x}-\frac1p>0$ for $x\in[0,p]$. Thus, $L_p>0$ on $[0,p]$, which yields \eqref{Jp.LB-left}.

Turning the bound \eqref{Jp.LB-right}, let $c_0\in (0,1/2)$ be a constant to be taken sufficiently small.
From \cite[Lemma 4.3]{BGLZ} we have that for $p\ll x\le 1-p$, 
\[
\jay_p(x) \sim x\frac{\log(x/p)}{\log(1/p)}.
\]
In particular,  for $x\in [c_0,1-c_0]$,
\begin{equation}	\label{Jp.macro}
\jay_p(x) \ge x-o(1)
\end{equation}
which easily yields the claim in this case.
From \cite[Lemma 4.4]{BGLZ} we have 
\[
\jay_p(x) \ge (x/x_0)^2 \jay_p(x_0)
\]
for any $0\le x\le x_0\le 1/2$, provided $p$ is at most a sufficiently small constant. 
Applying this with $x_0=c_0$, combined with \eqref{Jp.macro} at $x=c_0$, yields $\jay_p(x)\gs x^2/c_0$, which gives the claim for the range $x\in(0,c_0]$ assuming $c_0$ is sufficiently small. 

Now for the range $y:=1-p-x\in [0,c_0]$, setting $K_p(y):=\jay_p(1-p-y)-(1-p-y)^2$, it suffices to show
\begin{equation}	\label{Kp.goal}
K_p(y) \gs y^2
\end{equation}
for $0\le y\le c_0$. Since
\[
\jay_p(1-p-y) = 1-y + \frac1{\eye_p(y)} \bigg( \eye(1-y) + y\log\frac1{1-p}\bigg)
\ge 1-y -\frac{1}{\eye_p(1)}\bigg( y\log\frac1y + \log\frac1{1-y}\bigg)
\]
we have
\begin{equation}	\label{Kp.LB1}
K_p(y) \ge 2p+y - \frac{y}{\eye_p(1)}\bigg( \log\frac1y+O(1)\bigg) - O((p+y)^2).
\end{equation}
For $p^{3/4}\le y\le c_0$ this yields $K_p(y)\gs y$, giving \eqref{Kp.goal}.

Now set $y=tp$ for $t\le p^{-1/4}$. 
From \eqref{Kp.LB1} we get
\[
K_p(y) \ge (2-o(1))p + \frac{y}{\eye_p(1)} (\log t - O(1))
\]
giving $K_p(y)\gs y/\eye_p(1)$ for $t\ge C$ for a sufficiently large constant $C>0$. Since $\eye_p(1) \le (1/p)^{1/10}\le y^{-4/30}$, say, this yields \eqref{Kp.goal} for $Cp\le y\le p^{3/4}$. If $1\le t\ls 1$ then the \abbr{RHS} above is at least $(2-o(1))p \gs y$, so we have established \eqref{Kp.goal} for the range $p\le y\le c_0$.

For $y\le p$, i.e.\ $t\le 1$, we have
\[
K_p(y) \ge (2-o(1)) p - y \frac{\log (1/t)}{\eye_p(1)}.
\]
For $p^2\le y\le p$ the \abbr{RHS} above is at least $(1-o(1))p\gs y$, whereas for $y\le p^2$ we have $\log(1/t)\le \log(1/y)\le (1/y)^{1/10}$, say, giving a lower bound of $(2-o(1))p-y^{9/10}=(2-o(1))p \gs y^{1/2}$.
\end{proof}

Finally, we need the following elementary fact, which we state for general product probability spaces (still abbreviating $\|\cdot\|_q:=\|\cdot\|_{L_q(\mu^{\otimes 2})}$ as we do for the Lebesgue spaces).

\begin{lemma}	
\label{lem:product.set}
Let $(\Omega, \cF,\mu)$ be a probability space and let
$E$ be a subset of 
$\Omega^2$, measurable under the product $\sigma$-algebra, that is  asymmetric under interchanging of the coordinates.
Suppose there is a measurable function $h:[0,1]\to \R_{\ge0}$ such that $\|\chi_E-h\otimes h\|_q <\epsilon\mu(E)^{1/q}$ for some $q\ge1$.
Then there is a measurable set $S\subset[0,1]$ such that $\|\chi_E-\chi_{S \times S}\|_q\ls_q \eps^{1/2}
\mu(E)^{1/q}$. 
\end{lemma}

\begin{proof}
We may assume $\eps$ is sufficiently small depending on $q$ (for otherwise we take $S=\emptyset$).
For $k\in \Z$ let $h_k=h\chi_{S_k}$ with $S_k=\{ 2^{k-1}\le h<2^k\}$. 
We take $S:=S_{-1}\cup S_0\cup S_1$. Our aim is to show
\begin{equation}	\label{prod.goal1}
\mu(E\Delta S^2) \ls_q \eps^{q/2} \mu(E)
\end{equation}
First, by our assumption\,,
\[
\eps^q\mu(E)> \|\chi_E-h\otimes h\|_q^q\ge \int_{S^2\setminus E} h^q\otimes h^q d\mu^{\otimes 2}\gs_q \mu(S^2\setminus E)\,.
\]
It remains to show that $\mu(E\setminus S^2)\ls_q \eps^{q/2 }\mu(E)$. 
On any $S_k\times S_\ell$ with $k+\ell>2$ we have $h_k\otimes h_\ell -\chi_E \ge 1$, and hence $h_k\otimes h_\ell-\chi_E\in [2^{k+\ell-3},2^{k+\ell})$. Thus
\begin{align*}
\eps^q \mu(E) &> \sum_{k+\ell>2} \int_{S_k\times S_\ell} (h\otimes h-\chi_E)^q d\mu^{\otimes 2}\\
&\qquad\qquad\ge \sum_{k+\ell>2} 2^{(k+\ell-3)q} \mu(S_k)\mu(S_\ell)
\ge
\mu\bigg( E\cap \bigcup_{k+\ell>2} S_k\times S_\ell\bigg).
\end{align*}
From the second inequality above we moreover have
\begin{equation}	\label{Sk2+}
\mu(S_k) \ls_q \eps^{q/2} 2^{-kq} \mu(E)^{1/2}
\end{equation}
for all $k\ge 2$.
Now for $k+\ell<0$, $\chi_E-h\otimes h\in [\frac12,1]$ on $E\cap (S_k\times S_\ell)$, so 
\[
\eps^q \mu(E)> \sum_{k+\ell<0} \int_{E\cap (S_k\times S_\ell)} (\chi_E-h\otimes h)^q  d\mu^{\otimes 2}\gs_q \mu\bigg( E\cap \bigcup_{k+\ell<0} S_k\times S_\ell\bigg).
\]
By symmetry, to establish \eqref{prod.goal1} it now suffices to show
\begin{equation}	\label{prod.goal2}
\sum_{k\ge 2, k+\ell=i} \mu(S_k\times S_\ell) \ls_q \eps^{q/2} \mu(E)
\end{equation}
for each $i\in \{0,1,2\}$. 
Fixing such an $i$, from \eqref{Sk2+} we have that 
\[
\sum_{k\ge 2, k+\ell=i} \mu(S_k\times S_\ell)
\ls_q \eps^{q/2} \mu(E)^{1/2} \sum_{k\ge 2} 2^{-kq}\mu(S_{i-k}). 
\]
Now since 
\[
\mu(E)^{1/2}\asymp_q \|h\|_q^q \asymp_q \sum_{k\in\Z} 2^{kq}\mu(S_k) \asymp_q \sum_{k\in \Z} 2^{-kq} \mu(S_{i-k})
\]
for any $i=O(1)$
we obtain \eqref{prod.goal2} and hence the claim.
\end{proof}

\subsection{Proof of Proposition \ref{prop:graphon.stab}}
\label{sec:graphon.stab.pf}

Let $\uF,\uup,\eta$ and $g$ be as in the statement of the proposition.
From the trivial bound \eqref{gST.triv} we may assume $\eta$ is sufficiently small depending on $\uF$ and $\uup$; we will also assume $p_0(\uF,\uup,\eta)$ is sufficiently small without comment.
Setting $f=g-p$, we have
\begin{equation}	\label{Jp.bounds}
\frac12\int \jay_p\circ f \le 
(\UTab_\uF(\uup)+\eta) p^\Delta \ls_{\uF,\uup} p^\Delta.
\end{equation}
In particular \eqref{Jp.bound} holds with $K=O_{\uF,\uup}(1)$. 
Letting $d\in(0,1)$ to be chosen below depending on $g,p,\eta,\uF$ and $\uup$, with $\hf,\tf$ as in \eqref{dfn:tf} we set
\begin{equation}	\label{def:agbg}
a_g:= p^{-\Delta} \int \jay_p\circ \hf\,,\qquad b_g:= p^{-\Delta} \int \jay_p\circ \tf.
\end{equation}
Since $\hf$ and $\tf$ have disjoint supports, from \eqref{Jp.bounds} we have
\begin{equation}	\label{linear-constraint}
\frac12a_g+ b_g \le \UTab_\uF(\uup)+\eta.
\end{equation}

For brevity, we encapsulate the two cases that $g\in\cW_\cP$ or $g\in \cW$ is a general graphon by using the following convention: by \emph{measurable} we will mean Lebesgue measurable in the general case, whereas for the case that $g\in\cW_\cP$ we take ``measurable'' to mean measurable under the finite $\sigma$-algebra generated by $\cP$, or the product $\sigma$-algebra generated by $\cP\times \cP$ as the case may be. Thus, our goal is to locate measurable sets $S,T$ of appropriate size such that $g$ is well approximated by $g_{S,T}$.

The remainder of the proof is divided into the following five steps.
Steps 1--4 establish the proposition under the extra assumption that 
\begin{equation}	\label{assu.nonnegative}
g\ge p\qquad a.e.
\end{equation}
i.e.\ $f\ge0$ $a.e.$

\begin{itemize}
\item
{\bf Step 1:}
Show that $\tmax_k(a_g,b_g) \ge 1+\up_k - O(\eta)$ for each $k$, which, together with \eqref{linear-constraint}, gives that $(a_g,b_g)$ is an approximate extremizer of \eqref{def:phi}
(this step 
is a restatement of the arguments from \cite{BGLZ,BhDe}).

\item
{\bf Step 2:}
Deduce that the bounds 
in Step 1 
 actually hold with approximate equality.

\item
{\bf Step 3:}
Using the approximate equality in the applications of Finner's inequality (in \eqref{tp.finner1} and \eqref{tp.finner2}), apply Theorem \ref{thm:finner.stab} to deduce that $\hf$ and $\tf$ are well approximated by tensor products of univariate functions.

\item
{\bf Step 4:}
From approximate equality in the passage from the relative entropy functional $\jay_p$ to $L_\Delta$-norms (see \eqref{applyJp}), together with Lemmas \ref{lem:Jp} and \ref{lem:product.set}, deduce that $\hf$ and $\tf$ are well approximated by indicators of product sets ($\chi_{S\times S}$ and $\chi_{T\times T^c}$, respectively), concluding the proof under the additional assumption \eqref{assu.nonnegative}.

\item
{\bf Step 5:}
Remove the assumption \eqref{assu.nonnegative}.\\
\end{itemize}

\noindent{\bf Step 1.}
Let $k\in [m]$ be arbitrary.
Applying Lemma \ref{lem:terms} with $\eps=\eta/2$ we have that
for 
all $p$ 
sufficiently small depending on $\eta$, 
\begin{align}
1+\up_k-\eta
&\le 1+\eta + \sum_{\substack{\emptyset\ne {A}\in \cI(F_k^\star) \\ |{A}|<\verts(F_k)/2}} t(F_k^{A},\tf/p \,;A)	\notag\\
&\qquad\qquad+ \ind(F_k\text{ regular}) \cdot \bigg( t(F_k,\hf/p) + \ind(F_k\text{ bipartite}) \cdot t(F_k,
\chf/p) \bigg)\,.	\label{terms}
\end{align}
Next, from Theorem \ref{thm:finner} we have
\begin{equation}	\label{tp.finner1}
t(F_k,\hf/p) \le \|\hf/p\|_\Delta^{\edges(F_k)}\,,\qquad 
t(F_k,\chf/p)  \le \|\chf/p\|_\Delta^{\edges(F_k)}\,,\qquad 
\end{equation}
and
\begin{equation}	\label{tp.finner2}
t(F_k^{A},\tf/p\,;A) \le \|\tf/p\|_\Delta^{\Delta|{A}|}\qquad \forall\;{A}\in\cI(F_k^\star).
\end{equation}
Substituting these bounds in \eqref{terms}, we get
\begin{align}
1+\up_k 
&\le 2 \eta + P_{F_k^\star}(\|\tf/p\|_\Delta^\Delta) + \ind( F_k \text{ regular}) \|\hf/p\|_\Delta^{\edges(F_k)}	\notag\\
&= 2 \eta+ \tmax_k(\|\hf/p\|_\Delta^\Delta, \|\tf/p\|_\Delta^\Delta) \,,\label{finner}
\end{align}
where the second line follows from the definition of $\tmax_k$ and noting that $\edges(F_k)=\Delta\verts(F_k)/2$ when $F_k$ is regular. 
Finally, from Lemma \ref{lem:Jp},
\begin{equation}	\label{applyJp}
\|\hf/p\|_\Delta^\Delta \le p^{-\Delta}\|\hf\|_2^2\le a_g
\qquad \text{and}\qquad
\|\tf/p\|_\Delta^\Delta \le p^{-\Delta}\|\tf\|_2^2\le b_g
\end{equation}
(using only that the \abbr{RHS} of \eqref{Jp.LB-right} is nonnegative).
Combining with \eqref{finner} and by the monotonicity of $\tmax_k$ we get
\begin{equation}	\label{Tkab.LB}
\tmax_k(a_g,b_g) \ge 1+\up_k -2\eta\,.
\end{equation}
In the notation of \eqref{def:Qpad}, the bounds \eqref{linear-constraint} and \eqref{Tkab.LB} say that
\begin{equation}
(a_g,b_g)\in \Region(\uup,2\eta)\,.
\end{equation}
From Lemma \ref{lem:planar.stab}, assuming $\eta\le \eta_0(\uup)$ we have
\begin{equation}	\label{agbg.localized}
(a_g,b_g)\in B_{\qs}\big(
O_{\uF,\uup}(\eta)
\big)
\end{equation}
for some $\qs=(\csizes,\bsizes)\in \Opt(\phi;\uup)$, and moreover that $(a_g,b_g)$ is separated by 
distance $\gs_{\uF,\uup}1$ from all other
elements of the finite set $\Opt(\phi;\uup)$ (see Proposition \ref{prop:opt}(b)).
From \eqref{linear-constraint}, 
\eqref{finner}--\eqref{applyJp}
and the monotonicity of $\tmax_k$ we similarly conclude that
\begin{equation}	\label{Lr.localized}
p^{-\Delta}(\|\hf\|_r^r, \|\tf\|_r^r) \in B_{\qs}\big(O_{\uF,\uup}(\eta)\big)\,, \qquad r=2,\,\Delta\,.
\end{equation}
\,\\
\noindent{\bf Step 2.}
Since $\qs\in\Opt(\phi;\uup)$ it follows that $T_k(\qs)=1+\up_k$ for at least one value of $k\in [m]$. Let $\Ks\subset[m]$ denote the set of such $k$.
Since $T_k$ is locally Lipschitz
it follows that for $k\in \Ks$, 
\begin{equation}	\label{Tkab.UB}
T_k(a_g,b_g) \le 1+\up_k + O_{\uF,\uup}(\eta)\,.
\end{equation}
Hence, for $k\in \Ks$, all of the bounds \eqref{terms}--\eqref{applyJp} hold with equality up to an additive error of 
$O_{\uF,\uup}(\eta)$.
In particular,
\begin{equation}	\label{finner2.sharp}
\|\tf/p\|_\Delta^{\Delta|{A}|}-t(F_k^{A}, \tf/p\,;A) 
\ls_{\uF,\uup} \eta
\qquad \forall\, k\in \Ks\quad\forall\;{A}\in\cI(F_k^\star)
\end{equation}
and
\begin{equation}	\label{bg.tight}
b_gp^\Delta - \|\tf\|_2^2
\ls_{\uF,\uup}\eta p^\Delta\,,
\end{equation}
and if $F_k$ is regular for some $k\in \Ks$, then
\begin{equation}	\label{finner1.sharp}
\|\hf/p\|_\Delta^{\edges(F_k)} -t(F_k,\hf/p) 
\ls_{\uF,\uup}\eta
\end{equation}
and
\begin{equation}	\label{ag.tight}
a_gp^\Delta- \|\hf\|_2^2 
\ls_{\uF,\uup} \eta
p^\Delta.
\end{equation}
If $F_k$ is irregular for every $k\in \Ks$, it follows that $\csizes=0$ and we have from \eqref{agbg.localized}, \eqref{applyJp} that
\begin{equation}	\label{ag.small}
\|\hf\|_\Delta^\Delta\le\|\hf\|_2^2\le a_gp^\Delta =O_{\uF,\uup}(\eta p^\Delta).
\end{equation}
In the remainder of the proof we combine the above estimates with lemmas from Section \ref{sec:graphon.prelim} and Theorem \ref{thm:finner.stab} to locate the measurable sets $S$ and $T$. \\

\noindent{\bf Step 3.}
Using Theorem \ref{thm:finner.stab} we obtain the following consequence of \eqref{finner2.sharp} and \eqref{finner1.sharp}.
Here and in the remainder of this section, for functions $f_1,f_2:[0,1]^d\to\R_{\ge0}$ ($d=1$ or $2$), we write 
\[
f_1\approx_r f_2
\]
to mean $\|f_1-f_2\|_r^r\ls_{\uF,\uup}\eta^{c} p^\Delta$ for some $c=c(\uF)>0$ depending only on $\uF$.

\begin{claim}
\label{claim:tensor}
There exist measurable $\wh h,\wt h:[0,1]\to \R_{\ge0}$ supported on $D_d(g)^c$ and $D_d(g)$, respectively, with $\|\wh h\|_\Delta = \|\wt h\|_\Delta=1$, such that
\begin{equation}	\label{hf.approx}
\hf^\Delta \approx_1 \|\hf\|_\Delta^\Delta (\wh h\otimes \wh h)^\Delta
\end{equation}
and 
\begin{equation}	\label{tf.approx}
\tf^\Delta \approx_1 \|\tf\|_\Delta^\Delta (\wt h^\Delta\otimes \chi_{[0,1]})\,.
\end{equation}
\end{claim}

\begin{proof}
Since $\up_k>0$ for all $k$ we have that $\UTab_\uF(\uup)>0$ and so $(\csizes,\bsizes)\ne (0,0)$. 
If $\csizes=0$ then \eqref{hf.approx} holds trivially by \eqref{ag.small} and the triangle inequality (for arbitrary $\wh h$). The same reasoning gives $\|\tf\|_\Delta^\Delta =O_{\uF,\uup}(\eta p^\Delta)$
and hence \eqref{tf.approx} in the case that $\bsizes=0$.

If $\bsizes>0 $, then from \eqref{Lr.localized} it follows that for any $k\in \Ks$ we have (assuming $\eta$ is sufficiently small) that $\|\tf\|_\Delta \gs_{\uF,\uup} p$.
Then 
for any nonempty ${A}\in \cI(F_k^\star)$,
from dividing through by $\|\tf/p\|_\Delta^{\Delta|{A}|}\gs_{\uF,\uup}1$ in \eqref{finner2.sharp} 
we get
\[
t(F_k^{A}, \tf/\|\tf\|_\Delta\,;A) \ge 1 - 
O_{\uF,\uup}(\eta )\,.
\]
Taking ${A}$ to be any singleton $\{u\}$, we apply Theorem \ref{thm:finner.stab} with $V=\Verts(F_k^{\{u\}})$, $\cA=\Edges(F_k^{\{u\}})$, $f_{\{u,v\}}(x_u,x_v)= \tf(x_u,x_v)/\|\tf\|_\Delta$ for each $\{u,v\}\in \cA$, and weights $\lam_A\equiv 1/\Delta$, to obtain $\wt h:[0,1]\to \R_{\ge0}$ supported on $D_d(g)$ with $\|\wt h\|_\Delta=1$, such that \eqref{tf.approx} holds.
Here we have used Remark \ref{rmk:hB1} and the fact that in $F^{\{v\}}_k$, the sum of weights on any vertex other than $v$ is $1/\Delta<1$, to take the second factor of the tensor product to be $\chi_{[0,1]}$. 

If $\csizes>0$ then we must have that $F_k$ is regular for some $k\in \Ks$, and by similar lines as above we obtain $\wh h:[0,1]\to \R_{\ge0}$ supported on $D_d(g)^c$ such that \eqref{hf.approx} holds.
(Theorem \ref{thm:finner.stab} initially provides an approximation with some $\wh h_1\otimes \wh h_2$ in place of $\wh h\otimes \wh h$, but from the symmetry of $\hf$ and the triangle inequality it quickly follows that $\wh h_1$ and $\wh h_2$ are themselves close in $L_1$, so that we may take a single function $\wh h$.)
\end{proof}

\noindent{\bf Step 4.}
We now combine Claim \ref{claim:tensor} with \eqref{bg.tight}, \eqref{ag.tight}, \eqref{ag.small} to deduce the following, which immediately yields the claimed approximation \eqref{approx:gST} and concludes the proof of Proposition \ref{prop:graphon.stab} under the added assumption \eqref{assu.nonnegative}. 

\begin{claim}
\label{claim:indicator}
There are measurable sets $T\subset D_d(g)$ and $S\subset D_d(g)^c$ such that
\begin{equation}	\label{indic.goal}
\tf \approx_2 (1-p) \chi_{T\times[0,1]} 
\qquad \text{and}\qquad
\hf \approx_2 (1-p)\chi_{S\times S} \,.
\end{equation}
If $F_k$ is irregular for each $k\in \Ks$ then we can take $S=\emptyset$. 
\end{claim}

Indeed, 
we get the claimed approximation \eqref{approx:gST} by taking $S,T$ as in the above claim, noting that they have the claimed measure by \eqref{Lr.localized}
(note also that modifications of $g$ on $T\times T$ have negligible impact since $|T|^2=O_{\uF,\uup}(p^{2\Delta})$).

\begin{proof} 
We begin with the first approximation in \eqref{indic.goal}.
Considering 
$k \in \Ks$, we get from \eqref{bg.tight} and a straightforward application of Lemma \ref{lem:Jp}
the existence of a measurable set 
$\wt E\subset D_d(g)\times D_d(g)^c$ 
 such that 
\begin{equation}	\label{fromJpsharp}
\tf \approx_2 (1-p) \chi_{\wt E} \,.
\end{equation}
Next, note that 
\begin{align*}
\| \tf^\Delta - (1-p)^\Delta\chi_{\wt E} \|_2^2
&= \int_{\wt E^c} \tf^\Delta + \int_{\wt E} \big( (1-p)^\Delta- \tf^\Delta \big)^2\\
&\le \int_{\wt E^c} \tf^2 + \Delta^2 \int_{\wt E} \big( 1-p-\tf\big)^2 
\ls_\Delta \| \tf - (1-p)\chi_{\wt E}\|_2^2 \,,
\end{align*}
so from \eqref{fromJpsharp} we have 
\begin{equation}	\label{tfDelta.approx}
\tf^\Delta\approx_2 (1-p)^\Delta \chi_{\wt E}\,.
\end{equation}
Denote 
\[
\tf_1:[0,1]\to [0,1-p]\,,\qquad \tf_1(x) = \| \tf(x,\cdot)\|_\Delta,
\]
which is supported on $D_d(g)$. 
From \eqref{tf.approx} it easily follows that
\begin{equation*}
\tf^\Delta \approx_1 \tf_1^\Delta \otimes \chi_{[0,1]}.
\end{equation*}
Since both sides are bounded by 1 we have $\tf^\Delta\approx_2\tf_1^\Delta\otimes\chi_{[0,1]}$. 
Together with \eqref{tfDelta.approx} and the triangle inequality we get that
\begin{equation}	\label{tf.above}
\chi_{\wt E} \approx_2 h_1\otimes \chi_{[0,1]}
\end{equation}
where $h_1:= \tf_1^\Delta/(1-p)^\Delta$. Now, by
\eqref{fromJpsharp} and the triangle inequality, it suffices
to show that $\chi_{\wt E} \approx_2 \chi_{T\times [0,1]}$
for $T= \{ x\in [0,1]: |\wt E_x|>1/2\}$, 
where $\wt E_x:=\{y\in[0,1]:(x,y)\in \wt E\}$.
By Fubini's theorem and \eqref{tf.above}, this in turn will follow from showing
\begin{equation}	\label{claim.tf1}
h_1\approx_2 \chi_T. 
\end{equation}
To show this, note that
\begin{align*}
\|\chi_{\wt E} - h_1\otimes \chi_{[0,1]}\|_2^2
&= \int_{\wt E^c} h_1^2\otimes \chi_{[0,1]} + \int_{\wt E} (1- h_1\otimes \chi_{[0,1]})^2\\
&= \int_0^1(1-|\wt E_x|) h_1(x)^2  + \int_0^1 |\wt E_x| (1-h_1(x))^2 \\
&\ge \frac12 \int_{T^c} h_1(x)^2 + \frac12\int_{T} (1-h_1(x))^2 
= \frac12 \|h_1-\chi_T\|_2^2 \,.
\end{align*}
Combining the above with \eqref{tf.above} we obtain \eqref{claim.tf1}, hence 
that $\tf\approx_2(1-p)\chi_{T\times [0,1]}$, as stated.

Turning to the second approximation in \eqref{indic.goal}, in view of \eqref{ag.small} 
we may assume henceforth the existence of some regular $F_k$, $k\in \Ks$. In that case,
we get from \eqref{ag.tight} and 
Lemma \ref{lem:Jp} that there exists a measurable set $\wh E\subset D_d(g)^c\times D_d(g)^c$ such that
\[
\hf \approx_2 (1-p)\chi_{\wh E}  \,.
\]
Consequently, by the triangle inequality we only need to show $\chi_{\wh E} \approx_2 \chi_{S\times S}$
for some $S \subseteq D(g)^c$. Reasoning as we did for $\wt E$,
we deduce from \eqref{hf.approx} in Claim \ref{claim:tensor} that
\[
\chi_{\wh{E}} \approx_2 h\otimes h
\]
where $h=\hf_1^\Delta /((1-p)\|\hf\|_\Delta)^{\Delta/2}$, with $\hf_1(x) :=\|\hf(x,\cdot)\|_\Delta$. 
Finally,  our claim that $\chi_{\wh E} \approx_2 \chi_{S\times S}$
 now follows from Lemma \ref{lem:product.set}.
\end{proof}

\noindent{\bf Step 5.}
It only remains to remove the extra assumption \eqref{assu.nonnegative}.
For general $g$ as in the statement of the proposition, let $f=g-p$, $g_+=\max(g,p)$ and $f_+=g_+-p=\max(g-p,0)$. Note that $D_d(g)=D_d(g_+)$, so that $\tf$ and $\tf_+$ are both supported on $D_d(g)^c\times D_d(g)^c$, and $\hf,\hf_+$ are supported on $D_d(g)\times D_d(g)^c$.
By the proof for the case $g\ge p$ we have $(a_{g_+},b_{g_+})\in B_{\qs}( 
O_{\uF,\uup}(\eta))$ for some $\qs=(\csizes,\bsizes)\in \Opt(\phi;\uup)$ as in \eqref{agbg.localized}, and that \eqref{approx:gST} holds with $g_+$ in place of $g$.
By the monotonicity of $\eye_p$ on $[0,p]$ we have
\begin{equation*}
a_{g_+}\le a_g\,,\qquad b_{g_+}\le b_g
\end{equation*}
and
\begin{equation*}
1+\up_k -\eta \le t(F_k,g/p)\le t(F_k,g_+/p) \le \tmax_k(a_{g_+}, b_{g_+}) + O_{\uF,\uup}(\eta)
\end{equation*}
for every $k\in [m]$. From this it follows that $(a_{g},b_{g})\in B_{\qs}( O_{\uF,\uup}(\eta))$.
Hence,
\begin{equation}	\label{agbg.squeeze}
0\le a_g-a_{g_+} 
\ls_{\uF,\uup} \eta
\,,\qquad 0\le b_g-b_{g_+} 
\ls_{\uF,\uup} \eta.
\end{equation}
Applying \eqref{Jp.LB-left}, we get that 
\[
\|g-g_+\|_2^2 \ls_{\uF,\uup} \eta
p^{\Delta+1}\eye_p(1) 
\]
and \eqref{approx:gST} holds for $g$ by the same bound for $g_+$ and the triangle inequality.
\qed

\section{The conditional structure of \ER graphs}
\label{sec:struct.ER}

In this section we prove Theorem \ref{thm:struct.ER} 
on the conditional structure of \ER graphs on joint upper-tail events for homomorphism densities, with \Cref{prop:graphon.stab} being a key input.

\subsection{Quantitative LDPs and counting lemmas}
\label{sec:LDPs}

Recall that an $\R$-weighted graph is a symmetric function $X:[n]^2\to \R$ that is zero on the diagonal, and that $\cG_n$ and $\cQ_n$ denote the sets of $\{0,1\}$-weighted graphs and $[0,1]$-weighted graphs, respectively.
Recall also from \eqref{def:ergm} that $\nu_{n,p}^0$ is the Erd\H{o}s--R\'enyi($p$) measure on $\cG_n$.
In this section we write $\bG$ for a sample $\bG_{n,p}$ from $\nu_{n,p}^0$. 

\begin{defn}[Upper-\abbr{LDP}]
\label{def:upperLDP}
Given a collection $\mathbb{K}=(\cK_G)_{G\in\cG_n}$ of closed convex sets $\cK_G\subseteq\cQ_n$ with $G\in\cK_G$ for each $G$, 
for a set $\cE\subseteq\cG_n$ we define the $\mathbb{K}$-neighborhood of $\cE$ as
\[
(\cE)_{\mathbb{K}}:=\bigcup_{G\in \cE} \cK_G \subseteq\cQ_n.
\]
(Note that while $\cE$ is a subset of the discrete cube $\cG_n$, its $\mathbb{K}$-neighborhood is a subset of the solid cube $\cQ_n$.)
For a lower semi-continuous function $J:\cQ_n\to [0,\infty]$ and $\rate_\star,\rate_{\ME}>0$, we say that a probability measure $\mu$ on $\cG_n$ satisfies a quantitative upper-\abbr{LDP} under $\mathbb{K}$, with \emph{rate function} $J$, \emph{cutoff rate} $\rate_\star$ and \emph{metric entropy rate} $\rate_{\ME}$, if for any $\cE\subseteq\cG_n$,
\begin{equation}
\log\mu(\cE) \le  - \min\Big( \rate_\star\,,\; \inf\{J(Q): Q\in (\cE)_{\mathbb{K}}\} - \rate_{\ME} \Big) \,.
\end{equation}
\end{defn}

Our quantitative large deviation are formulated in terms of two different norms on the space of $\R$-weighted graphs.
The first is the spectral norm \eqref{def:specnorm},
for which we have the following result from \cite{CoDe}.

\begin{prop}[Large deviations: spectral norm] \quad\\
\label{prop:LD.spec}
\begin{enumerate}
\item[(a)] (Upper-\abbr{ldp}). Let $K_0,K_1\ge1$ and $\delta\in(n^{-100},1]$, and assume
\begin{equation}	\label{LD.spec.pbound}
np^\Delta \ge \frac{K_0\log n}{\delta^2}. 
\end{equation}
For each $G\in\cG_n$  let $\cK_G^{2\to2}(\delta)=\{Q\in\cQ_n: \|Q-G\|_{2\to2}\le \delta np^{\Delta/2}\}$. 
Then $\nu_{n,p}^0$ satisfies an upper-\abbr{LDP} under $\mathbb{K}_{2\to2}(\delta)=(\cK_G^{2\to2}(\delta))_{G\in\cG_n}$, with rate function $\eye_p$, and
\begin{equation}	\label{rate.bounds}
\rate_\star\gs K_1\rate_{n,p} \,,\qquad \rate_{\ME}\ls \frac{K_1}{K_0}\rate_{n,p}\,.
\end{equation}

\item[(b)] (Counting lemma). Let $p\in(0,1)$, $L\ge1$ and $\eps>0$ be arbitrary.
Suppose $\cK\subseteq \cQ_n$ is a convex set of diameter at most $\eps np^{\Delta_\star}$ in the spectral norm, and that for every induced strict subgraph $F'\prec F$ there exists $Q\in \cK$ such that
\begin{equation}
t(F',Q/p) \le L.
\end{equation}
Then for every $F'\preccurlyeq F$ and $Q_1,Q_2\in\cK$,
\begin{equation}
|t(F',Q_1/p) - t(F',Q_2/p)| \ls_F L\eps.
\end{equation}
\end{enumerate}
\end{prop}

\begin{proof}
Part (b) follows from \cite[Prop.\ 3.7]{CoDe}.
For part (a) we apply \cite[Prop.\ 3.4]{CoDe} with $K=K_1$, $\delta_o=c\delta p^{\Delta/2}$ for a sufficiently small constant $c>0$, and taking there $k= K_1 c^{-2}\delta^{-2}\log(1/p)$. With $N$ the cardinality of the relevant net in \cite[Prop.\ 3.4]{CoDe}, we have as in \cite[(3.11)]{CoDe}
that $\log N \ls K_1 \delta^{-2} n \log n \log(1/p) = O(K_1K_0^{-1} \rate_{n,p})$ 
(with the latter equality due to our assumption \eqref{LD.spec.pbound}), yielding, as claimed, the 
\abbr{rhs} of \eqref{rate.bounds}.
\end{proof}

Our second quantitative large deviation result is stated in terms of
a variant of the cut norm introduced in \cite{CDP} (a special case of a 
general class of norms for the study of $p$-sparse tensors). Specifically, for an $\R$-weighted graph $Z$ we define
\begin{equation}	\label{def:Bstar}
\|Z\|_\Base^\star = \sup_{T\in \cT} \frac{|\langle T,Z\rangle|}{\|T\|_\Base}\,,
\end{equation}
where $\cT=\{\1_{S_1}\otimes \1_{S_2}:S_1,S_2\subseteq [n] \}$ (a special instance of the notion of \emph{test tensor} from \cite{CDP}), and
\begin{equation}
\|T\|_\Base:= (|S_1|\vee np^{\Delta-1})(|S_2|\vee np^{\Delta-1})\,,\qquad T=\1_{S_1}\otimes\1_{S_2}\,, \quad S_1,S_2\subseteq[n]\,.
\end{equation}
We  have following results from \cite{CDP}.

\begin{prop}[Large deviations: $\Base^\star$-norm]\label{prop:LD.Bstar}\;$~$\\
\begin{enumerate}
\item[(a)] (Upper-\abbr{ldp}). For a sufficiently large absolute constant $C_0>0$, let $K_0\ge C_0$, $K_1\ge1$, and $\delta\in(0,1]$, and assume 
\begin{equation}\label{K0-cond}
np^{\Delta+1} \ge \frac{K_0\log n}{\delta^2(1\vee \log(1/p))}.
\end{equation}
For each $G\in\cG_n$ let $\cK_G^{\star}(\delta)$ be the convex hull of $\{G'\in\cG_n:\|G-G'\|_\Base^\star \le\delta p\}$. 
Then $\nu_{n,p}^0$ satisfies an upper-\abbr{LDP} with respect to $\mathbb{K}_{\star}(\delta)=(\cK_G^{\star}(\delta))_{G\in\cG_n}$ with rate function $\eye_p$ and $\rate_\star,\rate_{\ME}$ satisfying \eqref{rate.bounds}.
\item[(b)] (Counting lemma). Let $p\in(0,1)$, $L\ge1$ and $\eps>0$ be arbitrary.
Suppose $\cE\subseteq \cG_n$ has diameter at most $\eps p$ in the $\Base^\star$-norm, and that there exists $G_0\in\cE$ such that
\begin{equation}\label{LAo-cond}
t(F',G_0/p) \le L
\end{equation}
for every proper subgraph $F'\subsetneq F$. 
Then for every $Q_1,Q_2$ in the convex hull of $\cE$,
\begin{equation}\label{eq1:count}
|t(F,Q_1/p) - t(F,Q_2/p)| \ls_F L\eps.
\end{equation}
\end{enumerate}
\end{prop}

\begin{proof}
Part (a) is the specialization of \cite[Theorem 3.1(a)]{CDP} for the case of graphs (2-uniform hypergraphs) with the 
$\Base^\star$-norm \eqref{def:Bstar} and growth parameter
$w_{n,p}(\Base)= np^{\Delta+1}$, see \cite[Exmp. 2.6 \& (2.16)]{CDP}), and taking $\kappa=K_1p^\Delta$.
Part (b) is \cite[Theorem 2.10]{CDP}.
\end{proof}
\subsection{The continuity of graph collections in terms of planted clique-hub sizes}

We proceed to establish the asymptotic (as $n \to \infty$), continuity with respect to the
parameters $(\csize,\bsize)$, of the sets $\cG_n^1(\csize,\bsize,\xi)$ of \eqref{def:Gn'}
and $\cG_n^2(\csize,\bsize,\xi)$ of \eqref{def:Gn''}.
\begin{lemma}\label{lem:contain}
For $u=1,2$, any $\xi>0$ and $K$ finite
there exist $\eps'>0$ and $n_0 < \infty$, so that 
\begin{equation}\label{eq:cont-Ani}
\bigcup_{| \csize-\csize'|  \le \eps', |\bsize-\bsize'| \le \eps'}
\cG_n^u(\csize,\bsize,\xi) \subset \cG_n^u(\csize',\bsize',2\xi) \,, \qquad 
\forall n \ge n_0, \,  \|(\csize',\bsize')\|_\infty \le K \,.
\end{equation}
\end{lemma}
\begin{proof}
For both $u=1$ and $u=2$, we can and shall assume \abbr{wlog} that
$I=[1,|I|]$ and $J=[n-|J|,n]$ are disjoint intervals, setting the corresponding object 
to have $I'=[1,|I'|]$ and $J'=[n-|J'|,n]$. 

In case $u=2$, our claim \eqref{eq:cont-Ani} 
follows in view of \eqref{def:Gn''} and the triangle inequality, from 
\begin{equation}\label{eq:cont-Yn}
\sup_{
\|(\csize-\csize',\bsize-\bsize')\|_\infty 
\le \eps'}\; \sup_{Q \in \cQ_n(\csize,\bsize)} \; 
\inf_{Q' \in \cQ_n(\csize',\bsize')} \; \|Q-Q'\|_{2 \to 2} < \xi n p^{\Delta/2} \,.
\end{equation}
Indeed, with $|Q_{ij}-Q'_{ij}| \le 1$ for all $ij$, $Q=Q^{I,J}$ and $Q'=Q^{I',J'}$, clearly $\|Q-Q'\|_{2 \to 2} \le 
\|Z\|_{2 \to 2}$, where $Z_{ij} = 1_{Q_{ij} \ne Q'_{ij}}$. It is easy to check 
that $Z$ consists here of a reversed-$L$ shape of dimensions $\max(|I|,|I'|) \times | \,|I|-|I'| \, |$
and a disjoint cross-with-hole shape of dimensions $|\, |J| - |J'| \, | \times n$. With the $2 \to 2$ 
operator norm of each part of $Z$ thus bounded by the geometric mean of its two dimensions, we arrive at
\[
\|Q-Q'\|_{2 \to 2} \le (\sqrt{|\csize-\csize'|} + \sqrt{|\bsize-\bsize'|} ) n p^{\Delta/2} \le
2 \sqrt{\eps'} n p^{\Delta/2} \,, 
\]
from which we get \eqref{eq:cont-Yn} when  
$\eps' < \xi^2/4$.

When $u=1$ it suffices for \eqref{eq:cont-Ani} to show that 
$\cG_n^{I,J}(\xi) \subset \cG_n^{I',J'}(2\xi)$. That is, having at least 
$|I|^2/2-\xi n^2p^\Delta$ edges in $G[I]$ and at least $|J|(n-|J|) - \xi n^2p^\Delta$ edges
in $G[J,J^c]$ yields at least 
$|I'|^2/2-2 \xi n^2p^\Delta$ edges in $G[I']$ and at least $|J'|(n-|J'|) - 2 \xi n^2p^\Delta$ edges
in $G[J',J'^c]$. This holds universally for any graph $G$, provided that 
\begin{align*}
|I|^2 - (|I|^2-|I'|^2)_+ &\ge |I'|^2 - \xi n^2 p^\Delta \,. \\
|J| (n-|J|) - (|J|-|J'|)_+(n-|J|) - (|J'|-|J|)_+ |J| & \ge |J'| (n-|J'|) - \xi n^2 p^\Delta \,.
\end{align*}
With both $|J|, |J'|$ being $o(n)$ (since $p \ll 1$), we arrive at the simpler sufficient conditions 
\begin{align*}
\xi > (\csize'-\csize)_+, \qquad 
\xi > (\bsize'-\bsize)_+ \,.
\end{align*}
That is, any $\eps' < \xi$ will do to complete the proof.
\end{proof}

\subsection{Proof of Theorem \ref{thm:struct.ER}}  
First take $\uup=\underline{0}$.
Since $\Opt(\phi;\underline{0})=\{(0,0)\}$ and $\cG_n^1(0,0,\xi)=\cG_n$ for any $\xi>0$, 
part (a) then trivially holds. For part (b), note that $\cQ_n(0,0)=\{p\}$, where we abusively write $p$ for the element of $\cQ_n$ that is identically $p$ off the diagonal. Hence for any 
$\delta<\xi/2$,
\[
 ( \cG_n^2(0,0,\xi)^c)_{\mathbb{K}_{2\to2}(\delta)} \subset \big\{ Q \in \cQ_{n} : 
  \|Q-p\|_{2\to2} \ge \frac12 \xi n p^{\Delta/2} \big\} \,.
\]
For any element $Q$ of the \abbr{rhs}  
we have
\begin{equation}\label{eq:quad-diff}
\xi^2n^2p^\Delta \ls \|Q-p\|_{2\to2}^2\le \|Q-p\|_{\HS}^2 = 2\sum_{i<j} (Q_{ij}-p)^2
\end{equation}
where we write $\|X\|_\HS= (\sum_{i,j=1}^n X_{i,j}^2)^{1/2}$ for the Hilbert--Schmidt norm on $\R$-weighted graphs (viewed as kernels of operators on $\ell^2([n])$). 
Further, the \abbr{rhs} of \eqref{eq:quad-diff} is $\ls \eye_p(Q) /\log(1/p)$ 
(as $J_p(x) \ge x^2$, see Lemma \ref{lem:Jp} or \cite{LuZh14}).
Consequently,
\begin{equation}\label{eq:0-rate-op}
\eye_p\Big( ( \cG_n^2(0,0,\xi)^c)_{\mathbb{K}_{2\to2}(\delta)} \Big) \gs  \xi^2 \rate_{n,p} \;.
\end{equation}
Since $n p^{2 \Delta_\star} \gg \log n$, for any fixed $\eps_1\in (0,1)$ we have from
Proposition \ref{prop:LD.spec}(a) applied with $\delta:=\eps_1p^{\Delta_\star-\Delta/2}$, $K_1$ a sufficiently large constant, and $K_0$ sufficiently large depending only on $K_1$
and $\eps_1$, 
that for all $n$ sufficiently large,
\begin{equation}\label{eq:bd-cA-one}
\log\P(\bG\in  \cG_n^2(0,0,\xi)^c) \le -\min \Big\{  \rate_{n,p}\,, \;\eye_p\Big( \big( \cG_n^2(0,0,\xi)^c\big)_{\mathbb{K}_{2\to2}(\delta)}\Big) 
- \eps_1 \rate_{n,p} \Big\} \,.
\end{equation} 
Taking $\eps_1=c\xi^2$ for a sufficiently small constant $c>0$, from \eqref{eq:0-rate-op} and \eqref{eq:bd-cA-one} we deduce that 
\[
\log\P(\bG\in  \cG_n^2(0,0,\xi)^c) \le -\min(1, c'\xi^2)\ratenp
\]
for all $n$ sufficiently large and some absolute constant $c'>0$. 
Since $\P(\bG\in \cU_p(\uF,0))\gs_{\uF} 1$, the claim follows.

Now take
$\uup\ne0$.
We may assume \abbr{wlog} that $s_k>0$ for each $k$. Indeed, otherwise consider
the restriction $(\uFs,\uups)$ of $\uF,\uup$ to the indices with $s_k>0$, denoting by $\uF_0$
the complementary part of $\uF$. Then, setting
\begin{equation}	\label{def:Eu}
\cE_u:=\bigcup_{(\csize,\bsize)\in\Opt(\phi;\uup)}\cG_n^u(\csize,\bsize,\xi)\,,\qquad u=1,2
\end{equation}
we have that 
\[
\P(\bG\notin \cE_u |\bG\in \cU_p(\uF,\uup)) =\frac{\P(\bG\in \cE_u^c \cap \cU_p(\uF,\uup))}{\P(\bG \in  
\cU_p(\uF,\uup))}
\le \frac{\P(\bG\in \cE_u^c\cap\cU_p(\uFs,\uups))}
{\P(\bG\in \cU_p(\uFs,\uups))\P(\bG\in\cU_p(\uF_0,\underline{0}))}
\]
where in the numerator we used that $\cU_p(\uF,\uup)\subseteq\cU_p(\uFs,\uups)$, 
and in the denominator we applied the \abbr{FKG} inequality.
Clearly
$\P(\bG\in\cU_p(\uF_0,\underline{0}))\gs_{\uF} 1$, so as claimed, it suffices to fix
hereafter $\uF=\uFs$ and $\uup$ with $s_k>0$ for all $k$. 

Now, from \cite[Lemma 7.2]{CDP}
and the union bound, there exists $L=L(\uF,\uup)$ finite such that for all $n$ large enough,
\begin{equation}	\label{bd:goodL}
\P(\bG\in\cL_{\subseteq}(\uF,L)) \ge 1- e^{- (\UTab_\uF (\uup)  +1) r_{n,p}}  \,,
\end{equation}
where
\[
\cL_{\subseteq}(\uF,L) := \bigcap_{k\in[m]}\bigcap_{F\subseteq F_k}  \big\{Q\in\cQ_n: t(F,Q/p)\le L\big\}\,.
\]
Fixing $\xi >0$ and such $L$, for $u=1,2$ consider the sets
\begin{equation}
\cE_u'=\cE_u'(\uup,\xi) = \cG_n\cap \cL_{\subseteq}(\uF,L)\cap \cU_p(\uF,\uup) \cap \cE_u^c
\end{equation}
with $\cE_u$ as in \eqref{def:Eu}.
In view of \eqref{bd:goodL}, it suffices to show that for $u=1,2$, 
\begin{equation}	\label{structb.goal1}
\limsup_{n \to \infty} r_{n,p}^{-1} \log\P(\bG\in \cE_u') < - \UTab_\uF (\uup) \,.
\end{equation}

We first establish \eqref{structb.goal1} for the case $u=2$, which gives part (b) of the theorem. 
Setting $\delta:=\eps_1p^{\Delta_\star-\Delta/2}$ as before, the main step is to show
\begin{equation}\label{structb.goal2}
\eye_p\big( (\cE_2')_{\mathbb{K}_{2\to2}(\delta)} \big)
> (\phi_\uF(\uup) + \eta) \ratenp
\end{equation}
for $\eps_1,\eta>0$ sufficiently small depending on $\uF,\uup,\xi$. Indeed, granted \eqref{structb.goal2}, we have from 
Proposition \ref{prop:LD.spec}(a) applied with $K_1=C(1+\phi_\uF(\uup))$ 
and $K_0=CK_1/\eta^2$ for a sufficiently large constant $C<\infty$,
that for all $n$ sufficiently large,
\begin{align*}
\log\P(\bG\in  \cE_2') \le -\min \Big\{ (\phi_\uF(\uup)+1) \rate_{n,p}\,, \;\eye_p\big( (\cE_2')_{\mathbb{K}_{2\to2}(\delta)} \big)
- \tfrac12\eta \rate_{n,p} \Big\} 
\le -(\phi_\uF(\uup)+ \tfrac12\eta)\ratenp
\end{align*}
and \eqref{structb.goal1} follows. 

Turning to the proof of \eqref{structb.goal2}, let $\cP$ be the partition of $[0,1]$ into intervals of length $1/n$, such that $\cQ_n$ embeds in $\cW_\cP$ via 
\begin{equation}	\label{QtoW}
Q\mapsto g_Q\,,\qquad g_Q(x, y) = Q_{\lf xn\rf, \lf yn\rf}. 
\end{equation}
For $\eps>0$ 
let
\begin{equation}
\cQ_n'(\uup,\eps) = \bigcup_{(a,b)\in\Opt(\phi;\uup)} \bigcup_{\|(\csize',\bsize')-(\csize,\bsize)\| _\infty \le \eps} \cQ_n(a',b')
\end{equation}
and let $\cW'(\uup,\eps)\subset\cW_\cP$ denote the corresponding set of graphons via the identification \eqref{QtoW}.
Taking hereafter $\eps_1<\xi/2$, since $\Delta_\star \ge \Delta/2$, for any
$Q \in (\cE_2')_{\mathbb{K}_{2\to2}(\delta)}$ there exists $G \in \cE_2'$ with 
\[
\|Q-G\|_{2\to2}\le \eps_1 n p^{\Delta_\star} \le \frac{1}{2} \xi n p^{\Delta/2} \,.
\]
As $G$ does not lie in the set $\cE_u$ from \eqref{def:Eu},
it then follows by the triangle inequality for the spectral norm that
\begin{equation*}
\|Q-\CH^{I,J}\|_{2\to2} \ge \frac12\xi n p^{\Delta/2} \qquad \quad \forall \; \CH^{I,J}\in  \cQ_n'(\uup,0)
\end{equation*}
and consequently, that for any $Q \in  (\cE_2')_{\mathbb{K}_{2\to2}(\delta)}$
\begin{equation}\label{QCHdist}
\|g_Q-g_{\CH^{I,J}}\|_2 = \frac1n\|Q-\CH^{I,J}\|_{\HS} \ge \frac1n\|Q-\CH^{I,J}\|_{2\to2}
\ge \frac12 \xi p^{\Delta/2} \qquad \forall \CH^{I,J}\in \cQ_n'(\uup,0) \,.
\end{equation}
Further, 
as seen in the proof of \eqref{eq:cont-Yn} below, for $c=\frac{1}{32}$ and
any $g \in \cW'(\uup,c\xi^2)$ there exists $\CH^{I,J}\in\cQ_n'(\uup,0)$ such that
\[
\|g-g_{Q^{I,J}}\|_2
\le \frac1{4}\xi p^{\Delta/2} \,.
\]
Hence, by the triangle inequality, for any $Q$ satisfying \eqref{QCHdist},
\begin{equation}	\label{structb.goal3}
\|g_Q-g\|_2 \ge \frac14\xi p^{\Delta/2} \qquad \forall g \in \cW'(\uup,c\xi^2) \,.
\end{equation}
Next, for $Q \in (\cE_2')_{\mathbb{K}_{2\to2}(\delta)}$ and the corresponding $G \in \cE_2'$
we have from Proposition \ref{prop:LD.spec}(b), that for some finite $C=C(\uF,\uup) \ge 2$
\begin{equation}\label{eq:bd-tkQ}
t(F_k, Q/p) \ge t(F_k, G/p) - O_{\uF}(L\eps_1)  \ge 1+\up_k - C \eps_1 \,, \qquad \forall k \in [m]\,.
\end{equation}
Choosing $\eps_1 < \eta/C$ and $\eta>0$ small enough so that 
$\eta^{c_0} < 
c_1 \min(c \xi^2, \xi/4)$,  
for the positive $c_0=c_0(\uF)$ of Proposition \ref{prop:graphon.stab}
and 
some $c_1=c_1(\uF,\uup)>0$ sufficiently small,
it follows from  \eqref{structb.goal3}, \eqref{eq:bd-tkQ}
and Proposition \ref{prop:graphon.stab} (in the contrapositive), that
\[
\frac1{n^2}\eye_p(Q) = 
\eye_p(g_Q) > ( \phi_\uF(\uup)+\eta) p^\Delta \log (1/p) 
\]
for all $Q \in  (\cE_2')_{\mathbb{K}_{2\to2}(\delta)}$. Having
thus obtained \eqref{structb.goal2} and thereby \eqref{structb.goal1} for $u=2$,
this concludes the proof of part (b) of Theorem \ref{thm:struct.ER}. 

Now turn to prove part (a), namely \eqref{structb.goal1} for $u=1$.
Since $\cG_n^1(a,b,\xi)=\cG_n$ if $\max(\frac{\csize}{2},\bsize)\le\xi$, 
we may assume that $\max(\frac{\csize}{2},\bsize)>\xi$ for every $(\csize,\bsize) \in \Opt(\phi;\uup)$.
The main step is to show
\begin{equation}\label{structa.goal2}
\eye_p\big( (\cE_1')_{\mathbb{K}_{\star}(\delta_1)} \big)
> (\phi_\uF(\uup) + \eta) \ratenp
\end{equation}
for some $\delta_1,\eta>0$ sufficiently small depending on $\uF,\uup,\xi$. 
Indeed, since $n p^{\Delta+1} \log (1/p) \gg \log n$, we can then apply
Proposition \ref{prop:LD.Bstar}(a) to conclude exactly as we did for part (b). 

Now for any $Q\in\cK_G^{\star}(\delta_1)$ and $G\in\cE_1'$ we get by Proposition \ref{prop:LD.Bstar}(b)
that \eqref{eq:bd-tkQ} holds with $\delta_1$ replacing $\eps_1$. Thus,
following the preceding proof of part (b), we will arrive at \eqref{structa.goal2} 
upon showing that analogously with \eqref{QCHdist} we have that 
\begin{equation}
\label{structa.goal3}
\|Q-\CH^{I,J}\|_\HS^2 \ge \frac{\xi}{2} n^2 p^\Delta \qquad  \forall\CH^{I,J}\in \cQ_n'(\uup,0) \,.
\end{equation}
Fixing $\CH^{I,J} \in \cQ_n'(\uup,0)$, note that 
$\|Q-G\|_\Base^\star\le \delta p$ for some $G \in \cE_1'$, so in particular 
\[
|\sum_{I\times I} (Q_{ij}-G_{ij}) | \le \delta p \max(|I|, np^{\Delta-1})^2\le (\csize \vee 1) \delta n^2p^{\Delta+1} 
\le \xi n^2 p^{\Delta} 
\]
for all $n$ large enough (as $\csize \le 2 \UTab_\uF (\uup)$, $\delta \le 1$ and $p \ll 1$).
As $G \notin \cG_n^{I,J}(\xi)$, it thus follows that  
\[
2 \xi n^2 p^\Delta \le \sum_{I\times I} (1-G_{ij}) \le \xi n^2p^{\Delta} +
\sum_{I\times I} (1-Q_{ij}) \,.
\]
Consequently, by Cauchy--Schwarz,
\[
(\xi n^2 p^{\Delta})^2 \le \big[\sum_{I\times I} (1-Q_{ij}) \big]^2 
\le |I|^2 \big[ \sum_{I\times I} (1-Q_{ij})^2\big]
 \le \csize n^2 p^{\Delta} \Big[ \sum_{I\times I} (1-Q_{ij})^2 \big] \,,
\]
yielding the lower bound 
\begin{equation}\label{eq:lbd-clique}
\|Q-\CH^{I,J}\|_\HS^2 \ge \sum_{I\times I} (1-Q_{ij})^2 \ge \frac{\xi^2}{\csize} n^2 p^{\Delta} \,.
\end{equation}
Arguing as above on $J\times[n]\setminus J$, we similarly get for any $\delta \le \xi/4$ and all $n$ large 
enough
\[
\big|\sum_{J\times [n]\setminus J} (Q_{ij}-G_{ij}) \big| 
\le \delta p n (|J|+np^{\Delta-1}) \le \frac{\xi}{2} n^2 p^\Delta
\]
(as $\bsize \le \UTab_\uF (\uup)$ and $p \ll 1$). Since
$G \notin \cG_n^{I,J}(\xi)$ it then follows that 
\[
\big(\frac{\xi}{2} n^2 p^{\Delta}\big)^2 \le \big[\sum_{J\times [n] \setminus J} (1-Q_{ij}) \big]^2 
\le n |J| \big[ \sum_{J\times [n] \setminus J} (1-Q_{ij})^2\big]
 \le \frac{\bsize}{2} n^2 p^{\Delta} \|Q-\CH^{I,J}\|_\HS^2 \,.
\]
Combining this with \eqref{eq:lbd-clique} we find that 
\[
\|Q-\CH^{I,J}\|_\HS^2 \ge \frac{\xi^2}{2\bsize \vee \csize} n^2 p^\Delta \ge 
\frac{\xi}{2} n^2 p^\Delta 
\]
(by our assumption that $\frac{\csize}{2} \vee \bsize > \xi$). With \eqref{structa.goal3} established,
this concludes the proof.
\qed

\section{Proofs of results for ERGMs}
\label{sec:ERGMs}

In this section we establish Proposition \ref{prop:nmf}, Theorem \ref{thm:nmf2} and Theorem \ref{thm:struct.ERGM}. For the last result we rely on Theorem \ref{thm:struct.ER}, which we already proved in Section \ref{sec:struct.ER}.

We may assume \abbr{wlog} that $h(\1)=0$. Hereafter, $r=r_{n,p}$ of \eqref{def:rate}. Recalling 
\eqref{def:TF},  as in Section \ref{sec:graphon} we abbreviate $T_k:= T_{F_k}$, $k \le m$, and
also re-index these $m$ graphs so that $F_k$ is regular if and only if $k\le m'$ for some $m'\in[m]$.

\subsection{Proof of Proposition \ref{prop:nmf}}

We begin with a lemma showing
that the limited growth of $h(\cdot)$ allows for truncating the tails of $\Ham(\Gnp/p)$.

\begin{lemma}
\label{lem:truncate.tail}
Assume $p \gg n^{-1/(\Delta+1)}$.
Then, for $\hamlet(\cdot)$ satisfying \eqref{assu-h.growth}, 
and $\Ham(\cdot)$ of \eqref{hamilton},
\begin{equation}\label{eq:exptight}
\limsup_{L \to \infty} \limsup_{n \to \infty} \rate^{-1} \log 
\E \Big [e^{\rate\Ham(\Gnp/p)} \sum_{k=1}^m \ind( t(F_k,\Gnp/p) \ge L)\Big] = -\infty \,.
\end{equation}
\end{lemma}

\begin{proof}  Since $t(F_k,\Gnp/p) \le p^{-\edges(F_k)}$ are uniformly bounded when
$p\gs1$ we may assume hereafter \abbr{wlog} that $p \ll 1$. Further, since $r=r_{n,p}\to\infty$ for
the range of $p(n)$ we consider, 
in view of \eqref{hamilton} and \eqref{assu-h.growth} it suffices for \eqref{eq:exptight}
to show that for some 
$\eta>0$ and any $k \le m$, 
\[
\limsup_{n \to \infty} \rate^{-1} \log \E \big[ \exp(\eta \rate t(F_k,\Gnp/p)^{\Delta/\edges(F_k)} ) \big]< \infty 
\]
(combine \cite[Lemma 4.3.8]{DeZe-book} with H\"older's inequality). This in turn follows from the 
uniform large deviations upper bound 
\begin{equation}\label{eq:unif-rough-tail}
\limsup_{u,n \to \infty} \frac{1}{u^{\Delta/\edges(F_k)} \rate_{n,p}} \log \P \big( t(F_k,\Gnp/p) \ge  1+ u\big) < 0 
\end{equation}
which we derive in Proposition \ref{prop:rough-ubd} for 
$1 \gg p \gg n^{-1/(\Delta+1)}$.
\end{proof}

We turn to the proof of Proposition \ref{prop:nmf}, starting with the lower bound $\logmgf\ge\nmf-o(\ratenp)$. 
To this end, fixing positive integers $K,\zeta^{-1}$,
we consider the finite $\zeta$-mesh 
$J_\zeta=(\{\zeta,2\zeta,\ldots,K\})^m$ of $[0,K]^m$. 
By the monotonicity of $\underline{x} \mapsto \hamlet(\underline{x})$, we have that for any 
$\uup \ge \zeta \1$ 
\begin{equation}\label{eq:pre-lbd}
\logmgf \ge \rate  \cdot \hamlet((1-\zeta)\1+\uup) +
\log \P (\Gnp \in \cU_p(\uF,\uup-\zeta \1) ) \,.
\end{equation}
Next, recall the entropic optimization problem for joint upper tails of homomorphism counts
\begin{align}\label{def:probUTnmf}
\UTnmf_{n,p}(\uF,\uup) &:= \eye_p\big( \cU_p(\uF,\uup) \big)\,,  
\end{align}
in terms of
$\cU_p(\uF,\uup)$ of \eqref{def:Up} (here we use the notation \eqref{def:IpE}). 
From \cite[Prop. 9.1]{CDP} we have that for all $n$ large enough,
\begin{equation}\label{eq:lbd}
\inf_{\uup' \in J_\zeta} \Big\{ \log \P (\Gnp \in \cU_p(\uF,\uup'-\zeta \1) ) + (1+\zeta) \UTnmf_{n,p}(\uF,\uup') \Big\} 
\ge - o(r_{n,p})\,.
\end{equation}
Since $\hamlet(\cdot)$ is uniformly continuous on $[0,K]^m$, for any $\eta>0$ and
$\zeta=\zeta(\eta)$ small enough,
\begin{equation}\label{eq:h-cont}
\hamlet(\1+\uup) \ge (1+\zeta) \hamlet ((1+2\zeta) \1 + \uup)  - \eta \,, \qquad \forall \uup \in [0,K]^m \,.
\end{equation}
Setting for $K$ finite, the non-negative
\begin{align}
\probFEnmf^{(K)}_{n,p} :=
\sup_{\uup \in \R_{\ge 0}^m, \|\uup\| \le K} 
\Big\{ \rate_{n,p} \hamlet(\1+\uup) - \UTnmf_{n,p}(\uF,\uup) \Big\}  \,,
\label{eq:Phi-PsiK}
\end{align}
we get by \eqref{eq:pre-lbd}--\eqref{eq:h-cont} and the monotonicity  
of $\UTnmf_{n,p}(\uF,\cdot)$ that for any $\eta>0$ and such $\zeta=\zeta(\eta)$,
\begin{align}\label{eq:sup-K}
\logmgf &\ge 
\sup_{\uup' \in J_\zeta, \up_k \in [\up'_k,K]} \{ \rate_{n,p} \hamlet((1-\zeta)\1+\uup') -
(1+\zeta) \UTnmf_{n,p}(\uF,\uup) \} - o(r_{n,p}) \nonumber \\
& \ge (1+\zeta) \sup_{\uup \in \R_{\ge 0}^m, \|\uup\|_\infty \le K} \{ \rate_{n,p}
\sup_{\uup' \in J_\zeta, \up'_k \le \up_k} \{ \hamlet((1+\zeta) \1+\uup') \} -
\UTnmf_{n,p}(\uF,\uup) \}	\nonumber\\
&\qquad -  \eta \, \rate_{n,p} - o(r_{n,p}) \nonumber \\
&\ge \probFEnmf^{(K)}_{n,p}-\eta \, \rate_{n,p} - o(r_{n,p}) \,.
\end{align}
Hereafter, let $(\uup)_+$ denote the projection of $\uup \in \R^m$ onto $\R_{\ge 0}^m$.
With $\hamlet(\cdot)$ non-decreasing in each argument and 
$\UTnmf_{n,p}(\uF,\uup)=\UTnmf_{n,p}(\uF,(\uup)_+)$, 
decomposing the supremum in \eqref{NMF} according to $\up_k:=t(F_k,Q/p)-1$, yields that 
\begin{equation}\label{eq:Phi-Psi}
\nmf = \sup_{\uup \in \R_{\ge 0}^m} 
\Big\{ \rate_{n,p} \hamlet(\1+\uup) - \UTnmf_{n,p}(\uF,\uup) 
\Big\} \,.
\end{equation}
By definition, see \eqref{def:probUTnmf},
\[
\UTnmf_{n,p}(\uF,\uup) \ge \UTnmf_{n,p}(F_k,\up_k) \,, \qquad \forall k \in [m] \,.
\]
Recalling \eqref{assu-h.growth}, we have from 
\cite[Lemma 7.2]{CDP}, that for some $\eta=\eta(\uF)>0$,
\begin{equation}\label{Phi-lbd}
\UTnmf_{n,p}(F_k,\up) \ge \eta \, m \, r_{n,p} \, (\up^{\Delta/\edges(F_k)} -1) \,, \qquad \forall \up \ge 0,   \, 
k \in [m] \,.
\end{equation}
Consequently, for any $\uup \in \R_{\ge 0}^m$, 
\[
\UTnmf_{n,p}(\uF,\uup) \ge \eta \, r_{n,p} \bigg( \sum_{k=1}^m  \up_k^{\Delta/\edges(F_k)} -m\bigg)\,.
\]
In view of the growth condition \eqref{assu-h.growth} for $h(\cdot)$, the 
above bound implies
that the supremum in \eqref{eq:Phi-Psi} is attained at $\uup_n$ uniformly bounded in $n$,
thereby matching (for some fixed $K$ and all $n$ large), the expression 
$\probFEnmf^{(K)}_{n,p} $ of \eqref{eq:Phi-PsiK}. Having $\eta$ in \eqref{eq:sup-K}  
arbitrarily small 
thus yields the lower bound $\logmgf\ge\nmf-o(\ratenp)$.

Turning to prove the upper bound $\logmgf\le \nmf+o(\ratenp)$, in view of Lemma \ref{lem:truncate.tail}
it suffices to show that for $r=r_{n,p}$, any $K \in \N$, $\eta>0$
and all $n$ sufficiently large (depending on $\uF$, $\ube$, $K$ and $\eta$),
\begin{equation}	\label{Z:goal1}
\log \E \big[ e^{\rate\Ham(\Gnp/p)} \ind(\max_{k \le m} \{t(F_k,\Gnp/p)\} \le K)
\big] \le \nmf + 2 \eta \rate \,.
\end{equation}
Fixing $K,\zeta^{-1} \in \N$ let $J^\star_\zeta$ denote the finite $\zeta$-mesh 
analogous to $J_\zeta$, except for including now also $-\zeta$ and $0$ as possible 
values for each coordinate $\up_k$. We apply the upper bound of \cite[Thm. 1.1]{CDP}
at any $\uup \in J^\star_\zeta$ for the restriction 
$(\uFs,\uups)$ of $(\uF,\uup)$ to those $k \in [m]$ where $\up_k > 0$. 
Doing so, the monotonicity of $\hamlet(\cdot)$
yields the following bound on the \abbr{lhs} of \eqref{Z:goal1}:
\begin{align}\label{eq:bd-mgf}
 \log \Big[ \sum_{\uup \in J^\star_\zeta} & e^{r \hamlet( (1+\zeta) \1+\uup)} 
\P (\Gnp \in \cU_p(\uFs,\uups) ) \Big]
\nonumber \\
&\le \log |J^\star_\zeta| + \max_{\uup \in J^\star_\zeta} \big\{ r \hamlet( (1+\zeta) \1+\uup) 
+ \log \P (\Gnp \in \cU_p(\uFs,\uups) ) \big\}
\nonumber \\
&\le o(r) + \max_{\uup \in J^\star_\zeta} \big\{ r \hamlet( (1+\zeta) \1+\uup) - (1+\zeta)^{-1} 
\UTnmf_{n,p}(\uFs,\uups-\zeta \1)
\big\}  \,.
\end{align}
For $\uup \in J^\star_\zeta$, the vector $\uups - \zeta \1$ consists
of all the non-negative coordinates of $\uup - \zeta \1$, from which it follows that
$\UTnmf_{n,p}(\uFs,\uups-\zeta \1)=\UTnmf_{n,p}(\uF,(\uup-\zeta \1)_+)$. 
Plugging this in \eqref{eq:bd-mgf}
and considering $\zeta$ small enough that \eqref{eq:h-cont} holds,
yields \eqref{Z:goal1} and thereby completes the proof.
\qed

\subsection{Proof of Theorem \ref{thm:nmf2}} 

From Proposition \ref{prop:nmf} it suffices to show that 
\begin{equation}
\frac1{\ratenp}\nmf\to \psi_{\uF,\hamlet}.
\end{equation}
Recall from \cite[Prop.\ 1.10]{BhDe} that for any $n^{-1/\Delta}\ll p\ll 1$,  
\begin{equation}\label{eq:probUTab}
\lim_{n \to \infty} \frac{\UTnmf_{n,p}(\uF,\uup) }{\rate_{n,p}} = \UTab_\uF (\uup) \,,
\qquad \forall \uup \in \R_{\ge 0}^m \,.
\end{equation}
Considering \eqref{eq:Phi-Psi}, \eqref{psi-phi} 
and \eqref{eq:probUTab}, it thus remains only to show that 
\begin{equation}\label{eq:Psi-psi}
\limsup_{n \to \infty} 
\sup_{\uup \in \R_{\ge 0}^m} \Big\{ \hamlet(\1+\uup) - \frac{\UTnmf_{n,p}(\uF,\uup)}{\rate_{n,p}}
\Big\} \le \psi_{\uF,\hamlet} = \sup_{\uup \in \R_{\ge 0}^m} 
\Big\{ \hamlet(\1+\uup) - \UTab_\uF (\uup) \big\} \,.
\end{equation}
In the course of proving Proposition \ref{prop:nmf} we saw that the supremum on the \abbr{lhs} 
is attained at some $\uup_n$ which are uniformly bounded in $n$. Hence, by the continuity of $h$ we
obtain \eqref{eq:Psi-psi} and thereby complete the proof, upon showing that for any sequence
$\uup_n \to \uup_\infty$
\[
\liminf_{n \to \infty} \Big\{ \frac{\UTnmf_{n,p}(\uF,\uup_n)}{\rate_{n,p}} \Big\} 
\ge \UTab_\uF (\uup_\infty) \,.
\]
Now, by the monotonicity of $\uup \mapsto \UTnmf_{n,p}(\uF,\uup)$ we can replace
$\uup_n$
with $(1-\eps) \uup_\infty$.
Having done so, we use \eqref{eq:probUTab} and conclude upon taking $\eps \to 0$
(relying on the continuity of $\UTab_\uF (\cdot)$).
\qed

\subsection{Proof of Theorem \ref{thm:struct.ERGM}}

We proceed to prove  Theorem \ref{thm:struct.ERGM} using Theorem \ref{thm:struct.ER}, by 
combining the argument used when proving Proposition \ref{prop:nmf} with the 
containment property of \Cref{lem:contain}. Specifically, by
Theorem \ref{thm:nmf2}, parts (a) and (b) of Theorem \ref{thm:struct.ERGM} amount to having for
$u=1,2$, respectively, and for any fixed $\xi>0$,
\begin{equation}\label{eq:cond-ERGM}
\limsup_{n \to \infty} \rate ^{-1} \log \E\Big[e^{\rate\Ham(\bG_{n,p}/p)} \ind(\Gamma^u_{n}(\xi)) \Big]
<  \psi_{\uF,\hamlet}  \,, 
\end{equation}
 where
\begin{equation}\label{def:Gamma-eps}
\Gamma^u_n(\xi) :=
\bigcap_{(\csize',\bsize')\in \Opt(\psi)} (\cG^u_n(\csize',\bsize',2\xi))^c \,, \quad u=1,2 \,.
\end{equation}
Proceeding as in the proof of \eqref{Z:goal1}, while intersecting all the events there
with $\Gamma^u_{n}(\xi)$, we find analogously to \eqref{eq:bd-mgf}, that for any fixed $\xi>0$ 
and $\eta>0$, the \abbr{LHS} of \eqref{eq:cond-ERGM} is at most 
\[
2 \eta + \max_{\uup\in J^\star_\zeta} \big\{ \hamlet(\1+(\uup)_+) - \UTab_{\uFs} (\uups,\xi) \big\}  \,,
\]
for some $\zeta=\zeta(\eta)>0$, where 
\begin{equation}\label{def:phi-xi}
\UTab_\uF (\uup,\xi) := 
- \limsup_{n \to \infty} \rate^{-1} \log \P( \Gnp \in \cU_p(\uF,\uup) \cap \Gamma^u_n(\xi) ) \,. 
\end{equation}
Fixing hereafter $\xi>0$, upon taking $\eta \to 0$ we get \eqref{eq:cond-ERGM} provided 
that the function 
\[
f_\xi (\uup) := \psi_{\uF,\hamlet} + \UTab_{\uFs} (\uups,\xi)- \hamlet(\1+\uup)
\]
is bounded away from zero over $\uup \in \R^m_{\ge 0}$. To this end, 
consider the continuous mapping $\bT:\R^2_{\ge 0} \mapsto \R^m_{\ge 0}$  given in terms of 
$\tmax_k=\tmax_{F_k}$ of \eqref{def:TF}, by
\[
\bT (\csize,\bsize) := \{ \tmax_1(\csize,\bsize) -1 , \tmax_2(\csize,\bsize) -1 , \ldots, 
\tmax_m(\csize,\bsize) - 1 \} \,.
\]
With $h$ non-decreasing in each argument, recall from \eqref{def:phi} that 
for any $(\csize,\bsize) \in \Opt(\phi;\uup)$, 
\begin{equation}\label{eq:f0-ab}
f_0(\uup)=\psi_{\uF,\hamlet} + \frac{1}{2} \csize + \bsize - h(\1+\uup)
\ge \psi_{\uF,\hamlet} + \frac{1}{2} \csize + \bsize - \hamlet(\1 + \bT(\csize,\bsize)) \,.
\end{equation}
As seen while proving Proposition \ref{prop:nmf}, thanks to the growth condition \eqref{assu-h.growth}, 
the non-negative continuous function on the \abbr{rhs} of \eqref{eq:f0-ab} diverges when 
$\|(\csize,\bsize)\| \to \infty$, and it is therefore bounded away from zero on 
the complement of any small neighborhood of its 
bounded
set $\Opt(\psi)$ of global minimizers.
Consequently, for any $\eps>0$, the function $f_0(\cdot)$ is bounded away from zero on the 
complement of  
\begin{equation}\label{def:B-xi}
\B_{\eps} := \big\{ \uup \in \R^m_{\ge 0} : 
\max_{(\csize,\bsize) \in \Opt(\phi;\uup)} \min_{(\csize',\bsize') \in \Opt(\psi)} \{| \csize-\csize'|  
+ |\bsize-\bsize'| \} \le \eps \big\} \,.
\end{equation}
It follows from \eqref{def:phi} that 
$\UTab_{\uF} (\uup) = \UTab_{\uFs} (\uups)$ for any $\uup \in \R^m_{\ge 0}$,
with the same set of optimal $(\csize,\bsize)$ in both variational problems. Hence,
$f_\xi(\cdot) \ge f_0(\cdot)$ and we shall complete the proof upon showing that 
\begin{equation}\label{eq:f-inf}
\lim_{\eps\downarrow 0} \inf_{\uup \in \B_{\eps}} \{ f_\xi (\uup) \} > 0 \,.
\end{equation}
Turning to this task, combining Theorem \ref{thm:struct.ER}(a) with the upper bound of \cite[Thm. 1.1]{CDP} 
and \eqref{eq:probUTab}, we have some $\eta(\uup) >0$, such that for $u=1$,
\begin{equation}\label{eq:cond-er}
\limsup_{n \to \infty} \rate^{-1} \log \P\big(  \Gnp \in \Gamma^u_n(\uFs,\uups,\xi) \big) \le 
- \UTab_{\uFs} (\uups) - 2 \eta(\uup) \,, \qquad \forall \uup \in \B_1 \,,
\end{equation}
where 
\begin{equation}\label{def:Gamma-delta}
\Gamma^u_n(\uF,\uup,\xi) :=
\cU_p(\uF,\uup) \cap
\bigcap_{(\csize,\bsize)\in \Opt(\phi;\uup)} (\cG_n^u(\csize,\bsize,\xi))^c \,, \quad u=1,2\,.
\end{equation}
We similarly get in the setting of Theorem \ref{thm:struct.ER}(b), that \eqref{eq:cond-er} holds with $u=2$.
Now, parsing the definitions \eqref{def:Gamma-eps}, \eqref{def:B-xi} and \eqref{def:Gamma-delta}, 
we deduce from Lemma \ref{lem:contain} that for some $\eps=\eps(\xi)>0$ 
small enough, if $\uup \in \B_\eps$ then for all $n$ large enough,
\[
\cU_p(\uFs,\uups) \cap \Gamma^u_n(\xi) \subset \Gamma^u_n(\uFs,\uups,\xi) \,.
\]
Comparing \eqref{def:phi-xi} with \eqref{eq:cond-er}, the preceding containment
implies that throughout $\B_\eps$, 
\[
f_\xi (\uup) \ge \psi_{\uF,\hamlet} + \UTab_{\uF} (\uup) + 2 \eta(\uup) - h(\1+\uup) \,,
\]
yielding by the continuity of $\UTab_\uF(\cdot)-\hamlet(\1+\cdot)$, that for 
some $\zeta(\uup)>0$ and any $\uup \in \B_\eps$,
\begin{equation}\label{eq:phi-cont}
\inf_{\|\uup' - \uup \|_\infty < \zeta(\uup)} \{ f_\xi (\uup') \} \ge \eta(\uup)  > 0 \,.
\end{equation}
Since  
$\bT(\csize,\bsize) \ge \uup$ coordinate-wise for any $(\csize,\bsize) \in \Opt(\phi;\uup)$, 
we deduce from the boundedness $\Opt(\psi)$ that $\B_\eps$ is a bounded, hence pre-compact
subset of $\R^m_{\ge 0}$. Apply \eqref{eq:phi-cont} for a finite cover of 
$\B_\eps$ by $\|\cdot\|_\infty$-balls of centers $\uup_i$ and radii $\zeta(\uup_i)$, 
to arrive at \eqref{eq:f-inf}, thus completing the proof. 
\qed

\section{Edge-$F$ models: Proofs of Proposition \ref{prop:edgeF} and Corollary \ref{cor:edge-triangle}}
\label{sec:edgeF}

To lighten notation, we suppress the dependence of $\phi_F,\verts(F), \edges(F),s_c(F)$ on $F$ throughout 
this section and write $\psi(\beta):=\psi_{F,\beta f}$. Thus,
\begin{equation}
\Opt(\psi) = \{ (a,b): \beta f(T_F(a,b))-\frac12a-b=\psi(\beta)\} .
\end{equation}

From Theorem \ref{thm:struct.ERGM}, to obtain Proposition \ref{prop:edgeF} it suffices to prove the following.

\begin{prop}
\label{prop:edgeF-new}
With hypotheses as in Proposition \ref{prop:edgeF}, 
if $F$ is irregular, then
$
\Opt(\psi) = \{(0,\bsizes(\beta))\}.
$
If $F$ is regular, then there exists $\beta_c>0$ (depending only of $F,f$) such that for $\beta<\beta_c$ we have $\Opt(\psi) = \{(0,\bsizes(\beta))\}$, while if $\beta>\beta_c$, then $\Opt(\psi) = \{ (\csizes(\beta),0)\}$. 
\end{prop}

For the proof we need two lemmas. 

\begin{lemma}
\label{lem:monotone}
Let $g,h$ be real-valued functions on a closed (possibly infinite) interval $I$, and assume $g$ is strictly increasing. 
For $\beta\in \R$ and $s\in I$ let 
\[
U(\beta,s) = \beta g(s) - h(s). 
\]
Then for any $\beta_1<\beta_2$ and any maximizers $s_1,s_2\in I$ for $U(\beta_1,\cdot)$ and $U(\beta_2,\cdot)$ respectively, we have $s_2\ge s_1$.
\end{lemma}

\begin{proof}
Suppose toward a contradiction that $s_2<s_1$. 
Then
\begin{align*}
U(\beta_2,s_2) 
& = \beta_1 g(s_2) -h(s_2) + (\beta_2-\beta_1) g(s_2)\\
&\le \beta_1 g(s_1) - h(s_1) + (\beta_2-\beta_1)g(s_2)\\
&< \beta_1 g(s_1) - h(s_1) + (\beta_2-\beta_1)g(s_1)\\
&=\beta_2 g(s_1) - h(s_1) \le U(\beta_2,s_2),
\end{align*}
a contradiction. 
\end{proof}

For the function $U(\beta,s)$ of \eqref{def:Ubs}, we let $S^\star(\beta)$ denote the set of maximizers for $U(\beta,\cdot)$ in $\R_{\ge0}$. From Proposition \ref{prop:opt}(c) we have that $S^\star(\beta)\ne\emptyset$.

\begin{lemma}
\label{lem:deriv}
Assume $F$ is $\Delta$-regular. 
Then for all $\beta\ge0$ we have that $s_c\notin S^\star(\beta)$.
\end{lemma}

\begin{proof}
Since $\beta f(1+\cdot)$ is continuous and differentable on $\R_+$ with $\phi(s)$
the minimum of two differentable functions that are equal only at 
$s=0$ and at $s=s_c>0$, it suffices to verify that 
\[
\lim_{s\uparrow s_c} \phi'(s) > \lim_{s\downarrow s_c} \phi'(s).
\]
Indeed, it then follows that $s_c$ cannot be a local maximum for $\beta f(1+\cdot) - \phi(\cdot)$. 
The above amounts to showing that
\begin{equation}	\label{deriv.goal}
\frac1{\verts} s_c^{2/\verts-1} < b_0'(s_c) = \frac1{P_F'(\frac12 s_c^{2/\verts})} \,,
\end{equation}
where we denote $b_0(s):= P_F^{-1}(1+s)$. 
Writing $P_F(b) = 1+ \verts b+R(b)$, where $R(b)$ collects all terms of degree at least 2, we have for $b>0$ that
\[
P_F'(b) = \verts + R'(b) = \verts + \sum_{U\in \cI: |U|\ge2} |U|b^{|U|-1} \le \verts \Big( 1+ \frac{R(b)}{2b} \Big) \,,
\]
where we applied the bound $\edges/\Delta=\verts/2$ on the size of an independent set in $F$.
Now since
\[
1+s_c = P_F(b_0(s_c)) = 1 +\verts b_0(s_c) + R(b_0(s_c)) \,,
\]
when combining with the previous display, we have that
\[
P_F'(b_0(s_c)) \le \verts \big( 1+ s_c^{1-2/\verts} - \frac\verts2 \big)< \verts s_c^{1-2/\verts} 
\]
(using here that $\verts>2$ since $\Delta\ge2$). This rearranges to give \eqref{deriv.goal} and the claim follows. 
\end{proof}

\begin{proof}[Proof of Proposition \ref{prop:edgeF-new}]
Write
\begin{equation}
a_0(s) = s^{2/\verts}\,,\qquad
b_0(s) = P_{F^\star}^{-1}(1+s)\,.
\end{equation}
It is shown in  \cite{BGLZ} that for $F$ irregular and any $s\ge0$, or $F$ regular and $s\in[0,s_c)$, we have
\begin{equation}	\label{BGLZb}
\Opt(\phi;s) = \{ (0, b_0(s))\}
\end{equation}
while if $F$ is regular and $s>s_c$, then
\begin{equation}	\label{BGLZa}
\Opt(\phi;s) = \{ (a_0(s), 0)\}\,.
\end{equation}
In the case that $F$ is irregular the proposition then immediately follows from Proposition \ref{prop:opt}(d), with $\bsizes(\beta) = b_0(s_\star(\beta))$. 

Assume now that $F$ is regular. Recalling the notation \eqref{def:Ubs},
from Proposition \ref{prop:opt}(d) and \eqref{BGLZb}--\eqref{BGLZa} is suffices to show there exists $\beta_c=\beta_c>0$ such that
\begin{equation}	\label{goal.hub}
\beta\in [0,\beta_c)\;\Rightarrow S^\star(\beta) = \{s_{hub}^\star(\beta)\}\subset(0,s_c)
\end{equation}
and
\begin{equation}	\label{goal.clique}
\beta\in (\beta_c,\infty)\;\Rightarrow S^\star(\beta) = \{s_{clique}^\star(\beta)\}\subset(s_c,\infty)\,.
\end{equation}

From Lemma \ref{lem:monotone} applied to $U_{hub}$ (taking $f(1+\cdot)$ for $g$ and $P_{F^\star}^{-1}(1+\cdot)$ for $h$), along with the assumption that $s^\star_{hub}(\beta)$ is the unique maximizer of $U_{hub}(\beta,\cdot)$ for all $\beta\ge0$, it follows that $s^\star_{hub}:\R_{\ge0}\to [0,s_c]$ is continuous and non-decreasing. Indeed, the monotonicity is a direct consequence of the lemma, and from this it follows that any point $\beta_0$ of discontinuity of $s^\star_{hub}$ must be a jump discontinuity. However, from the joint continuity of $U_{hub}(\cdot,\cdot)$ we would then have that the left and right limits $\lim_{\beta\uparrow\beta_0}s^\star_{hub}(\beta)$ and $\lim_{\beta\downarrow\beta_0}s^\star_{hub}(\beta)$ would both be maximizers for $U_{hub}(\beta_0,\cdot)$, which contradicts the uniqueness assumption. 
By the same reasoning we get that $s^\star_{clique}:\R_{\ge0}\to [s_c,\infty)$ is continuous and non-decreasing. 

Clearly $S^\star(0) = \{0\} = \{s^\star_{hub}(0)\}$. 
Moreover, since the global maximum of $U(\beta,\cdot)$ is at least the maximum over $[0,s_c]$ and $[s_c,\infty)$ respectively, we have
\[
S^\star(\beta)\cap [0,s_c)\ne \emptyset \;\Rightarrow \; S^\star(\beta)\cap [0,s_c) = \{ s^\star_{hub}(\beta)\}
\]
and similarly
\[
S^\star(\beta)\cap (s_c,\infty)\ne \emptyset \;\Rightarrow \; S^\star(\beta)\cap (s_c,\infty) = \{ s^\star_{clique}(\beta)\}
\]
We set
\[
\beta_c := \sup\{ \beta\ge0: S^\star(\beta)\cap [0,s_c)\ne \emptyset\}.
\]
Note that since $\phi$ and its derivative are bounded on $[0,s_c]$, there exists $B<\infty$ such that $s^\star_{hub}(\beta) = s_c$ for all $\beta\ge B$. From Lemma \ref{lem:deriv} this implies $\beta_c<\infty$.
Since $S^\star(\beta)\ne\emptyset$ for all $\beta\ge0$ we conclude that $S^\star(\beta) = \{ s^\star_{clique}(\beta)\}$ for all $\beta>\beta_c$, and from Lemma \ref{lem:deriv} we have $s^\star_{clique}(\beta)\in (s_c,\infty)$ for all such $\beta$, giving \eqref{goal.clique}.

To argue that $S^\star(\beta)=\{ s^\star_{hub}(\beta)\}$ for all $\beta<\beta_c$, suppose towards a contradiction that $S^\star(\beta)\cap (s_c,\infty)\ne\emptyset$ for some $\beta<\beta_c$. By definition this means that $S^\star(\beta)$ has nonempty intersection with both $[0,s_c)$ and $(s_c,\infty)$. But from Lemma \ref{lem:monotone} it follows that $S^\star(\beta') \subset (s_c,\infty)$ for all $\beta'>\beta$, since the minimal element of $S^\star(\beta')$ bounds the maximal element of $S^\star(\beta)$. We thus obtain a contradiction, so $S^\star(\beta)=\{ s^\star_{hub}(\beta)\}$ for all $\beta<\beta_c$. 
It only remains to note that from Lemma \ref{lem:deriv} if follows that $s^\star_{hub}(\beta)\in (0,s_c)$ for such $\beta$, which gives \eqref{goal.hub} and completes the proof.
\end{proof}

\begin{proof}[Proof of Corollary \ref{cor:edge-triangle}]
For $F=C_3$ we have $\verts(F) = 3$ and $P_F(x) = 1+ 3x$, and hence 
\[
\phi(s) = \min\{ \frac12s^{2/3}, \frac13s\} = \begin{cases} \frac13s & s\in [0, \frac{27}8]\\ \frac12 s^{2/3} & s\in [\frac{27}8,\infty).\end{cases}
\]
From Proposition \ref{prop:edgeF}, it suffices to verify that 
\begin{enumerate}
\item[(a)] $\csizes(\beta) =s^\star_{clique}(\beta)^{2/3} = (\gamma\beta)^{\frac{2}{2-\gamma}}$; 
\item[(b)] $\bsizes(\beta) = P_{C_3}^{-1}(1+ s^\star_{hub}(\beta)) = \frac13
(\gamma \beta)^{\frac{3}{3-\gamma}}$; and
\item[(c)] $s\mapsto U(\beta, s) = \beta s^{\gamma/3} - \phi(s)$ achieves its global maximum in 
$[0,\frac{27}8)$ when $0 \le \beta < \beta_c $, and in $(\frac{27}8,\infty)$ when $\beta>
\beta_c $, for  
\[
\gamma \beta_c= \Big( \frac{6-2\gamma}{6-3\gamma} \Big)^{(2-\gamma)(3-\gamma)/\gamma}\,.
\]
\end{enumerate}
For (a), one merely verifies that $U_{clique}(\beta, s) = \beta s^{\gamma /3} - \frac12s^{2/3}$ is maximized at $s^\star_{clique}(\beta) = (\gamma\beta)^{\frac{3}{2-\gamma}}$.
Similarly, for (b) one verifies that $U_{hub}(\beta,s) = \beta s^{\gamma /3}-\frac13s$ is maximized at $s^\star_{hub}(\beta) = (\gamma\beta)^{\frac{3}{3-\gamma}}$. 
Finally, for (c) one verifies that $U_{hub}(\beta,(\gamma\beta)^{\frac{3}{3-\gamma}})>U_{clique}(\beta,(\gamma\beta)^{\frac{3}{2-\gamma}})$ 
if and only if $\beta<\beta_c $, and that further
%
$$
\big[s^\star_{hub}(\beta_c)\big]^{\gamma/3} = 
\Big( \frac{6-2\gamma}{6-3\gamma} \Big)^{2-\gamma}
< \Big(\frac{3}{2}\Big)^{\gamma} < 
\Big( \frac{6-2\gamma}{6-3\gamma} \Big)^{3-\gamma}
= \big[s^\star_{clique}(\beta_c)\big]^{\gamma/3}
$$
(as upon taking the logarithm, the inequalities are $I_{\gamma/3}(\gamma/2)>0$ 
and $I_{\gamma/2}(\gamma/3)>0$).
\end{proof}

\begin{appendix}

\section{Stability of Finner's inequality}
\label{sec:finner}

In this appendix we prove Theorem \ref{thm:finner.stab}.
In what follows we abuse notation by writing e.g. $\prod_A f_A$ instead of $\prod_A f_A\circ \pi_A$ with $\pi_A: \Omega\to \Omega_A$ the coordinate projection mapping. 
We further use $\|\cdot\|_{q}$ for the $L_q$ norms, whenever the underlying space is clear from the context.

\subsection{Stability of H\"older's inequality}

We shall prove Theorem \ref{thm:finner.stab} by induction on $n$ (following Finner's argument for the case $\eps=0$, i.e.\  characterizing the case for equality), relying on the following stability property of H\"older's inequality.
\begin{lemma}
\label{lem:holder.stab}
For any $\lambda \in (0,1)$, $\eps \in [0,1]$ and
$g : \Omega \to \R_{\ge0}$ on a probability space $(\Omega,\nu)$,
\begin{equation}\label{eq:basic-stab}
\int g d\nu \le 1, \qquad 
1-\eps \le \int g^\lambda d\nu \quad \Longrightarrow \quad \|g-1\|_1 \le 2 {\bar C}_\lam \, \eps^{1/2}
\end{equation}
where ${\bar C}_\lam :=\sqrt{\frac{2}{\lambda(1-\lambda)}}$.
\end{lemma}

\begin{remark} 
\label{rmk:holder.c}
The case of $g = 1 \pm \bar C_\lambda \eps^{1/2}$ and $\nu$ the Bernoulli($1/2)$ measure shows that  the bound \eqref{eq:basic-stab} is optimal up to a factor 2 for small $\eps$. 
\end{remark}

\begin{proof} Note that $\varphi(x):=1-(1+x)^\lam+\lam x - \frac{\lam (1-\lam)}{2} x^2 {\bf 1}_{\{x<0\}}$ 
is non-negative on $[-1,\infty)$ 
(indeed, $\varphi(0)=\varphi'(0)=0$ and $\varphi''(x) \ge 0$ both on $[-1,0)$ and on $\R_{\ge0}$). In particular, 
$\int \varphi(\wh{g}) d\nu \ge 0$ for $\wh{g}:=g-1 : \Omega \to [-1,\infty)$. Our assumptions that 
$\int \wh{g} d\nu \le 0$ and 
\[ 
\eps \ge \int  (1-(1+\wh{g})^\lam) d\nu = \int \varphi(\wh{g}) d\nu -\lam \int \wh{g} d\nu 
+\frac{\lam(1-\lam)}{2} \int_{\{\wh{g}<0\}} \wh{g}^2 d\nu \,,
\]
thus yield that
\[
\bar C_\lam^2 \, \eps \ge  \int_{\{\wh{g}<0\}} \wh{g}^2 d\nu \ge 
\Big(\int_{\{\wh{g}<0\}} |\wh{g}| d\nu\Big)^2
\ge \frac{1}{4} \|\wh{g}\|_1^2 
\]
as claimed (for the last inequality, note that
$\int |\wh{g}| d\nu = 2 \int_{\{\wh{g}<0\}} |\wh{g}| d\nu +\int \wh{g} d\nu$).
\end{proof}

Using Lemma \ref{lem:holder.stab} we deduce the following stability of the (classical) generalized H\"older inequality.

\begin{prop}[Stability of the generalized H\"older inequality]
\label{prop:genhold.stab}
Suppose $f_i \ge 0$ on a probability space $(\Omega,\mu)$ are such that $\int f_i d\mu \le 1$.
Then, for any $m \ge 2$, $\lam_i > 0$ such that $\sum_i \lam_i \le 1$ and $\eps \in [0,1]$, 
\begin{equation}\label{eq:gen-stab}
1-\eps\le \int \prod_{i=1}^m f_i^{\lam_i} \d\mu \quad \Longrightarrow \quad 
\|f_k-f_\ell\|_1 \le C(\lam_k,\lam_\ell) \, \eps^{1/2} \, \quad \forall k,\ell \in [m] \,,
\end{equation}
where $C(\lam,\lam')=(\lam+\lam')^{-1/2} C_{\lam/(\lam+\lam')}$ and 
$C_\lam = 2 \bar C_\lam + 1/(\lam \wedge (1-\lam))$.
\end{prop}

\begin{proof} Set the strictly positive $p_i=1/\lam_i$ for $i \ne k,\ell$, with $q=1/(\lam_k+\lam_\ell) \ge 1$ 
and $r=1/(\sum_i \lam_i) \ge 1$, so that $1/r = 1/q + \sum_{i \ne k,\ell} 1/p_i$. It then follows
from the \abbr{lhs} of \eqref{eq:gen-stab} and the generalized H\"older inequality, that for 
$\lam:=\lam_k q$,
\[
1- q \eps \le (1-\eps)^q \le \big(\| \prod_{i=1}^m f_i^{\lam_i} \|_1\big)^q 
\le \big(\| \prod_{i=1}^m f_i^{\lam_i} \|_r\big)^q \le \big(\int f_\ell^{1-\lam} f_k^{\lam} d\mu\big)
\big(\prod_{i \ne k,\ell} \| f_i^{\lam_i} \|_{p_i}\big)^q \,.
\]
Since by assumption $\|f_i^{\lam_i}\|_{p_i}=(\int f_i d\mu)^{1/p_i} \le 1$, we have thus reduced the \abbr{lhs}
of \eqref{eq:gen-stab} to 
\begin{equation}\label{eq:eps-ass}
1-\eps' \le \int f^{1-\lam}  h^{\lam} d\mu = \|f\|_1 \int g^{\lam} d\nu \le  \int g^{\lam} d\nu \,,
\end{equation}
for $\eps'=q \eps \wedge 1$, $f=f_\ell$, $h=f_k$, $g=(h/f) {\bf 1}_{\{f>0\}}$ and the probability measure
$\nu=\frac{f}{\|f\|_1} \mu$. Further, 
\begin{equation}\label{eq:g-nu}
\|f\|_1 \int g d\nu =\|h\|_1 - \int_{\{f=0\}} h d\mu \le \|h\|_1 \,,
\end{equation}
so in case $\|f_k\|_1 \le \|f_\ell\|_1$, we deduce from \eqref{eq:basic-stab} that 
\begin{equation}\label{eq:f-pos}
\int_{\{f>0\}} |h-f| d\mu = \|f\|_1 \int |g-1| d\nu \le 2 \bar C_{\lam} \, \sqrt{\eps'} \,.
\end{equation}
In addition, combining \eqref{eq:eps-ass}, \eqref{eq:g-nu} and Jensen's inequality, we arrive at
\[
1-\eps' \le \|f\|_1 \int g^{\lam} d\nu \le \|f\|_1^{1-\lam} \big(\|f\|_1 \int g d\nu\big)^{\lam} 
=\|f\|_1^{1-\lam} \big(\|h\|_1 - \int_{\{f=0\}} h d\mu\big)^{\lam} \,.
\]
Consequently, as $\lam,\eps' \le 1$, 
\begin{equation}\label{eq:f-zero}
\int_{\{f=0\}} |h-f| d \mu = \int_{\{f=0\}} h d \mu \le 1-(1-\eps')^{1/\lam} \le \frac{\eps'}{\lam} \le 
\frac{\sqrt{\eps'}}{\lam} \,,
\end{equation}
which, together with \eqref{eq:f-pos}, results with $\|f_k-f_\ell\|_1 \le C_{\lam} \sqrt{\eps'}$.
The same applies when $\|f_\ell\|_1 < \|f_k\|_1$, except for exchanging the roles
of $f$ and $h$ by mapping $\lam \mapsto 1-\lam$.  
\end{proof}

\begin{remark} From the proof of Proposition \ref{prop:genhold.stab}, 
when $\int f_i d\mu$ is constant in $i$, the result improves to $C_\lam=2 (\bar C_\lam +1)$ 
and consequently, to $C(\lam,\lam') \le 6(\lam \wedge \lam')^{-1/2}$, 
by choosing the roles of $f$ and $h$ according to whether $\lambda > 1/2$ or not.
 
Proposition \ref{prop:genhold.stab} at $\lam_i=r/p_i$, where $\sum_i 1/p_i = 1/r$, amounts to
\begin{equation}\label{eq:holder-stab}
\|g_i\|_{p_i} = 1, \quad 
1-\eps \le \Big\| \prod_i g_i \Big\|^r_r \qquad \Longrightarrow \qquad 
\int \Big|  |g_k|^{p_k}- |g_\ell|^{p_\ell}\Big| 
d\mu \le C(\lam_k,\lam_\ell) \, \sqrt{\eps} \,.
\end{equation}
\end{remark}

\subsection{Proof of Theorem \ref{thm:finner.stab}}

Clearly we may assume that $\cA$ covers $V$.
Without loss of generality we may further assume $\cB=\{\{v\}\}_{v\in V}$.
We may also assume 
\begin{equation}	\label{wlog:cover}
\sum_{A\ni v} \lam_A =1\qquad \forall\;v\in V.
\end{equation}
Indeed, for any $v$ for which this does not hold, we can add $\{v\}$ to the set system $\cA$, setting $\lam_{\{v\}}=1-\sum_{A\ni v}\lam_A$ and $f_{\{v\}}=1$. 

We proceed by induction on $n:=|V|$, showing that for any set $V$ of size $n$ and $\cA,(f_A)_{A\in\cA}$ and $\Lambda$ as in the theorem statement, there exist functions $h_v:\Omega_v\to\R_{\ge0}$ with $\int h_v d\mu_v =1$ for each $v\in V$ such that
\begin{equation}	\label{finner.stab:induction}
\|f_A-h_A \|_{L_1(\Omega_V)} \le C_n\eps^{c_n}
\end{equation}
for constants $C_n(\cA,\Lambda)>0$ and $c_n(\cA,\Lambda)\in (0,\frac12]$ to be determined, where $h_A:=\bigotimes_{v\in A}h_v$. 

For the base case $n=1$, the $f_A$ are functions over a common space $(\Omega,\mu)$, so we can apply Proposition \ref{prop:genhold.stab} to conclude $\|f_A-f_{A'}\|_1\ls_{\Lambda} \eps^{1/2}$ for all $A,A'\in\cA$.
Now from our hypothesis \eqref{finner.stab:LB} and Theorem \ref{thm:finner} we have 
\[
1-\eps\le \int \prod_A f_A^{\lam_A}d\mu \le \prod_{A} \bigg( \int f_A\,d\mu\bigg)^{\lam_A} \le \bigg( \int f_{A_0} d\mu \bigg)^{\lam_{A_0}}
\]
for each $A_0\in \cA$. Fixing an arbitrary $A_0$, by Lemma \ref{lem:holder.stab} one has for $h \equiv 1$ 
that $\|h-f_{A_0}\|_1\ls_{\Lambda} \eps^{1/2}$, and \eqref{finner.stab:induction} follows by the triangle inequality, with $c_1=1/2$ and some $C_1$ sufficiently large depending on $\Lambda$.

Consider now the case that $|V|=n\ge2$ and that the theorem statement holds for any $V$ with $|V|<n$. 
Denote $\cA_n=\{A\in\cA: |A|=n\}$.
The case $\cA=\cA_n$ is handled exactly as in the case $n=1$ so we assume $\cA\ne\cA_n$.

Consider first the case that $\cA_n=\emptyset$.
For $v\in V$ we denote the contracted set system
\begin{equation}\label{Av.contract}
\cA^{(v)}=\big\{ A\setminus\{v\}: A\in\cA\big\}
\end{equation}
over $V\setminus\{v\}$ (retaining repeats). 
Define functions
\begin{equation}
g_{A,(v)} = \int f_Ad\mu_v : \Omega_{A\setminus\{v\}} \to\R_{\ge0}\,, \qquad v\in A.
\end{equation}
By \eqref{finner.stab:LB} and the generalized H\"older inequality,
\[	
1-\eps \le \int \prod_{A\notni v} f_A^{\lam_A} \bigg( \int \prod_{A\ni v} f_A^{\lam_A} \,d\mu_{v}
\bigg) \,d\mu_{V\setminus\{v\}}
\le \int \prod_{A\notni v} f_A^{\lam_A} \prod_{A\ni v} g_{A,(v)}^{\lam_A} \,d\mu_{V\setminus \{v\}}.
\]
Since $\int g_{A,(v)} d \mu_{A \setminus \{v\}}= \int f_A d\mu_A \le 1$ for all $A\ni v$, we can apply the induction hypothesis, with set system $\cA^{(v)}$ over $V\setminus\{v\}$, to obtain functions $h_u^{(v)}:\Omega_u\to \R_{\ge0}$ for each $u\in V\setminus\{v\}$ with $\int h_u^{(v)} d \mu_u =1$, such that
\begin{align}	\label{fA.compare}
\| f_A- h_A^{(v)}\|_{L_1(\Omega_A)}  &\le C(v) \eps^{c(v)} \,, \qquad \forall A \notni v \,, \\
\label{gA.compare}
 \|g_{A,(v)}- h_{A\setminus\{v\}}^{(v)} \|_{L_1(\Omega_{A\setminus\{v\}})} &\le C(v)  \eps^{c(v)} \,, 
 \qquad \forall A \ni v \,,
\end{align}
where 
$h_A^{(v)}:= \bigotimes_{u\in A} h_u^{(v)}$
and $C(v):=C_{n-1}(\cA^{(v)}, \Lambda^{(v)})$, $c(v):=c_{n-1}(\cA^{(v)},\Lambda^{(v)})$. (Here $\Lambda^{(v)}$ is the collection of weights for $\cA^{(v)}$ inherited from $\Lambda$ under the contraction \eqref{Av.contract}.)

Having obtained the family of functions $\{h_u^{(v)}: u,v\in V, u\ne v\}$ satisfying \eqref{fA.compare}--\eqref{gA.compare}, we now
fix arbitrary distinct $w,z\in V$ and take 
\begin{equation}		
h_u :=
\begin{cases} h_u^{(w)} & u\ne w\\ h_w^{(z)} & u=w.\end{cases}
\end{equation}
It only remains to verify \eqref{finner.stab:induction}, i.e.\, that
\begin{equation}		\label{finner.stab:remains}
\big\|f_A - h_w^{(z)} \otimes h_{A\setminus\{w\}}^{(w)} \big\|_{L_1(\Omega_A)} \le C_n \eps^{c_n}
\end{equation}
for appropriate $C_n,c_n$ and each $A\in \cA$ containing $w$ (for all other $A$ the claim is immediate from \eqref{fA.compare}, taking $C_n\ge C(w)$ and $c_n\le c(w)$).

We first claim that for any $A\in \cA, u\in A$ and $v\notin A$,
\begin{equation}	\label{hAuv}
\|h_{A\setminus \{u\}}^{(u)} - h_{A\setminus\{u\}}^{(v)}\|_{L_1(\Omega_{A\setminus\{u\}})} \le 
C(u) \eps^{c(u)} +  
C(v)  \eps^{c(v)}.
\end{equation}
Indeed, from the triangle inequality and \eqref{gA.compare} the \abbr{lhs} above is bounded by
\begin{align*}
\| g_{A,(u)} - h_{A\setminus\{u\}}^{(v)}\|_{L_1(\Omega_{A\setminus \{u\}})} + C(u)  \eps^{c(u)}
\end{align*}
and we can express the first term above as
\[
\Big\| \int \big(f_A  - h_A^{(v)}\big) d\mu_u\Big\|_{L_1(\Omega_{A\setminus\{u\}})} 
\le \| f_A-h_A^{(v)}\|_{L_1(\Omega_A)}
\le C(v)  \eps^{c(v)}
\]
where we applied Minkowski's inequality and \eqref{fA.compare}.

We now establish \eqref{finner.stab:remains}.
For the case that $w\in A$ and $z\notin A$ this follows for any $C_n\ge 2 C(z)+C(w)
$, $c_n\le c(z) \wedge c(w)$ from the triangle inequality, 
\eqref{fA.compare} and \eqref{hAuv} with $v=z$ and $u=w$. 
Now assume $\{w,z\}\subseteq A$. Under our assumption that $\cA_n=\emptyset$ we can select an arbitrary $v\in V\setminus A$ and
bound the \abbr{lhs} of \eqref{finner.stab:remains} by
\begin{align*}
&\|f_A-h_A^{(v)}\|_{L_1(\Omega_A)} 
+ \| h_A^{(v)} - h_w^{(v)}\otimes h_{A\setminus\{w\}}^{(w)} \|_{L_1(\Omega_A)} + \|(h_w^{(z)}-h_{w}^{(v)})h_{A\setminus\{w\}}^{(w)} \|_{L_1(\Omega_A)}\\
&\qquad \le C(v)  \eps^{c(v)} + \|h_{A\setminus\{w\}}^{(v)} - h_{A\setminus\{w\}}^{(w)} \|_{L_1(\Omega_{A\setminus\{w\}})} 
+ \|h_w^{(z)} - h_w^{(v)}\|_{L_1(\Omega_w)}.
\end{align*}
The second term on the \abbr{rhs} is  at most $C(z) \eps^{c(z)} + C(w) \eps^{c(w)}$
from \eqref{hAuv}.
For the third term,
\begin{align*}
\|h_w^{(z)} - h_w^{(v)}\|_{L_1(\Omega_w)}
&= \Big\| \int \big(h_{A\setminus \{z\}}^{(z)} - h_{A\setminus\{z\}}^{(v)} \big) d\mu_{A\setminus\{w,z\}} \Big\|_{L_1(\Omega_w)}\nonumber\\
&\le \|h_{A\setminus\{z\}}^{(z)} - h_{A\setminus\{z\}}^{(v)} \|_{L_1(\Omega_{A\setminus \{z\}})}
\end{align*}
which is at most $C(v)  \eps^{c(v)} + C(z) \eps^{c(z)}$ by \eqref{hAuv} 
(if $A=\{w,z\}$ then $A \setminus \{z\} = \{w\}$, trivially yielding the same bound).
Altogether we conclude in case $\cA_n=\emptyset$, that \eqref{finner.stab:induction} holds for any 
\begin{equation}
C_n\ge C_n'(\cA,\Lambda):=5\max\{C(u) :u\in V\}\;,\qquad c_n\le c_n'(\cA,\Lambda):=\min\{c(u):u\in V\}.
\end{equation}

Finally, suppose $\emptyset\ne \cA_n\ne \cA$. 
Put $\lam_\star:= \sum_{A\in\cA_n}\lam_A <1$ and $F:= \prod_{\{A:|A|<n\}} f_A^{\lam_A'}$ with $\lam_A':= \lam_A/(1-\lam_\star)$. Then, from  \eqref{finner.stab:LB} and the generalized H\"older inequality,
\begin{equation}	\label{fin.F}
1-\eps \le \int F^{1-\lam_\star} \prod_{\{A: |A|=n\}} f_A^{\lam_A} d\mu_V \le \Big(\int Fd\mu_V\Big)^{1-\lam_\star}.
\end{equation}
Thus,
\[
1-\frac{\eps}{1-\lam_\star} \le (1-\eps)^{1/(1-\lam_\star)} \le\int F\,d\mu_V= \int\prod_{\{A: |A|<n\}} f_A^{\lam_A'}d\mu_V.
\]
Applying the result for the case that $\cA_n=\emptyset$, with $\cA':=\cA\setminus\cA_n$ in place of $\cA$ and $\Lambda'=(\lam_A')_{A\in\cA'}$ in place of $\Lambda$, we obtain $(h_v)_{v\in V}$ such that 
\begin{equation}	\label{fAhA1}
\|f_A-h_A\|_{L_1(\Omega_A)}\le C_n'(\cA',\Lambda') \bigg(\frac{\eps}{1-\lam_\star}\bigg)^{c_n'(\cA',\Lambda')}\qquad\forall\;A\in \cA'.
\end{equation}
Assuming $C_n\ge C_n'(\cA',\Lambda')(1-\lam_\star)^{-c_n'(\cA',\Lambda')}$ and $c_n\le c_n'(\cA',\Lambda')$, it only remains to establish \eqref{finner.stab:induction} for $A\in \cA_n$. 
Now from Proposition \ref{prop:genhold.stab} and the first inequality in \eqref{fin.F} it follows that
\begin{equation}
\|f_A-F\|_{L_1(\Omega_V)} \ls_\Lambda \eps^{1/2}	\quad \forall \;A\in\cA: \,|A|=n\,,
\end{equation}
so by the triangle inequality and taking $C_n$ larger, if necessary, it suffices to show 
\begin{equation}
\|F-h_V\|_{L_1(\Omega_V)} \le C_n \eps^{c_n}
\end{equation}
for possibly adjusted values of $C_n,c_n$.
We obtain this by expanding the difference as a telescoping sum over $A\in\cA'$ and applying \eqref{fAhA1}. 
Enumerating the elements of $\cA'$ as $A_j, 1\le j\le m$, we have
\[
F-h_V = \prod_{\{A: |A|<n\}} f_A^{\lam_A'} - \prod_{\{A: |A|<n\}} h_A^{\lam_A'}
= \sum_{i=1}^m (-1)^{i-1} (f_{A_i}^{\lam_{A_i}'}-h_{A_i}^{\lam_{A_i}'})\prod_{j<i} f_{A_j}^{\lam_{A_j}'} \prod_{j>i} h_{A_j}^{\lam_{A_j}'}
\]
(note we used \eqref{wlog:cover} in the first equality).  
Taking $L_1$-norms on both sides and applying the triangle inequality and Finner's inequality
(for $\cA'$ and $\Lambda'$), we obtain
\begin{align*}
\|F-h_V\|_{L_1(\Omega_V)} 
&\le \sum_{i=1}^m \int \Big| f_{A_i}^{\lam_{A_i}'}-h_{A_i}^{\lam_{A_i}'} \Big| 
\prod_{j<i} f_{A_j}^{\lam_{A_j}'} \prod_{j>i} h_{A_j}^{\lam_{A_j}'}
d\mu_V\\
&\le  \sum_{i=1}^m \Big( \int \Big| f_{A_i}^{\lam_{A_i}'}-h_{A_i}^{\lam_{A_i}'} \Big|^{1/\lam_{A_i}'} d\mu_{A_i} \Big)^{\lam_{A_i}'}\\
&\le \sum_{i=1}^m \Big( \int \Big| f_{A_i}-h_{A_i} \Big|d\mu_{A_i} \Big)^{\lam_{A_i}'}
\end{align*}
(using the elementary bound $|x-y|^p  \le |x^p-y^p|$ for $x,y \ge 0$, $p \ge 1$, in the last step).
The claim now follows by substituting the bounds \eqref{fAhA1} and taking 
\[
C_n\ge \sum_{i=1}^m \big[C_n'(\cA',\Lambda')(1-\lam_\star)^{-c_n'(\cA',\Lambda')}\big]^{\lam'_{A_i}}
\,,\qquad
c_n\le c_n'(\cA',\Lambda')\cdot \min_i\{\lam'_{A_i}\}\, .
\]

\section{Proof of Proposition \ref{prop:opt}}
\label{sec:opt}

\begin{proof}[Proof of Part (a)]

All but the last claim \eqref{phi.LB} are immediate from the fact that for each $k$,
$\tmax_k$ is continuous, non-decreasing and unbounded, with 
$\tmax_k(0,0)=0$.
Now for \eqref{phi.LB}, since $\phi_{\uF}(\uup)\ge \phi_{F_k}(s_k)$ for each $k\in[m]$, it suffices to establish the case $m=1$.
We claim that for all $a,b \ge0$ with $a+b\ge1$,
\begin{equation}	\label{TF.bound}
T_F(a,b)\ls_F (a+b)^{\edges(F)/\Delta}.
\end{equation}
Indeed, $\edges(F)/\Delta$ is an upper bound for the size of any independent set in $F^\star$, and hence for the degree of $P_{F^\star}$, and in the case that $F$ is regular we have $\verts(F)/2= \edges(F)/\Delta$. 
Now let $C=C(F)>0$ to be taken sufficiently large. Then for arbitrary  $s\ge C$, if $a,b\ge0$ are such that $T_F(a,b)\ge 1+s$, then taking $C$ sufficiently large it follows that $a+b\ge1$, and from \eqref{TF.bound} we get that $\frac12a+b \gs_F s^{\Delta/\edges(F)}$. The claim follows. 
\end{proof}

\begin{proof}[Proof of Part (b)]
For the case that $\uup=\underline{0}$ we have that $\Opt(\phi;\underline{0})$ is the singleton set 
$\{(0,0)\}$ and the claim follows. 

Assume now that $\uup\ne\underline{0}$. By throwing out redundant constraints $T_k \ge 1$ we may assume \abbr{wlog} that $s_k>0$ for each $k$, while re-indexing $F_k$ (as in Section \ref{sec:ERGMs}),
so that $F_k$ is regular if and only if $k\le m'$ for some $m'\in[m]$.
Setting $a_k^\star(\up_k)=\up_k^{2/\verts(F_k)}$, $b_k^\star(\up_k)=P_{F_k^\star}^{-1}(1+\up_k)$, note that for any $k\le m'$ 
the curve $\Gamma_k(\up_k)=\{(a,b):T_k(a,b)=1+\up_k\}\cap\R_{\ge0}^2$ is a smooth arc 
of positive curvature with endpoints $(a_k^\star,0)$ and $(0,b_k^\star)$.
For $k>m'$ we have that $\Gamma_k(\up_k)=\{(a,b_k^\star(\up_k)): a\in \R_{\ge0}\}$ is a horizontal ray with endpoint on the $b$-axis. 
The infimum of the increasing linear function $(a,b)\mapsto\frac12a+b$ over $\Region:=\cap_k\{T_k 
(a,b)
\ge1+\up_k\}$ 
must be attained at some finite point on the boundary of $\Region$. The boundary consists of a finite connected union of smooth curves overlapping only at their endpoints: 
the ray $\Gamma_{vert}':=\{(0,b):b\ge b_0\}$ with $b_0=\max_k \{ b_k^\star(\up_k) \}$, a connected infinite subset $\Gamma_{horiz}'$ of the ray $\{(a,b^\star):a\ge0\}$ with 
$b^\star=0\vee\max_{k>m'}\{b_k^\star(\up_k)\}  \le b_0$, 
and a (possibly empty) finite union of sub-arcs $\Gamma_\alpha'$ of the bounded curves 
$\Gamma_k(\up_k), k\le m'$. (In particular, if $m'=0$ then $b^\star=b_0$ and $\Region$ is 
the axis-aligned quadrant $\{(a,b): a\ge0,b\ge b_0\}$, so the infimum is attained at the single point $(0,b_0)$.)
One easily sees that for each $\alpha$, $\inf\{\frac12a+b: (a,b)\in \Gamma'_\alpha\}$ cannot be achieved on the interior of $\Gamma'_\alpha$. Indeed, $(a,b)\mapsto\frac12a+b$ is strictly monotone on $\Gamma'_{vert},\Gamma'_{horiz}$, and since it is linear it can only have a local maximum on the interior of any of the $\Gamma_\alpha'$ (being a subset of a level curve of one of the strictly convex functions $T_k, k\le m'$). 
Thus, the infimum can only be achieved at one of the finitely-many intersection points of the curves $\Gamma'$.
\end{proof}

\begin{proof}[Proof of Part (c)]

We abbreviate $\psi:=\psi_{\uF,h}$ and denote the \abbr{rhs} of \eqref{psi-phi} by $\psi'$. 
For $a,b\ge0$ we hereafter denote
\[
\uup(a,b):=(T_1(a,b)-1,\dots, T_m(a,b)-1).
\]
For $\uup,\uup'\in \R_{\ge0}^m$ we understand $\uup\ge\uup'$ to mean $s_k\ge s_k'$ for each $k\in[m]$.

We first argue $S^\star$ is nonempty and bounded.
Indeed, this follows from the continuity of $h$, part (a) and the assumption \eqref{assu-h.growth}. 

Now to show $\psi'\le \psi$, since $S^\star$ is nonempty we may fix an arbitrary $\uup'\in \R_{\ge0}^m$ such that $\psi'=h(\1+\uup') - \phi_{\uF}(\uup')$. 
We have
\begin{align*}
\psi' = h(\1+\uup') - \phi_{\uF}(\uup')
&= \sup_{a,b\ge0} \{ h(\1+\uup') - \frac12a-b: \uup(a,b)\ge \uup'\}\\
&\le \sup_{a,b\ge0} \{ h(1+\uup(a,b)) - \frac12a-b\} = \psi
\end{align*}
where for the inequality we used the assumption that $h$ is monotone. 

To see that $\psi\le\psi'$ (which in fact holds under no assumptions on $h$),
letting $a,b\ge0$ be arbitrary, we have
\[
h(\1+\uup(a,b))-\frac12a-b \le h(\1+\uup(a,b)) -\phi_\uF(\uup(a,b)) \le \psi'
\]
and the claim follows upon taking the supremum over $a,b$ on the \abbr{lhs}.
\end{proof}

\begin{proof}[Proof of Part (d)]
For the containment $\supseteq$, from parts (b) and (c) we may fix arbitrary $\uup\in S^\star$ and $(a,b)\in \Opt(\phi;\uup)$. 
Thus, $\phi_{\uF}(\uup)=\frac12a+b$, and $\uup(a,b)\ge \uup$. 
Then we have
\begin{align*}
h(\1+\uup(a,b)) -\frac12a-b &= h(\1+\uup(a,b)) - \phi_{\uF}(\uup) \\
&\ge h(\1+\uup) - \phi_\uF(\uup) 
= \psi \ge h(\1+\uup(a,b)) - \frac12a-b
\end{align*}
where in the first inequality we used the monotonicity assumption, and in the last we used the formula \eqref{psi-phi} established in (c). 
Thus, the inequalities in fact hold with equality, and hence $(a,b)\in \Opt(\psi)$. 

For the containment $\subseteq$, note that from (b), (c) and the containment $\supseteq$ just established, it follows that $\Opt(\psi)$ is nonempty (this can also be seen directly following similar reasoning as in the proof of (a)). 
Thus, we may fix an arbitrary element $(a,b)\in \Opt(\psi)$. 
We claim that $\uup(a,b)\in S^\star$ and $(a,b)\in \Opt(\phi; \uup(a,b))$, from which the result follows.
Indeed, 
\[
\psi = h(\1+\uup(a,b)) -\frac12a-b \le h(\1+\uup(a,b))-\phi_\uF(\uup(a,b)) \le \psi
\]
where in the final bound we used the relation \eqref{psi-phi} established in (c). 
Thus, equality holds throughout, and from equality in the first bound we get that $(a,b)\in \Opt(\phi;\uup(a,b))$, while equality in the second bound implies $\uup(a,b)\in S^\star$. By (c) the set
$S^\star$  is pre-compact, so $\phi_\uF(\cdot)$ is bounded on $S^\star$. Hence, by (b) and 
the preceding containment, $\Opt(\psi)$ is also bounded.
\end{proof}

\section{Order of the upper tail}
\label{sec:rough-ubd}

In this appendix we establish the following proposition used in the proof of \Cref{lem:truncate.tail}.

\begin{prop}\label{prop:rough-ubd}
For any graph $H$ of max degree $\Delta$, there exist finite $C(H)$ and positive $c(H)$ such that
if $p\in (0,1/2]$ and $np^{\Delta+1} \ge C(H)$, then
for any $\up \ge 2$, 
\begin{equation}	\label{goal}
\log \P(t(H,\Gnp/p) \ge s) \le
\begin{cases}
 -c(H) \up^{\Delta/\edges(H)} p^\Delta n^2 \log (1/p) \,,& \Delta\ge 2,\\
- c(H) \up^{\Delta/\edges(H)} p^\Delta n^2 \log \up      \,, & \Delta =1.
\end{cases}
\end{equation}
\end{prop}

\begin{remark}
Using \cite[Thms.\ 2.10 and 3.1]{CDP}, the same proof below yields the 
corresponding bound in the case that $\Gnp$ is the $r$-uniform \ER hyper-graph, 
for any $r \ge 2$, any $r$-graph $H$, assuming $np^{\Delta'(H)}\ge C(H)$, 
with $\Delta'(H)$ as defined in \cite{CDP}.
\end{remark}

\begin{proof} We argue by induction on $\edges(H)$, having as induction hypothesis that 
the bound \eqref{goal} holds with $F$ and $\Delta(F)$ in place of $H$ and $\Delta(H)$, 
for all graphs $F$ with $\edges(F)<\edges(H)$.  

The claim in the  case $\Delta=1$ (which includes the base case $\edges(H)=1$) follows from a standard tail bound for the binomial distribution, along with the fact that $t(H_1\cup H_2, X) = t(H_1,X)t(H_2,X)$ for $H$ a 
disjoint union of two graphs $H_1,H_2$. (In particular we get $c(H)=c/\edges(H)$ in this case.)

Assume now that $\Delta\ge2$. To establish \eqref{goal} for $H$ we use the union bound,
\begin{equation}\label{eq:union}
\P\big(\Gnp\in\cU_p(H,\up) \big) \le \P\big(\Gnp\in\cU_p(H,\up) \cap \cL_{\subsetneq} (H,\up) \big) + \sum_{F \subsetneq H} \P\big( \Gnp\in
\cU_p(F,\up) \big)\,,
\end{equation}
where 
\[
\cL_{\subsetneq} (H, \up) := \bigcap_{F\subsetneq H}  \{Q\in \cQ_n: t(F, Q/p)\le \up \} \,.
\]
As $t(H, Q/p)\le p^{-\edges(H)}$ for any $Q \in \cQ_n$, it suffices to consider only  
\begin{equation}\label{UB.better}
\up \le p^{-\edges(H)} \,.
\end{equation}
Since \eqref{LAo-cond} applies for any $\cE \subset \cG_n$ intersecting $\cL_{\subsetneq} (H,L)$, upon taking $\delta = \delta(H)>0$ small we get 
by \eqref{eq1:count} of Proposition \ref{prop:LD.Bstar}(b) that 
\[
\big(\cU_p(H,\up) \cap \cL_{\subsetneq} (H,\up) \,\big)_{\mathbb{K}^\star(\delta)} \subseteq \cU_p(H,\up/2) \,.
\]
Consequently, see \eqref{Phi-lbd},
\[
\eye_p\Big( \big( \cU_p(H,\up) \cap \cL_{\subsetneq} (H,\up) 
\, \big)_{\mathbb{K}^\star(\delta)} \Big) \gs_H \up^{\Delta/\edges(H)} \rate_{n,p} \,. 
\]
We thus get the bound on the \abbr{rhs} of \eqref{goal} for the first term in \eqref{eq:union}
(and some small $c(H)>0$), as a consequence of the upper-\abbr{LDP} 
of Proposition \ref{prop:LD.Bstar}(a) at $K_1=c(H) \up^{\Delta/\edges(H)}$.
Indeed, \eqref{K0-cond} is satisfied for the assumed range of $p$ once we set
\[
C(H) \ge 2 K_0 (\Delta+1) \delta^{-2} \,.
\] 
For the remaining terms in \eqref{eq:union},
 taking 
$c(H) \le \min\{c(F) : F \subsetneq H\}$, we have by the 
induction hypothesis, for any $F \subsetneq H$ with $\Delta(F)\ge 2$, 
\begin{align}\label{induct}
\log \P(t(F, \Gnp/p)\ge \up) 
&\le - c(F)( \up^{1/\edges(F)} p )^{\Delta(F)} n^2 \log (1/p) \nonumber \\
&\le  - c(H)( \up^{1/\edges(H)} p )^{\Delta(H) } n^2 \log (1/p) \,,
\end{align}
since $\edges(F) <\edges(H)$ (we only need $\edges(F)\le \edges(H)$), $\up p^{\edges(H)} \le 1$ (see \eqref{UB.better}) and $\Delta(F)\le \Delta(H)$. 
For the case that $F \subsetneq H$ with $\Delta(F)=1$, we get from the case $\Delta=1$ of \eqref{goal} that 
\begin{align}	\label{goal2.D1}
\log \P(t(F, \Gnp/p)\ge \up) 
&\le  - c(F) \up^{1/\edges(F)} p  n^2 \log \up \,.
\end{align}
The \abbr{rhs} of \eqref{induct} increases in $\Delta(H) \ge 2$, thereby \eqref{goal2.D1}
yields \eqref{induct}  whenever 
\begin{equation}
\label{goal2}
\up^{1/\edges(F) - 2/\edges(H)} \gs p\log (1/p)\,.
\end{equation}
Note that \eqref{goal2} trivially holds if $\edges(F)\le \edges(H)/2$, while otherwise we
get from  \eqref{UB.better} upon recalling that $\edges(H) \ge \edges(F)+1$, that 
\[
\up^{1/\edges(F) - 2/\edges(H)} \ge
p^{-\edges(H)(1/\edges(F) - 2/\edges(H))} = p^{2-\edges(H)/\edges(F)} \ge p^{1-1/\edges(F)}
\gs p\log(1/p) \,.
\]
We have thus established \eqref{goal2} and thereby \eqref{induct} for any $F \subsetneq H$.
Plugging this back in \eqref{eq:union} completes our induction step.
\end{proof}

\end{appendix}

\begin{acks}
 The first author was supported in part by NSF grant DMS-2154029.
 The second author was supported in part by NSF grant DMS-1954337.
We thank the anonymous referees for the valuable suggestions which
improved the presentation of our results.
\end{acks}



\bibliographystyle{imsart-number} 
\bibliography{LDT}       


\end{document}